\documentclass[11pt,reqno]{amsart}
\usepackage{amsmath,amsfonts,amssymb,amsthm,ifpdf}
\usepackage{graphicx,color}
\ifpdf
\usepackage[pdftex,pdfstartview=FitH,pdfborderstyle={/S/B/W 1}, %
colorlinks=true, linkcolor=blue, urlcolor=blue,citecolor=blue,%
pagebackref=true %
]{hyperref}

\usepackage[utf8]{inputenc} 
\usepackage{wrapfig}
\usepackage[authoryear]{natbib}

\usepackage{mathtools}
\usepackage{breakurl}

\usepackage{graphicx}

\usepackage{graphicx}
\usepackage{comment}
\usepackage{color}
\usepackage{setspace}
\usepackage{enumerate}
\usepackage{enumitem}

\usepackage{ esint }

\usepackage[english]{babel}
 
\usepackage{thmtools}

\usepackage{macros_clt}

\usepackage[left=1.25in,bottom=1in,right=1.25in,top=1in,footskip=0.25in]{geometry}
 
\newlist{thmlist}{enumerate}{1}
\setlist[thmlist]{label=(\roman{thmlisti}), ref=\thetheorem(\roman{thmlisti})}

\declaretheorem[
    name=Theorem,
    Refname={Theorem,Theorems},
    ]{theorem}
 
\declaretheorem[
    name=Theorem,
    Refname={Theorem, Theorems},
   numbered=no
   ]{theorem*} 
 
\declaretheorem[
    name=Lemma,
    Refname={Lemma,Lemmas},
   sibling=theorem
   ]{lemma}

\declaretheorem[
    name=Proposition,
    Refname={Proposition,Propositions},
	sibling=theorem
	]{proposition}

\newtheorem{remark}[theorem]{Remark}

\newtheorem*{proposition*}{Proposition}

\usepackage{nameref,hyperref}
\usepackage[capitalize]{cleveref}

\Crefname{theorem}{Theorem}{Theorems}
\Crefname{lemma}{Lemma}{Lemmas}

\addtotheorempostheadhook[theorem]{\crefalias{thmlisti}{theorem}}
\addtotheorempostheadhook[lemma]{\crefalias{thmlisti}{lemma}}

\usepackage{amsmath}
\makeatletter
\let\save@mathaccent\mathaccent
\newcommand*\if@single[3]{%
  \setbox0\hbox{${\mathaccent"0362{#1}}^H$}%
  \setbox2\hbox{${\mathaccent"0362{\kern0pt#1}}^H$}%
  \ifdim\ht0=\ht2 #3\else #2\fi
  }
\newcommand*\rel@kern[1]{\kern#1\dimexpr\macc@kerna}
\newcommand*\widebar[1]{\@ifnextchar^{{\wide@bar{#1}{0}}}{\wide@bar{#1}{1}}}
\newcommand*\wide@bar[2]{\if@single{#1}{\wide@bar@{#1}{#2}{1}}{\wide@bar@{#1}{#2}{2}}}
\newcommand*\wide@bar@[3]{%
  \begingroup
  \def\mathaccent##1##2{%
    \let\mathaccent\save@mathaccent
    \if#32 \let\macc@nucleus\first@char \fi
    \setbox\z@\hbox{$\macc@style{\macc@nucleus}_{}$}%
    \setbox\tw@\hbox{$\macc@style{\macc@nucleus}{}_{}$}%
    \dimen@\wd\tw@
    \advance\dimen@-\wd\z@
    \divide\dimen@ 3
    \@tempdima\wd\tw@
    \advance\@tempdima-\scriptspace
    \divide\@tempdima 10
    \advance\dimen@-\@tempdima
    \ifdim\dimen@>\z@ \dimen@0pt\fi
    \rel@kern{0.6}\kern-\dimen@
    \if#31
      \overline{\rel@kern{-0.6}\kern\dimen@\macc@nucleus\rel@kern{0.4}\kern\dimen@}%
      \advance\dimen@0.4\dimexpr\macc@kerna
      \let\final@kern#2%
      \ifdim\dimen@<\z@ \let\final@kern1\fi
      \if\final@kern1 \kern-\dimen@\fi
    \else
      \overline{\rel@kern{-0.6}\kern\dimen@#1}%
    \fi
  }%
  \macc@depth\@ne
  \let\math@bgroup\@empty \let\math@egroup\macc@set@skewchar
  \mathsurround\z@ \frozen@everymath{\mathgroup\macc@group\relax}%
  \macc@set@skewchar\relax
  \let\mathaccentV\macc@nested@a
  \if#31
    \macc@nested@a\relax111{#1}%
  \else
    \def\gobble@till@marker##1\endmarker{}%
    \futurelet\first@char\gobble@till@marker#1\endmarker
    \ifcat\noexpand\first@char A\else
      \def\first@char{}%
    \fi
    \macc@nested@a\relax111{\first@char}%
  \fi
  \endgroup
}
\makeatother
\makeatletter
\renewcommand\paragraph{\@startsection{paragraph}{4}{\z@}%
                                    {1ex \@plus1ex \@minus.2ex}%
                                    {-1em}%
                                    {\normalfont\normalsize\bfseries}}
\makeatother
 \usepackage[foot]{amsaddr}

 \title[Central Limit Theorems for Smooth Optimal Transport Maps]{Central Limit Theorems for  \\ Smooth Optimal Transport Maps}
\author{Tudor Manole$^1$}
\author{Sivaraman Balakrishnan$^{2,3}$}
\author{Jonathan Niles-Weed$^{4,5}$}
\author{Larry Wasserman$^{2,3}$}
\address{$^1$\normalfont{Statistics and Data Science Center, Massachusetts Institute of Technology}}
\address{$^2$\normalfont{Department of Statistics and Data Science, Carnegie Mellon University}}
\address{$^3$\normalfont{Machine Learning Department, Carnegie Mellon University}}
\address{$^4$\normalfont{Center for Data Science, New York University}}
\address{$^5$\normalfont{Courant Institute of Mathematical Sciences, New York University}}
\email{\texttt{tmanole@mit.edu, siva@stat.cmu.edu, jnw@cims.nyu.edu, larry@stat.cmu.edu}}
\begin{document}

\begin{abstract} 
One of the central objects in the theory of optimal transport is the Brenier
map: the unique monotone transformation which pushes forward an absolutely continuous
probability law
onto any other given law. A line of recent work has 
analyzed  $L^2$ convergence rates of plugin estimators of  Brenier maps, 
which are defined as the Brenier map between  density estimators
of the underlying distributions. In this work, we show that such 
estimators satisfy a pointwise central limit theorem when the underlying laws 
are supported on the flat torus of dimension $d \geq 3$. 
 We also derive a negative result, showing
that these estimators do not converge weakly in $L^2$ when the
dimension is sufficiently large. Our proofs hinge upon
a quantitative linearization of the Monge-Amp\`ere equation, which may be of independent interest. 
This result allows us to reduce our
problem to that of deriving limit laws for the solution of a uniformly elliptic partial
differential equation with a stochastic right-hand side, subject to periodic boundary conditions.
\end{abstract} 
\maketitle

\allowdisplaybreaks
 \vspace{-0.2in}

%
%

\section{Introduction} 
\label{sec:introduction} 

One of the central objects in the   theory of optimal transport is the 
{\it Brenier map}: the unique monotone mapping which pushes forward an absolutely continuous probability law
onto any other given law.  
 Brenier maps play  a prominent  role in a wide range of 
 recent scientific applications; for instance, they have been used in high energy physics 
 to model the flow of energy between 
collider events~\citep{komiske2019a}, in computational 
biology to model the development of embryonic cells~\citep{schiebinger2019},  and 
in biophysics to quantify the colocalization of cellular samples~\citep{tameling2021}.
Brenier maps are also used as part of increasingly-many statistical methodologies, and we refer 
to~\cite{kolouri2017}, \cite{panaretos2019a}, and \cite{chewi2024} for surveys.

In these various applications, the underlying probability laws are typically unknown, 
and the practitioner must estimate the Brenier map on the basis of i.i.d.~observations. 
A common class of estimators are {\it plugin estimators}, which are defined as Brenier maps between estimators
of the underlying laws.
These plugin estimators are stochastic Brenier maps  that depend on the underlying
sample points in a highly nonlinear fashion, which makes their large-sample  behaviour difficult to determine.
Nevertheless, quantifying their limiting distribution is of great importance
for performing statistical inference. 
The aim of this manuscript is to inititate the study of limit laws for 
Brenier maps between  absolutely continuous probability measures in general dimension.

We will focus our attention throughout on probability measures supported on
the $d$-dimensional flat torus $\bbT^d = \bbR^d/\bbZ^d$.
Let $P$ and $Q$ denote two absolutely continuous probability measures
with respect to the uniform law on $\bbT^d$, with respective densities~$p$ and~$q$. 
Building upon the seminal work of~\cite{knott1984} and~\cite{brenier1991}, 
it was shown by~\cite{cordero-erausquin1999} that there exists a $P$-almost everywhere
uniquely-defined mapping $T_0:\bbT^d \to \bbT^d$ which is the gradient
of a convex function $\varphi_0:\bbR^d \to \bbR$, and which pushes forward $P$ onto $Q$, in the sense that:
\begin{equation}
\label{eq:pushforward}
Q(B) = P\big(T_0^{-1}(B)\big), \quad \text{for all Borel sets } B \subseteq \bbT^d.
\end{equation}
When equation~\eqref{eq:pushforward} holds, we write ${T_0}_{\#}P=Q$. 
The mapping $T_0=\nabla\varphi_0$ is called an optimal transport map, or Brenier map, 
while any convex function $\varphi_0$ whose gradient is equal to $T_0$ almost everywhere
is referred to as a Brenier potential. 
We will always assume in this manuscript that $p,q,1/p,$ and $1/q$ are bounded functions over $\bbT^d$.
As we recall below, these conditions imply that $T_0$ admits a unique representative which is continuous
over $\bbT^d$~\citep{caffarelli1992a}, and we always assume that $T_0$ is taken to be  this representative. Likewise, the potential $\varphi_0$ 
admits a representative which is continuously differentiable and uniquely defined up to an additive constant, 
and we will always assume that $\varphi_0$ is taken to be this representative, with the additive constant 
chosen such that $\varphi_0$ has mean zero over the unit hypercube. With this convention, $T_0$ and $\varphi_0$ are uniquely defined, 
and can both be evaluated pointwise without any~ambiguity.
 
We will be primarily interested in the one-sample problem, in which only the measure  $Q$ is sampled from. 
As discussed in Remark~\ref{rem:two_sample}, our techniques can immediately be 
extended to the case where $P$ is also sampled from, however 
we do not carry out this extension in order to keep our exposition concise. 
Let $Q_n = (1/n)\sum_{i=1}^n \delta_{Y_i}$ be an  empirical measure comprised of  
i.i.d. random variables  $Y_1, \dots, Y_n$  with common law $Q$.  
Our main object of interest is the optimal transport map  
which pushes forward $p$ onto the {\it kernel density estimator} of $q$, which is defined  as
$$\hat q_n = Q_n \star K_{h_n} = \int_{\bbR^d} K_{h_n}(\cdot - y) dQ_n(y),$$
where $K\in \calC^\infty(\bbR^d)$ is a smooth mollifier which integrates to unity,  $K_{h_n}(\cdot) = K(\cdot/h_n)/h_n^d$
for some nonnegative sequence $h_n \downarrow 0$, and where the integration in the above display is to be interpreted
by extending $Q_n$ to a measure on $\bbR^d$ via $\bbZ^d$-periodicity. The kernel $K$ will typically be permitted to take on negative values, 
thus $\hat q_n$ does not necessarily define a probability density. When it does, 
we define $\hat Q_n$ to be the probability law with density $\hat q_n$. 
Furthermore, over the probabilistic event that $\hat q_n$ and $1/\hat q_n$ are nonnegative bounded functions on $\bbT^d$, 
we define $\hat T_n$ to be the unique continuous  optimal
transport map pushing $P$ forward onto $\hat Q_n$, and $\hat\varphi_n$  to be  
the unique continuously differentiable Brenier potential, admitting mean zero over the unit hypercube,
such that $\hat T_n = \nabla\hat\varphi_n$. 
Over the complement of this event, 
$\hat T_n$ can be defined as any continuous vector field from $\bbT^d$ into itself, without changing any of our conclusions; for instance, 
one may take $\hat T_n = \mathrm{Id}$.   

When $K$ is taken to be a Gaussian kernel, the density estimator $\hat q_n$ is simply  the regularization
of the empirical measure $Q_n$ with respect to the heat kernel at time $h_n^2$. 
This form of regularization has received a great deal of recent interest in the optimal transport literature,
ever since it was used by~\cite{ambrosio2019b}, 
building upon a conjecture of~\cite{caracciolo2014}, to derive exact asymptotics for the quadratic  optimal matching problem. 
Unlike these works, however, our aim here is not to treat $\hat T_n$ as an approximation of the Brenier map $T_n$ pushing
$P$ forward onto $Q_n$. Indeed,  our results will typically require $h_n$ to vanish 
at too slow of a rate for such an approximation to be meaningful. 
We instead view~$\hat T_n$ as the object of interest,  motivated by the fact that it provides a better approximation
of $T_0$ than the empirical map $T_n$, when the underlying densities are smooth.

The estimator $\hat T_n$ has appeared in the past works of~\cite{gunsilius2022}, \cite{deb2021a}, and \cite{manole2021}, where the 
main goal was to derive its expected $L^2(\bbT^d)$ convergence rate. In particular, 
the latter of these references shows that if $P$ and $Q$ admit densities  $p,q$
which lie in the space $\calC_+^{s}(\bbT^d)$ of positive and $s$-H\"older smooth functions over~$\bbT^d$, for some $s > 0$, then
\begin{equation} \label{eq:map_rate_past}
\bbE \|\hat T_n - T_0\|_{L^2(\bbT^d)}^2 \lesssim\begin{cases}
1/n, & d = 1 \\
\log n / n, & d = 2 \\
n^{-\frac{2(s+1)}{2s+d}}, & d \geq 3,
\end{cases}
\end{equation}
assuming that $h_n \asymp h_n^* = n^{-1/(d+2s)}$. 
Furthermore, this convergence rate cannot be improved in full generality~\citep{hutter2021},
except potentially by a logarithmic factor when $d=2$. 
The above result is reminiscent of the corresponding convergence rate of the kernel density estimator itself: 
under the same conditions as above, and again taking $h_n\asymp h_n^*$, it holds that
\begin{equation} \label{eq:density_estimation_rate}
\bbE \|\hat q_n-q\|_{L^2(\bbT^d)}^2 \lesssim n^{-\frac{2s}{2s+d}},\quad \text{for any fixed }d \geq 1.
\end{equation}
The convergence rates~\eqref{eq:density_estimation_rate} and~\eqref{eq:map_rate_past}
share the same qualitative behaviour when the smoothness parameter $s$ grows to infinity, in which
case both rates formally
approach the parametric rate $1/n$, ignoring logarithmic factors. On the other hand, these two convergence rates differ significantly 
when $s$ is formally taken to vanish. In this case, the density estimation rate becomes slower than any polynomial rate, 
whereas the optimal transport map estimation rate approaches 
the familiar convergence rate of the empirical measure under the Wasserstein distance for positive and bounded densities
over the torus~\citep{ajtai1984,divol2021a,bobkov_akt}. More generally, the density estimation rate is always
strictly slower than the optimal transport map estimation rate, and, as discussed by~\cite{hutter2021}, 
this discrepancy can be anticipated from the fact that optimal transport maps are typically
of one degree of smoothness higher than the corresponding densities.  
Furthermore, it was shown by~\cite{manole2021} that
\begin{equation}\label{eq:stability_bound_past}
\|\hat T_n - T_0\|_{L^2(\bbT^d)}^2 \asymp \|\hat q_n - q\|_{H^{-1}(\bbT^d)}^2,
\end{equation}
which shows that equation~\eqref{eq:map_rate_past} can be interpreted 
as the convergence rate for the kernel density estimator under a negative Sobolev norm of first order. 

These results provide an essentially sharp characterization of the $L^2$ convergence rate of~$\hat T_n$, however they
do not offer any insight into its pointwise or uniform behaviour.  
The main result of this manuscript is to derive the pointwise limiting distribution
of $\hat T_n$, in the regime $d \geq 3$. 
\begin{theorem*}[Informal]
Let $s > 2$, $d \geq 3$, and assume $p,q \in \calC_+^s(\bbT^d)$. Then, 
there exists a sequence $h_n = o(h_n^*)$ such that for all $x \in \bbT^d$,  
there exists a positive  definite $d \times d$ matrix $\Sigma(x)$ such that, as $n\to\infty$, 
\begin{equation}
\label{eq:informal_stmt_clt}
\sqrt{nh_n^{d-2}}\big( \hat T_n(x) - T_0(x)\big) \overset{w}{\longrightarrow} N(0,\Sigma(x)).
\end{equation}
\end{theorem*}
The limiting covariance $\Sigma(x)$ can be made completely explicit, and will be described in 
Theorem~\ref{thm:clt_map} below. 
Furthermore, the result in fact holds for a wide range of sequences $h_n$ which
are of lower order than $h_n^*$. As we shall see, the latter condition is crucial for the centering constant 
in equation~\eqref{eq:informal_stmt_clt} to be the population quantity $T_0(x)$ itself.

To the best of our knowledge, this result is the first   central limit theorem 
for  optimal transport maps between   absolutely continuous distributions 
in dimension greater than one. The one-dimensional counterpart of this result was established 
by~\cite{ponnoprat2023},  in which case the appropriate scaling sequence is simply $\sqrt n$ rather than $\sqrt{nh_n^{d-2}}$, as could have 
heuristically been anticipated
from equation~\eqref{eq:map_rate_past}. 
Their work makes use of the representation of univariate optimal transport
maps in terms of quantile functions, which of course cannot be exploited in general dimension. 
The most challenging regime $d=2$ is left open by our work, except for the special
case where $P$ and $Q$ are both equal to the uniform distribution on the torus, which we discuss
in Section~\ref{sec:discussion}.
 
We note that limit laws for several related problems 
 have appeared
in recent literature. In the case where $P$ and $Q$ are discrete distributions, 
\cite{klatt2022} have derived central limit theorems for optimal transport
couplings, while~\cite{delbarrio2024} and~\cite{sadhu2023} considered the case where only one of $P$ or $Q$ is discrete. 
Furthermore,~\cite{harchaoui2020}, \cite{gunsilius2022_causal}, \cite{gonzalez-sanz2022}, \cite{goldfeld2022a},
and~\cite{gonzalez2023} 
derived limit laws for entropically regularized variants of the optimal transport map with a fixed regularization parameter. 
Let us emphasize that in each of these past works, the limit laws hold in a weak sense, and 
their scaling is of the parametric order $\sqrt n$, which is a reflection of the fact that the collection of 
optimal transport potentials arising
in these problems forms a Donsker class. This property fails to hold
in our setting, as can be anticipated from the fact  that our pointwise
central limit theorem
exhibits a  nonparametric scaling $\sqrt{nh_n^{d-2}}$. 
Furthermore, one cannot hope for a weak limit
law to hold under our assumptions: we will show in Theorem~\ref{thm:uniform_impossible} below
that there is no scaling of the process $\hat T_n - T_0$ which converges to a non-degenerate limit in $L^2(\bbT^d)$,
for a sensible range of values~$h_n$. Such a result is akin to the failure of weak limit
laws for the kernel density estimator itself~\citep{nishiyama2011,stupfler2014,stupfler2016}.

Let us provide a heuristic summary of our proof strategy. 
Our starting point is the celebrated interior regularity theory for optimal transport maps 
developed by \cite{caffarelli1992a},  which we recall in Section~\ref{sec:background} below.  
In their strongest form, these results imply that, 
under the smoothness condition $p,q \in \calC_+^s(\bbT^d)$, 
the Brenier potential $\varphi_0$ lies in $\calC_{\mathrm{loc}}^2(\bbR^d)$. 
In fact, we will see that with probability tending to one, 
the {\it estimator} $\hat\varphi_n$ itself lies in $\calC_{\mathrm{loc}}^2(\bbR^d)$, with H\"older
norm over any fixed compact set which is uniformly bounded in $n$. 
Over this high-probability event, the pushforward conditions ${{{\hat{T}}_n}}{\phantom{}}_\# P = \hat Q_n$ and ${T_0}_\# P =Q$
are equivalent to  the solvability of the 
following {\it Monge-Amp\`ere equations}
$$\det(\nabla^2\hat\varphi_n) = \frac{p}{\hat q_n(\nabla\hat\varphi_n)},\quad \text{and}\quad 
  \det(\nabla^2\varphi_0) = \frac{p}{q(\nabla\varphi_0)},
\quad \text{over } \bbR^d,$$
in the classical sense. It is well-known that these equations formally linearize  
to second-order uniformly elliptic partial differential equations~\citep{villani2003}. By leveraging this fact, 
we use a   Taylor expansion argument to show that 
\begin{equation}\label{eq:formal_expansion}
\hat T_n = T_0 + \nabla  u_n + R_n,
\end{equation}
where $R_n$ is a Taylor remainder which vanishes at a fast rate in $L^\infty(\bbT^d)$,
and where $ u_n$ is the unique mean-zero solution to the partial differential equation
\begin{equation}\label{eq:main_pde}
Lu_n=-\div(q \nabla  u_n(\nabla\varphi_0^*)) = \hat q_n - q,\quad \text{over } \bbT^d.
\end{equation}
Here, $\varphi_0^*$ denotes the Legendre-Fenchel transform of $\varphi_0$. 
Let us briefly comment on equation~\eqref{eq:main_pde}. In the special case where $p$
and $q$ are both equal to the uniform law on $\bbT^d$, the above is simply
the periodic Poisson equation, which has appeared in similar
linearization strategies in the works 
of~\cite{ambrosio2019}, \cite{ambrosio2019b},
\cite{ambrosio2019c},
\cite{goldman2022}, 
\cite{clozeau2023}, and references therein, in the context of deriving exact
asymptotics for various optimal matching problems. 
When  $p$ and $q$ are equal but not necessarily uniform, 
the differential operator above becomes the Witten Laplacian 
$-\div(q\nabla \cdot)$, as noted by~\cite{greengard2022}. 
For general densities $p$ and $q$, the operator $L$ 
is not
strictly-speaking of elliptic type, 
but with enough regularity assumptions on $\varphi_0$, 
we will show  that it is closely
connected to a self-adjoint uniformly elliptic operator.
See Remark~\ref{rem:operator} for further interpretation of the operator~$L$.

In view of the formal expansion~\eqref{eq:formal_expansion}, 
the problem of deriving a central limit theorem for~$\hat T_n$ reduces
to the problem of deriving a pointwise central limit theorem for the gradient
of the solution to the PDE~\eqref{eq:main_pde}. While $L^2$  rates for estimating coefficients of elliptic 
PDEs have been studied in the literature (cf.~\cite{nickl2020}, \cite{giordano2020}, and references therein), we are not aware of existing pointwise convergence rates or limit laws for solutions to elliptic PDEs with a stochastic right-hand side.
The bulk of our effort goes into this point. Letting $Q_{h_n}$ be the probability distribution with density 
$q_{h_n} = \bbE[\hat q_n(\cdot)]$, 
we begin by decomposing $\nabla u_n$ into the
terms
\begin{align}
\label{eq:nabla_un_decomp}
\nabla  u_n(x) = \nabla L^{-1}[\hat q_n - q_{h_n}](x) + \nabla L^{-1}[q_{h_n} - q](x),\quad \text{for any given } x \in \bbT^d.
\end{align}
The first (stochastic) term on the right-hand side of the above display
will   characterize the fluctuations of $\hat T_n(x)$ around its mean, 
while the second (deterministic) term will characterize the bias of $\hat T_n(x)$. 
We will show that the condition $h_n = o(h_n^*)$ 
is essentially sufficient for the deterministic term to be of negligible order; the heart of the matter
lies in the stochastic term. 
To derive its  asymptotic distribution, we appeal to a variant of the ``coefficient freezing'' method
which is commonly used to derive a priori estimates for elliptic PDE. Specifically, we will argue that, 
in a neighborhood of the point $T_0(x)$, 
the solution to equation~\eqref{eq:main_pde} is well-approximated by the solution $v_n$ to the equation
\begin{equation}
\label{eq:coeff_freezing_pde}
-q(T_0(x)) \div(\nabla v_n(G_x)) = \hat q_n-q_{h_n}, \quad \text{over } \bbT^d,
\end{equation}
where $G_x$ denotes the first-order Taylor approximation
of $\nabla\varphi_0^*$ at the point $T_0(x)$. This approximation is useful because, unlike $\nabla u_n$, 
the map $\nabla v_n$ can be expressed as the evaluation of a  {\it convolution operator}
 at the smoothed empirical process. Indeed, we will see that there exists a map 
$\Theta:\bbT^d  \to \bbR^d$
for which the following representation holds:
\begin{align}\label{eq:vn_is_convolution}
\nabla v_n(x) = \int_{\bbT^d} \Theta(y-T_0(x)) d(\hat Q_n-Q_{h_n})(y).
\end{align}
The map $\Theta$ is the gradient of the periodic Green's function associated to the differential operator appearing
on the left-hand side of equation~\eqref{eq:coeff_freezing_pde}, which can be derived in closed form when working over the torus. 
By associativity of convolutions, we can further write
$$\nabla v_n(x) = \big[(\Theta \star K_{h_n}) \star (Q_n-Q)\big](T_0(x)).$$
The map in the above display is manifestly a kernel density estimator with respect to the vector-valued ``kernel'' $\Theta \star K_{h_n}$.
Its limiting distribution can be derived using elementary means, and ultimately describes the limiting distribution of $\hat T_n(x)$.

Though our main interest is in pointwise limit laws, 
as a byproduct of these results, we will also derive  
nonasymptotic pointwise rates of convergence for the estimator $\hat T_n$. 
A great deal of recent work has analyzed $L^2$ rates of  estimation 
for optimal transport maps~\citep{hutter2021,  deb2021, manole2021, pooladian2021, ghosal2022, gunsilius2022,divol2022,pooladian2023}, 
and qualitative uniform convergence results have been studied in the works of~\cite{chernozhukov2017, panaretos2019a, hallin2021, delara2021, ghosal2022,segers2022},
but our work is perhaps the first to provide near-optimal rates for pointwise estimation
  (albeit under the strong assumption that the underlying domain is the flat torus).

\subsection{Outline}
The remainder of this manuscript is organized as follows. 
After summarizing notational conventions, and providing a brief summary of the regularity theory of optimal transport maps in the following two subsections, we
state our main results in Section~\ref{sec:main_results}. 
In Section~\ref{sec:linearization}, we 
formalize the functional Taylor expansion~\eqref{eq:formal_expansion}.
We analyze the stochastic and deterministic terms from equation~\eqref{eq:nabla_un_decomp} in Section~\ref{sec:bias_variance_bounds}.
In Section~\ref{sec:pf_main_results}, we combine these elements to prove our main results.
In Section~\ref{sec:discussion}, we discuss extensions of our results to the regime $d=2$.
In Appendix~\ref{app:function_spaces}, we summarize elementary definitions and properties
about the function spaces used throughout this manuscript, while in Appendix~\ref{app:pde}, we 
provide a self-contained summary of some
elementary properties of elliptic differential equations subject to periodic boundary conditions. 
In Appendix~\ref{app:kde}, we derive convergence rates for the kernel density
estimator under various H\"older and Sobolev norms.
In Appendices~\ref{app:additional_proofs_2}--\ref{app:additional_discussion}, we provide proofs of technical results
deferred from Sections~\ref{sec:main_results}--\ref{sec:discussion}, and finally,  
we state several technical lemmas used throughout the manuscript in Appendix~\ref{app:additional_technical_lemmas}.

\subsection{Notation}\label{sec:notation}
We write $a\vee b = \max\{a,b\}$ and $a\wedge b = \min\{a,b\}$, and
$a_+ = a \vee 0$ for all $a,b \in \bbR$. If $a \geq 0$, $\lfloor a \rfloor$ and $\lceil a \rceil $ denote
the respective floor and ceiling of $a$. 
Given vectors $x,y \in \bbR^d$, we write $\|x\|$ and $\langle x,y\rangle$ to denote their 
Euclidean norm and inner product. Similarly, for two matrices $A,B \in \bbR^d$, $\|A\|$ denotes the Frobenius
norm of $A$, and $\langle A, B\rangle=\tr(A^\top B)$ denotes the corresponding Frobenius inner product.
For $1 \leq r \leq \infty$, we denote by $L^r(\bbT^d)$ the standard Lebesgue spaces with respect to the 
uniform probability law
$\calL$ on $\bbT^d$. Given $s \in \bbR$, $\alpha > 0$, and $r > 1$, we denote by 
$\calC^\alpha(\bbT^d)$  the periodic H\"older spaces (which are commonly also
denoted $\calC^{\lfloor \alpha\rfloor,\alpha-\lfloor \alpha\rfloor}(\bbT^d)$), and by $H^{s,r}(\bbT^d)$ 
the periodic Sobolev spaces, both of which are defined explicitly in Appendix~\ref{app:function_spaces}
for completeness. We use the nonstandard notation of appending a subscript ``$0$'' to these spaces to
indicate their restriction to mean zero maps, and we emphasize that this notation does not indicate a boundary condition. 
Thus, we write
$$L_0^r(\bbT^d) = \left\{ u\in L^r(\bbT^d): \int_{\bbT^d} u d\calL = 0\right\},$$
and for any $s > 0$, we write 
$\calC_0^s(\bbT^d) = \calC^s(\bbT^d) \cap L^1_0(\bbT^d)$ and  $H_0^{s,r}(\bbT^d) = H^{s,r}(\bbT^d) \cap L_0^1(\bbT^d)$. 
Furthermore, we let $\calC_0^\infty(\bbT^d) = \bigcap_{j\geq 1} \calC_0^j(\bbT^d)$, and
\begin{align*}
\calC_+^s(\bbT^d) &= \Big\{ f \in \calC^s(\bbT^d): f > 0~\text{over } \bbT^d\Big\}.
\end{align*}
When $r=2$, we abbreviate the spaces $H^{s,2}(\bbT^d)$ and $H_0^{s,2}(\bbT^d)$ by 
$H^s(\bbT^d)$ and $H^s_0(\bbT^d)$ respectively. 
Given a map $T=(T_1,\dots,T_d):\bbT^d \to \bbT^d$, we will write by abuse of notation
$$\|T\|_{\calC^s(\bbT^d)} := \sum_{i=1}^d \big\| T_i \big\|_{\calC^s(\bbT^d)}, \quad 
\|T\|_{L^r(\bbT^d)} :=  \sum_{i=1}^d \big\| T_i \big\|_{L^r(\bbT^d)},$$
for any $s > 0$, $1 \leq r \leq \infty$. 
Given a map $f \in L^1(\bbR^d)$, we denote its Fourier transform by $\calF[f](\xi) = \int_{\bbR^d} f(x)e^{-2\pi i \langle x,\xi\rangle}dx$
for any $\xi \in \bbR^d$. If instead $f \in L^2(\bbT^d)$, we continue to denote by $\calF[f](\xi)$ the Fourier
coefficients of $f$, now restricted to  $\xi \in \bbZ^d$. We also write $\bbZ_*^d = \bbZ^d\setminus\{0\}$.
See Appendix~\ref{app:function_spaces} for further details.
Given a map $f$ lying in the Schwartz class on $\bbR^d$, denoted $\calS(\bbR^d)$, 
we denote the homogeneous Sobolev seminorms of order $s \in \bbR$ by
$$\|f\|_{\dot H^s(\bbR^d)} = \big\| \|\cdot\|^s \calF[f]\big\|_{L^2(\bbR^d)}.$$
Given a twice differentiable map $f:\bbR^d \to \bbR$, the gradient of $f$ is denoted $\nabla f$, its Hessian is denoted
$\nabla^2 f$, and its Laplacian is denoted $\Delta f = \sum_{i=1}^d \partial^2 f / \partial x_i^2$. The divergence of a 
differentiable vector field $F=(F_1,\dots,F_d):\bbR^d \to \bbR^d$ is denoted $\div(F)=\sum_{i=1}^d \partial F_i/\partial x_i$. 

Given two real numbers $a,b > 0$, we write $a \lesssim b$ if there exists a universal
constant $C > 0$---not depending on $a$ and $b$ but possibly depending
on quantities which are either clear from context or 
specified explicitly---such that $a \leq C b$. We write $a\asymp b$ if $a \lesssim b\lesssim a$. 
When we wish to emphasize the dependence of $C$ on a particular quantity $g$, we may write $\lesssim_g$
or $\asymp_g$. Furthermore, given a sequence of nonnegative real numbers $(\alpha_n)_{n\geq 1}$, we write
$n^{-a} \ll \alpha_n \ll n^{-b}$ if there exist constants $c > 0$
and $b < c' < a$ such that $\alpha_n=c\cdot n^{-c'}$.

Most constants in this manuscript will depend on quantities such as  $\|p\|_{\calC^s(\bbT^d)}$ and $ \|p^{-1}\|_{L^\infty(\bbT^d)}$.
We therefore introduce the following abbreviation: for any $s>0$, $k\geq 1$, $f_1, \dots, f_k \in \calC^s_+(\bbT^d)$, 
\begin{equation}
\label{eq:omega}
\omega_s(f_1, f_2, \dots, f_k)  := \sum_{j=1}^k \big(\|f_j\|_{\calC^s(\bbT^d)}  + \|f_j^{-1}\|_{L^\infty(\bbT^d)}\big).
\end{equation}
Given a convex function $\varphi:\bbR^d \to \bbR$, we denote by
$$\varphi^*(y) = \sup_{x \in \bbR^d} \big\{ \langle x,y\rangle - \varphi(x)\big\}, \quad \text{for all } y \in \bbR^d,$$
the Legendre-Fenchel transform of $\varphi$. 

\subsection{Background on the Periodic Optimal Transport Problem}\label{sec:background}
Recall that the flat torus $\bbT^d$ is the set of equivalence classes 
of the form $[x]=\{x+\xi:\xi\in \bbZ^d\}$, for all $x \in [0,1)^d$. 
We abuse notation by typically writing $x$ in place of $[x]$. 
$\bbT^d$ is endowed with the standard metric
$$\|x-y\|_{\bbT^d} = \min\{\norm{x-y+\xi}: \xi \in \bbZ^d\},
\quad x,y \in \bbT^d,$$
where we emphasize, of course, that $\|\cdot\|_{\bbT^d}$ does not define a norm, and
that the notation $\|x-y\|_{\bbT^d}$ is not meant to be taken at face value. We abbreviate $\|x-0\|_{\bbT^d}$
by $\|x\|_{\bbT^d}$.
Let $\calP(\bbT^d)$ be the set of Borel probability measures on $\bbT^d$, which we identify
with the set of $\bbZ^d$-periodic Borel measures $P$ on $\bbR^d$ such that $P([0,1)^d) = 1$,
and let $\calPac(\bbT^d)$ be the subset of measures in $\calP(\bbT^d)$ which are absolutely continuous with respect
to the Lebesgue measure on $\bbR^d$. 
 A function $f:\bbT^d \to \bbR$ is understood to 
be a function on $\bbR^d$ which is $\bbZ^d$-periodic, and we write $T:\bbT^d \to \bbT^d$
when $T$ is a map from $\bbR^d$ to $\bbR^d$ such that $[T(x)] = [T(y)]$ whenever $[x]=[y]$.
Any such map is uniquely determined by the value it takes on the unit hypercube (or any of its $\bbZ^d$-translates), and
it will therefore be convenient to fix the notation
\begin{equation}
\calQ = [0,1]^d.
\end{equation}
The following result can be deduced from~\cite{cordero-erausquin1999}, together with the interior regularity theory
for optimal transport maps developed by~\cite{caffarelli1992a}. 
We also refer, for instance, to~\citet[Section 2]{ambrosio2012}, for further details on how these results are adapted to the periodic setting. 
\begin{theorem}\label{thm:caffarelli}
Let $P,Q \in \calPac(\bbT^d)$ admit respective densities $p,q$, and assume there exists $\gamma > 0$
such that $\gamma^{-1} \leq p,q \leq \gamma$ over $\bbT^d$. Then, there exists 
a unique convex function $\varphi_0: \bbR^d \to \bbR$ such that $\int_{\calQ} \varphi_0d\calL=0$,
 and such that the following properties hold.
\begin{thmlist}
\item \label{thm:caffarelli_weak_reg} There exist constants $C,\epsilon > 0$ depending only on $d,\gamma$ such that
$$\|\varphi_0\|_{\calC^{1+\epsilon}(\calQ)} \leq C.$$
\item  The map $\varphi_0 - \|\cdot\|^2/2$ is $\bbZ^d$-periodic, and it holds that
$$T_0(x+\xi) = T_0(x) + \xi,\quad \text{for all } x\in \bbR^d, ~ \xi \in\bbZ^d.$$ 
In particular, $T_0$ defines a mapping from $\bbT^d$ into itself.
\item  $T_0 = \nabla\varphi_0$ is the unique continuous optimal transport map
pushing $P$ forward onto~$Q$.  Furthermore,  $T_0$ is invertible, and its inverse $T_0^{-1} = \nabla\varphi_0^*$
is the unique continuous optimal transport map pushing forward $Q$ onto $P$. \\[-0.1in]
\end{thmlist}
If, in addition, we assume that $p,q \in \calC^\epsilon(\bbT^d)$ for some $\epsilon > 0$, then
the following hold.
\begin{enumerate}
\item[(iv)] There exists  $\lambda > 0$ depending only on $\omega_\epsilon(p,q),d,\epsilon$ such that 
$$\|\varphi_0\|_{\calC^{2+\epsilon}(\calQ)} \leq \lambda,\quad\text{and,}\quad \nabla^2\varphi_0 \succeq  I_d/\lambda~~\text{over } \bbR^d.$$
\item[(v)] The map $\varphi_0$ is a classical solution to the Monge-Amp\`ere equation
$$\det(\nabla^2\varphi_0) = \frac{p}{q(\nabla\varphi_0) },\quad \text{over } \bbR^d.$$
\end{enumerate} 
\end{theorem}
In particular, Theorem~\ref{thm:caffarelli_weak_reg}
implies that the estimator $\hat T_n$ is well-defined, and can be evaluated pointwise. 

\section{Main Results}
\label{sec:main_results}
In this section, we state our main results regarding the asymptotic behaviour of the Brenier map $\hat T_n$.
We begin by stating our key technical result, namely a quantitative linearization estimate for the Monge-Amp\`ere
equation. 
For the remainder of the manuscript, we fix once and for all a real number $s > 2$
such that $s\not\in \bbN$,
which will typically represent the H\"older smoothness exponent of the densities $p$ and $q$.

\subsection{Linearization of the Monge-Amp\`ere Equation} 
Define a constant $\beta$ such that
\begin{equation}\label{eq:beta}
0 < \beta < \min\{1, s-2\}.
\end{equation}
Given $p,q \in \calC^{2+\beta}_+(\bbT^d)$, it follows from Caffarelli's regularity theory
(Theorem~\ref{thm:caffarelli}) that
$\varphi_0,\varphi_0^* \in \calC^{4+\beta}(\bbT^d)$. In particular, we may then define the  operator
$$Lu = -\div(q\nabla u(\nabla\varphi_0^*)),$$
for all $u \in \calC^2(\bbT^d)$. 
As we shall see in Section~\ref{sec:linearization}, it can be deduced from standard 
elliptic regularity theory that $L$
 is a bijection of $\calC^{2+\beta}_0(\bbT^d)$ onto $\calC_0^\beta (\bbT^d)$, with inverse denoted $L^{-1}$.

Let $\hat Q \in \calP_{\mathrm{ac}}(\bbT^d)$ be any
distribution with density $\hat q$ over $\bbT^d$, and assume that $\hat q\in \calC_+^{2+\beta}(\bbT^d)$. 
Let $\hat\varphi$ be the unique  Brenier
potential lying in $\calC^{2}(\bbR^d)$ whose gradient pushes forward $P$ onto $\hat Q$, and which
satisfies $\int_{\calQ} \hat\varphi d\calL = 0$.  
The following result shows that, in a very strong sense, the deviations
$\hat\varphi-\varphi_0$ are well-approximated by the solution to a linear partial differential equation, 
 whenever $\hat q$ is in a H\"older
neighborhood of $q$. 
\begin{theorem}
\label{thm:linearization}
Let $d \geq 1$, and assume $p,q,\hat q \in \calC_+^{2+\beta}(\bbT^d)$. Then, there exists 
a constant $C = C(\omega_{2+\beta}(p,q,\hat q),d,\beta) > 0$ such that
\begin{align}\label{eq:linearization}
\big\| (\hat\varphi - \varphi_0) - L^{-1}[\hat q - q] \big\|_{\calC^{2+\beta}(\bbT^d)} \leq C \|\hat q - q \|_{\calC^{1+\beta}(\bbT^d)}\|\hat q - q \|_{\calC^{\beta}(\bbT^d)}.
\end{align}
\end{theorem} 
Theorem~\ref{thm:linearization} is proved in Section~\ref{sec:linearization}.
In particular, this result implies that the approximation
\begin{equation}
\label{eq:heuristic_approximations}
  \nabla\hat\varphi - \nabla\varphi_0 \approx \nabla L^{-1}[\hat q-q]
  \end{equation}
holds up to an error which decays in uniform norm on the order
$\|\hat q - q \|_{\calC^{1+\beta}(\bbT^d)}\|\hat q - q \|_{\calC^{\beta}(\bbT^d)}$.
We believe it should be possible to replace the latter quantity  by $\|\hat q - q\|_{\calC^\beta(\bbT^d)}^2$, 
but our current statement is sufficiently sharp for the applications which
we have in mind. 
In Appendix~\ref{app:kde}, we state a convergence rate for the kernel
density estimator under H\"older norms, which will allow us to bound this approximation error
 when $\hat q$ is taken
to be the random density $\hat q_n$.

Our proof of Theorem~\ref{thm:linearization} relies on the following   H\"older
stability bound for Brenier potentials, which may be of independent interest:
under the same conditions as Theorem~\ref{thm:linearization}, we show in Proposition~\ref{prop:potential_stability}
below that
\begin{equation}
\|\hat \varphi - \varphi_0\|_{\calC^{2+\beta}(\bbT^d)} \lesssim \|\hat q - q \|_{\calC^\beta(\bbT^d)}.
\end{equation}
Equation~\eqref{eq:thm_potential_stability_C2} is closely related to a qualitative Sobolev stability result 
for the Monge-Amp\`ere equation proven by~\cite{dephilippis2013}. We emphasize that their assumptions
are  weaker than ours, and in particular do not imply that the Brenier potentials are twice differentiable, 
which makes their analysis significantly more involved.
Building upon their result, \cite{gunsilius2022} derived
a qualitative stability result for Brenier potentials
under the $\calC^{2+\beta}$ norm, but we 
are not aware of any quantiative bounds akin to the one above.

\begin{remark}[Interpretation of the Operator $L$]
\label{rem:operator}
We have not seen the general operator~$L$ appear in previous literature, 
thus we provide a heuristic interpretation here.
We will use some notions from the differential characterization of the optimal transport problem; see~\cite{ambrosio2008} for formal definitions. 
With the same notation as Theorem~\ref{thm:linearization}, let $u = \hat\varphi - \varphi_0$, 
and define 
$$S_t = \mathrm{Id} + t \nabla u \circ \nabla\varphi_0^*,\quad 0 \leq t \leq 1.$$
$S_t$ is a  generalized geodesic joining the densities $q$ and $\hat q$; in particular, 
the curve of densities $q_t = {S_t}_\# q$ satisfies $q_0=q$ and $q_1=\hat q$. 
Given a random variable $Y \sim q$, the induced curve of random variables $Y_t = S_t(Y)$
solves the following velocity equation
$$\partial_t Y_t = v_t(Y_t),\quad \text{with } v_t =  \nabla u \circ \nabla\varphi_0^* \circ S_t^{-1}.$$
These facts imply that the sequence $(q_t,v_t)_{0\leq t\leq 1}$ solves the continuity equation, 
$$\partial_t q_t = - \div(q_t v_t),\quad 0 \leq t \leq 1, $$
and hence, upon integrating with respect to time, 
$$\hat q - q = - \int_0^1 \div(q_tv_t)dt.$$
Now, as $\hat q$ approaches $q$, it is natural to expect that $S_t$ is well-approximated by
the identity map, so that, in turn, the vector field $q_t v_t$ is well-approximated by the time-independent vector
field $q\cdot \nabla u\circ \nabla\varphi_0^*$. Under this approximation, one heuristically  obtains
$$\hat q - q \approx -\div(q\nabla u \circ \nabla\varphi_0^*) = L[\hat\varphi-\varphi_0].$$
Theorem~\ref{thm:linearization} rigorously quantifies the approximation error in the above display,
but using a distinct proof technique.
\end{remark}

\subsection{Pointwise Convergence Rate}
Before presenting a central limit theorem for the estimator $\hat T_n$, 
it will be fruitful to state a bound on its pointwise
convergence rate. 
We begin with some notation. Let $h_n^* = n^{-1/  {(d+2s)}}$. For all  $x \in \bbT^d$
and $h_n > 0$,
set
$$T_{h_n}(x) = \bbE[\hat T_n(x)],\quad \text{and}\quad q_{h_n}(x) = \bbE[\hat q_n(x)].$$
We will require the following
assumption on the kernel $K$, for some $\alpha > 0$. 
\begin{enumerate}[leftmargin=4cm,listparindent=-\leftmargin,label=\textbf{Condition K($\alpha$).}]   
\item  \label{assm:kernel}
$K \in \calC_c^\infty((0,1)^d) $ is an even kernel which satisfies
\begin{equation*}
\sup_{\xi \in \bbR^d\setminus\{0\}} \frac{|\calF[ K](\xi) - 1|}{\|\xi\|^{\alpha}} < \infty.
\end{equation*}
\end{enumerate} 
Condition~\Kernel{$\alpha$} implies  that  the Fourier transform of $K$ 
is real,  that $K$ integrates to unity,
and that its moments up to order $\lfloor \alpha-1\rfloor$ vanish~\citep{gine2016}. These properties
are essential for obtaining the sharp convergence rate of the kernel density estimator when 
the smoothness  of the underlying densities is large.  

\begin{proposition}
\label{prop:pointwise_rate_map}
Let $d \geq 3$, and $p,q\in \calC_+^s(\bbT^d)$. Assume that the kernel $K$ satisfies condition~\Kernel{$s+1$}, and that for some $c > 0$, 
\begin{align}
\label{assm:bandwidth_rate_map}
n^{-\frac 1 {d+s+2}} \ll h_n \ll n^{-\frac 1 {d+4(s-1)}}.
\end{align}
Then, for any $\epsilon >0$,
there exists a constant $C = C(\omega_s(p,q),K,s, \epsilon,d) > 0$ such that 
\begin{equation}
\label{eq:pointwise_rate_map}
\big\| T_{h_n} - T_0\big\|_{L^\infty(\bbT^d)}  \leq C h_n^{s+1-\epsilon},\quad\text{and}\quad
\sup_{x \in \bbT^d} \bbE\big\|\hat T_n(x) - T_{h_n}(x)\big\|
\leq \frac{C }{\sqrt{nh_n^{d-2}}}.
\end{equation} 
\end{proposition}
The proof of Proposition~\ref{prop:pointwise_rate_map} appears in Section~\ref{sec:pf_main_results}. 
Condition~\eqref{assm:bandwidth_rate_map} on the bandwidth $h_n$ is never vacuous
due to the assumption that $s > 2$. In particular, the bandwidth $h_n\asymp h_n^*$ satisfies this condition, 
and for this choice, the two
quantities in equation~\eqref{eq:pointwise_rate_map} are essentially on the same order. In this case,
 Proposition~\ref{prop:pointwise_rate_map} implies
\begin{align}
\label{eq:pointwise_rate_hnstar}
\sup_{x\in \bbT^d}\bbE \big\| \hat T_n(x) - T_0(x)\big\| \lesssim n^{-\frac{s+1-\epsilon}{d+2s}}.
\end{align}
This result shows that the $L^2(\bbT^d)$ convergence
rate of $\hat T_n$, stated in equation~\eqref{eq:map_rate_past}, in fact holds in a
pointwise sense, at the price of the exponent $\epsilon$, which can be made arbitrarily small. 
Furthermore, by a simple modification of Theorem~6 of~\cite{hutter2021}, it can be 
seen that, in an information-theoretic sense, no
other estimator can achieve  a rate faster
than  $n^{-(s+1)/(2s+d)}$  for estimating $T_0(x)$
at a point $x \in \bbT^d$, uniformly over all measures $P,Q$
 satisfying the conditions of Proposition~\ref{prop:pointwise_rate_map}.
In this sense, $\hat T_n$ is  a  minimax optimal estimator, again, up to the arbitrarily small exponent~$\epsilon$. 
Though we conjecture that this exponent is superfluous, 
it cannot be easily removed with our current proof technique.
Its presence is related to the fact that 
the Poisson equation $\Delta u = f$ on $\bbT^d$  is not necessarily solvable in the classical sense for $f \in L_0^\infty(\bbT^d)$, 
but admits a solution $u \in \calC^{2+\epsilon}(\bbT^d)$ whenever $f \in \calC_0^\epsilon(\bbT^d)$.

Proposition~\ref{prop:pointwise_rate_map} is comparable to  
pointwise bounds for density estimation. It is well-known that, under weaker conditions than those 
of Proposition~\ref{prop:pointwise_rate_map}, 
it holds for any $d \geq 1$ that
$$\|q_{h_n} - q\|_{L^\infty(\bbT^d)} \lesssim h_n^s, \quad \text{and}\quad \sup_{x\in\bbT^d}\bbE |\hat q_n(x) - q_{h_n} (x) |\lesssim \frac 1 {\sqrt{nh_n^d}}$$
(cf.~\cite{gine2016}).
Both of these bounds are slower than those of Proposition~\ref{prop:pointwise_rate_map} by a factor of $h_n^{-1}$,
which is a reflection of the fact that optimal transport maps typically enjoy one degree of smoothness more than
the corresponding densities. This can be inferred heuristically from equation~\eqref{eq:heuristic_approximations}.
We nevertheless note that the optimal order of the bandwidth~$h_n$ which minimizes the sum of the above two terms
is $h_n^*$, as in the case of Proposition~\ref{prop:pointwise_rate_map}.

Proposition~\ref{prop:pointwise_rate_map} suggests that if the sequence of random variables
$$\sqrt{nh_n^{d-2}}(\hat T_n(x)-T_0(x))$$ 
were to admit a non-degenerate
limiting distribution, then $h_n$ would have be taken
of lower order than the optimal bandwidth $h_n^*$, a condition often referred to as {\it undersmoothing}. 
We derive limit laws under this condition, next.

\subsection{Pointwise Central Limit Theorem}

Our main result is the following. 
\begin{theorem}
\label{thm:clt_map}
Let $d \geq 3$, and $p,q \in \calC_+^s(\bbT^d)$. 
Assume that condition~\Kernel{$s+1$} holds, and 
\begin{equation}
\label{assm:bandwidth_map_clt}
 n^{-\frac 1 {d+4}} \ll h_n \ll n^{-\frac 1 {d+ 2s}}.
\end{equation}
Then, for all $x \in \bbT^d$, 
$$ \sqrt{nh_n^{d-2}}\big(\hat T_n(x) - T_0(x)\big) \overset{w}{\longrightarrow} N(0, \Sigma(x)),\quad \text{as} ~~ n \to \infty,$$
where, for all $x \in \bbT^d$,  $\Sigma(x)$ is the positive definite
 matrix with finite entries given by
\begin{equation}\label{eq:Sigma}
\Sigma(x) = \frac{1}{p(x)} \int_{\bbR^d} \xi\xi^\top  \left(\frac{ \calF[K](\calM(x)\xi)}{2\pi \langle \calM(x)\xi,\xi\rangle}\right)^2     d\xi,\quad \text{with } \calM(x) = \nabla^2\varphi_0^*(\nabla\varphi_0(x)).
\end{equation}
\end{theorem}  
Theorem~\ref{thm:clt_map} shows that the estimator $\hat T_n(x)$ obeys a central limit theorem
centered at its population counterpart $T_0(x)$, when the bandwidth $h_n$ lies in the range~\eqref{assm:bandwidth_map_clt}. 
We emphasize again that this range is never empty under the assumption $s > 2$. The upper bound of 
equation~\eqref{assm:bandwidth_map_clt} implies the undersmoothing condition $h_n = o(h_n^*)$, while the lower
bound is needed  in order for the error of our linearization of $\hat T_n$ to be of sufficiently low order.

To gain some intuition for the limiting covariance matrix $\Sigma(x)$, 
it is once again fruitful to compare our result to 
 density estimation. Under the   conditions of~Theorem~\ref{thm:clt_map}, 
it is a simple observation that for all $y \in \bbT^d$, 
$$\sqrt{nh_n^d} \big(\hat q_n(y) - q(y)\big) \overset{w}{\longrightarrow} 
N\left(0, q(y) \|K\|_{L^2(\bbR^d)}^2\right).$$
Thus, the limiting variance of $\hat q_n(y)$ is on the same scale as 
the~$L^2(\bbR^d)$ norm of the kernel~$K$. In contrast, 
the limiting covariance of the estimator $\hat T_n(x)$ satisfies
$$\text{tr}(\Sigma(x)) \asymp \frac 1 {p(x)} \int_{\bbR^d} \left(\frac{\calF[K](\xi)}{\|\xi\|}\right)^2d\xi
= \frac {\|K\|_{\dot H^{-1}(\bbR^d)}^2} {p(x)},$$
and is thus on the same scale as the 
first-order negative Sobolev seminorm of the kernel~$K$. 
This should not be surprising in view of the equivalence~\eqref{eq:stability_bound_past}.
Furthermore, the following can be said about the off-diagonal entries
of $\Sigma(x)$. 
\begin{lemma}
\label{cor:radial}
Assume the same conditions as Theorem~\ref{thm:clt_map}. Assume further
that $K$ is radial, and  that $P=Q$. 
 Then, the covariance matrix $\Sigma(x)$ defined in Theorem~\ref{thm:clt_map}
 is diagonal.  
\end{lemma}
Lemma~\ref{cor:radial} shows that, for radial kernels, the estimator $\hat T_n(x)$ has 
asymptotically independent entries when $P = Q$. This property typically fails to hold
when $P \neq Q$, however.
\subsection{Failure of Weak Convergence}
\label{sec:uniform_asymptotics} 
 
We have shown that the sequence $\hat T_n-T_0$ enjoys a pointwise central limit theorem 
under suitable conditions. We next state a negative result, showing that, under identical conditions,
the process $\hat T_n-T_0$
does not converge weakly in $L^2(\bbT^d)$ to a non-degenerate limit, when $d \geq 3$. 
\begin{theorem}
\label{thm:uniform_impossible}
Assume the same conditions as Theorem~\ref{thm:clt_map}.
Let $(\alpha_n)_{n\geq 1}$ be a positive sequence, and define the process
$$ \bbG_n = \alpha_n(\hat\varphi_n - \varphi_0),\quad n=1,2,\dots$$
viewed as a random element in $H_0^1(\bbT^d)$.  
Then, the following hold. 
\begin{enumerate}
\item[(i)]  If $\alpha_n = o(\sqrt{nh_n^{d-2}})$, then $\bbG_n$ converges weakly to 0 in $H_0^1(\bbT^d)$. 
\item[(ii)] If $\alpha_n \gtrsim \sqrt{nh_n^{d-2}}$, then $\bbG_n$ does not converge weakly in $H_0^1(\bbT^d)$. 
\end{enumerate}
\end{theorem}
We prove Theorem~\ref{thm:uniform_impossible} in Section~\ref{sec:pf_uniform_asymptotics},
by showing that the $H_0^1(\bbT^d)$ projection of $\bbG_n$  along any fixed direction  
typically vanishes at a  significantly faster rate than the convergence rate of $\bbG_n$ 
under the $H_0^1(\bbT^d)$ norm. To establish this result, we require the same conditions on the bandwidth
$h_n$ as in Theorem~\ref{thm:clt_map}. We do not rule out the possibility that $\bbG_n$ could converge
weakly when $h_n$ falls outside of this range, and in particular when $h_n = o(n^{-1/(d+4)})$.

\section{Quantitative Linearization of the Monge-Amp\`ere Equation}
\label{sec:linearization}
The aim of this section is to prove Theorem~\ref{thm:linearization} and Proposition~\ref{prop:potential_stability}. 
We begin by stating several important properties of the operator $L$. 
\subsection{The Differential Operator $L$}
Let $p,q \in \calC_+^{2+\beta}(\bbT^d)$, where the constant $\beta$ was fixed in equation~\eqref{eq:beta}, and define the operator
$$L:H_0^2(\bbT^d) \to L_0^2(\bbT^d), \quad Lu= -\div(q\nabla u(\nabla\varphi_0^*)).$$
A simple calculation reveals that the $L^2(\bbT^d)$ adjoint of $L$ is given by
$$L^*:H_0^2(\bbT^d) \to L_0^2(\bbT^d),\quad L^* v = -\div(p \nabla v(\nabla\varphi_0)),$$
which implies, in particular, that $L$ is self-adjoint if and only if $P=Q$. 
It will be convenient to note that $L$ is 
closely-related to a distinct operator $E$, which in turn is self-adjoint even when $P \neq Q$. 
This operator is defined by
$$E:H_0^2(\bbT^d) \to L_0^2(\bbT^d), \quad Eu = -\div(A \nabla u),$$
where we fix once and for all the matrix-valued map
$$A(x) = p(x) \nabla^2\varphi_0^*(\nabla\varphi_0(x)),\quad x \in \bbT^d.$$
The relation between $L$ and $E$ is described next.
\begin{lemma}
\label{lem:relation_E_L}
Let $p,q \in \calC_+^{2+\beta}(\bbT^d)$. 
Then, for any $u \in \calC_0^2(\bbT^d)$, it holds that
\begin{equation}
\label{eq:relation_E_L_1}
Eu = L[u](\nabla\varphi_0) \det(\nabla^2\varphi_0),\quad\text{and}\quad 
 Lu = E[u](\nabla\varphi_0^*) \det(\nabla^2\varphi_0^*).
 \end{equation}
Furthermore, 
\begin{equation}
\label{eq:relation_E_L_2}
E u
 =  - \det(\nabla^2\varphi_0) \Big\{ q(\nabla\varphi_0) \langle ( \nabla^2\varphi_0)^{-1}, \nabla^2 u\rangle + \langle \nabla q(\nabla\varphi_0), \nabla u\rangle\Big\}.
 \end{equation} 

\end{lemma}
\begin{proof}[Proof of  Lemma~\ref{lem:relation_E_L}]
For any  $u \in H_0^2(\bbT^d)$ and any test function $v \in \calC_0^\infty(\bbT^d)$, 
\begin{align*}
\langle Eu,v\rangle_{L^2(\bbT^d)} 
 = \int_{\bbT^d} \nabla v^\top \nabla^2\varphi_0^*(\nabla\varphi_0) \nabla u dP 
 &= \int_{\bbT^d} \nabla v (\nabla\varphi_0^*)^\top \nabla^2\varphi_0^* \nabla u(\nabla\varphi_0^*)dQ,
\end{align*}
where the final inequality follows by a change of variable. Deduce that 
 \begin{align*}
\langle Eu,v\rangle_{L^2(\bbT^d)}
= \int_{\bbT^d} \nabla [v (\nabla\varphi_0^*)]^\top \nabla u(\nabla\varphi_0^*)dQ   
 &= \int_{\bbT^d} v (\nabla\varphi_0^*) Lu,
 \end{align*}
 where we again integrated by parts. The above is equivalent to 
 $$\langle Eu,v\rangle_{L^2(\bbT^d)} = \langle v, L[u](\nabla\varphi_0)\det(\nabla^2\varphi_0)\rangle_{L^2(\bbT^d)},$$
 which implies the first claim of equation~\eqref{eq:relation_E_L_1}. The second claim follows analogously. 
To deduce equation~\eqref{eq:relation_E_L_2}, note that we may expand $L$ as 
 $$Lu
  = -q \langle \nabla^2\varphi_0^*, \nabla^2 u(\nabla\varphi_0^*)\rangle - \langle \nabla q, \nabla u(\nabla\varphi_0^*)\rangle,$$
  thus
   $$L[u](\nabla\varphi_0) = -q(\nabla\varphi_0) \langle \nabla^2\varphi_0^*(\nabla\varphi_0), \nabla^2 u\rangle - \langle \nabla q(\nabla\varphi_0), \nabla u\rangle.$$
Since $\nabla\varphi_0^*$ is the inverse
of $\nabla\varphi_0$, it holds that $(\nabla^2\varphi_0)^{-1} = \nabla^2\varphi_0^*(\nabla\varphi_0)$,
and claim~\eqref{eq:relation_E_L_2} now follows from claim~\eqref{eq:relation_E_L_1}.
%
\end{proof} 
With Lemma~\ref{lem:relation_E_L} in place, the properties of the operator $L$
can be deduced from the standard theory of elliptic PDEs
subject to periodic boundary conditions, which we summarize
in Appendix~\ref{app:pde}.
Indeed, under the smoothness assumption $p,q \in \calC_+^{2+\beta}(\bbT^d)$, 
it follows from Theorem~\ref{thm:caffarelli} that the eigenvalues of $\nabla^2\varphi_0^*$, and hence of the matrix $A$,
are uniformly bounded
from below over $\bbT^d$ by positive constants depending on $\omega_{2+\beta}(p,q)$.
The operator $E$ is therefore uniformly elliptic. Furthermore, 
the entries of $A$ lie in $\calC^{1+\beta}(\bbT^d)$, 
thus it follows from Lemma~\ref{lem:isomorphism} in Appendix~\ref{app:pde} that the mapping 
$E: H_0^2(\bbT^d) \to L_0^2(\bbT^d)$ is a bijection, 
whose restriction to $\calC_0^{2+\beta}(\bbT^d)$ is a bijection onto 
$\calC_0^\beta(\bbT^d)$. 
Additionally, the following norm equivalences hold:
\begin{align}
\label{eq:norm_equivalence_E}
\|Eu\|_{L^2(\bbT^d)} \asymp \|u\|_{H^2(\bbT^d)}, \quad \|Eu\|_{\calC^{2+\beta}(\bbT^d)} \asymp \|u\|_{\calC^\beta(\bbT^d)},
\end{align}
where the implicit constants depend only on $\omega_{2+\beta}(p,q), d, \beta$. From here, we may deduce the following. 
\begin{lemma}
\label{lem:inverse_L_to_E}
Let $p,q \in \calC_+^{2+\beta}(\bbT^d)$. Then, the mapping $L:H_0^2(\bbT^d) \to L_0^2(\bbT^d)$
is a bijection, whose restriction to $\calC_0^{2+\beta}(\bbT^d)$ is a bijection onto
$\calC_0^\beta(\bbT^d)$, and satisfies the norm equivalences
$$\|Lu\|_{L^2(\bbT^d)} \asymp \|u\|_{H^2(\bbT^d)},\quad 
\|Lu\|_{\calC^\beta(\bbT^d)} \asymp \|u\|_{\calC^{2+\beta}(\bbT^d)}.$$
Furthermore, for all $f \in \calC_0^\beta(\bbT^d)$, it holds
$$E^{-1}f = L^{-1}[f(\nabla\varphi_0^*)\det(\nabla^2\varphi_0^*)],\quad\text{and}\quad L^{-1}f = E^{-1}[f(\nabla\varphi_0)\det(\nabla^2\varphi_0)].$$
\end{lemma}
By reasoning as in   Appendix~\ref{app:pde}, 
it can be seen that for any $f \in L_0^2(\bbT^d)$, 
  the unique  map $u \in H_0^2(\bbT^d)$
which solves the equation $Lu=f$ over $\bbT^d$ satisfies the identity
\begin{equation}\label{eq:weak_sol_L}
\langle u,v\rangle_A = \langle f, v(\nabla\varphi_0^*)\rangle_{L^2(\bbT^d)}, \quad \text{for all } v \in H_0^1(\bbT^d),
\end{equation}
where we define the bilinear form
$$\langle u,v\rangle_A := \int_{\bbT^d} \langle A\nabla u,\nabla v\rangle d\calL
 = \int_{\bbT^d} \langle \nabla\varphi_0^*(\nabla\varphi_0) \nabla u,\nabla v\rangle dP,\quad 
\text{for all } u,v \in H_0^1(\bbT^d).$$
Using again the assumption that $p,q \in \calC_+^{2+\beta}(\bbT^d)$, and hence
that $A$ has its eigenvalues bounded from above and below by positive constants over $\bbT^d$,
the above bilinear form  defines an inner product on $H_0^1(\bbT^d)$, 
which is equivalent to the standard inner product $\langle \cdot,\cdot\rangle_{H^1(\bbT^d)}$. 
The characterization~\eqref{eq:weak_sol_L} implies the following simple norm equivalence which will be used
repeatedly, and which we prove for
completeness in Appendix~\ref{app:pf_dual_norm_equiv}. 
\begin{lemma}
\label{lem:dual_norm_equiv}
Let $p,q \in \calC_+^{2+\beta}(\bbT^d)$. Then, for all $u \in H_0^2(\bbT^d)$, 
$$\|Lu\|_{H^{-1}(\bbT^d)} \asymp \|u\|_{H^{1}(\bbT^d)},$$
where the implicit constants depend only on $\omega_{2+\beta}(p,q),d,\beta$.
\end{lemma} 
With these properties in place, we turn to the proof of Theorem~\ref{thm:linearization}. 
 
\subsection{Linearization of Monge-Amp\`ere: Proof of Theorem~\ref{thm:linearization}}
Recall that we define $\calQ=[0,1]^d$. Throughout the proof, let $C = C(\omega_{2+\beta}(p,q,\hat q),d,\beta) > 0$
denote a constant whose value is permitted to change from line to line.
Under the assumption $p,q,\hat q \in \calC_+^{2+\beta}(\bbT^d)$, it follows from Theorem~\ref{thm:caffarelli} that,
\begin{equation}\label{eq:linearization_prop1}
\|\hat\varphi\|_{\calC^{2+\beta}(\calQ)} \vee \|\varphi_0 \|_{\calC^{2+\beta}(\calQ)} \leq C
\end{equation}
and
\begin{equation}\label{eq:linearization_prop2}
\nabla^2 \hat \varphi \succeq I_d/C, \quad \nabla^2 \varphi_0 \succeq I_d/C,\quad\text{over } \calQ.
\end{equation}
With the above regularity estimates, we may define the operators
$$\Psi: \calC_0^{2+\beta}(\calQ) \to \calC_0^{\beta}(\calQ),\quad
\Psi[\varphi] = p-\det(\nabla^2\varphi)q(\nabla\varphi),$$
and
$$\widehat \Psi:\calC_0^{2+\beta}(\calQ)\to \calC_0^{\beta}(\calQ),\quad
\widehat \Psi[\varphi] = p-\det(\nabla^2\varphi) \hat q(\nabla\varphi).$$
As shown in the following Lemma, these operators are Fr\'echet differentiable, with 
derivatives that are closely related to the operator $E$.
\begin{lemma}
 \label{lem:frechet}
Let $p,q,\hat q \in \calC^{2+\beta}_+(\bbT^d)$. Then, the maps $\Psi$ and $\widehat \Psi$
are Fr\'echet differentiable at any strongly convex function $\varphi \in \calC_0^{2+\beta}(\calQ)$, with 
Fr\'echet derivatives respectively given for all $u \in \calC_0^{2+\beta}(\calQ)$ by 
$$\Psi_{\varphi}'u = - \det(\nabla^2\varphi)\Big[q(\nabla\varphi)\langle (\nabla^2\varphi)^{-1}, \nabla^2 u\rangle +	\langle \nabla u, \nabla q(\nabla \varphi)\rangle\Big],$$
and,
$$\widehat\Psi_{\varphi}'u  =  -
\det(\nabla^2\varphi)\Big[\hat q(\nabla\varphi)\langle (\nabla^2\varphi)^{-1}, \nabla^2 u\rangle +	\langle \nabla u, \nabla \hat q(\nabla \varphi)\rangle\Big].$$
Furthermore, there exists a constant $C = C(\omega_{2+\beta}(p,q,\hat q),\beta) > 0$ such that 
for all $u \in \calC_0^{2+\beta}(\calQ)$,
$$\Big\| \Psi[\varphi+u] - \Psi[u] - \Psi_{\varphi}'[u]\Big\|_{\calC^\beta(\calQ)}\leq C \|u\|_{\calC^{2+\beta}(\calQ)}^2.$$
The above display also continues to hold with $\Psi_\varphi$ replaced by $\hat\Psi_\varphi$. 
\end{lemma}
The proof of Lemma~\ref{lem:frechet} is elementary, and appears in Appendix~\ref{app:pf_lem_frechet}.  
%
%
Notice that
$\hat\Psi [\hat\varphi] = \Psi[\varphi_0] = 0$, and thus
\begin{align}
\label{eq:z_step}
\widehat\Psi[\hat\varphi] - \widehat \Psi[\varphi_0]
 &= \Psi[\varphi_0]-\widehat\Psi[\varphi_0] 
 =  (\hat q(\nabla\varphi_0)-q(\nabla\varphi_0))\det(\nabla^2\varphi_0).
\end{align} 
Thus, applying Lemma~\ref{lem:frechet} under properties~\eqref{eq:linearization_prop1}--\eqref{eq:linearization_prop2}, we deduce that
\begin{align}
\label{eq:prior_psin_psi}
\left\|\widehat \Psi_{\varphi_0}'[\hat\varphi - \varphi_0] - (\hat q(\nabla\varphi_0)-q(\nabla\varphi_0))\det(\nabla^2\varphi_0)\right\|_{\calC^{\beta}(\calQ)} 
  \leq C\|\hat \varphi - \varphi_0\|_{\calC^{2+\beta}(\calQ)}^2.
\end{align}
The following Lemma will allow us to replace $\widehat \Psi_{\varphi_0}$ by 
$\Psi_{\varphi_0}$ in the above display
\begin{lemma}
\label{lem:comparing_psin_psi}
Assume the same conditions as Theorem~\ref{thm:linearization}.
Then,   
there exist a constant $C =C(\omega_{2+\beta}(p,q,\hat q),d,\beta) > 0$ such that
for all $u \in \calC^{2+\beta}_0(\bbT^d)$, 
\begin{align*} 
\big\|\widehat\Psi_{\varphi_0}'[u] -  \Psi_{\varphi_0}'[u]\big\|_{\calC^\beta(\bbT^d)}  
&\leq C \Big(\|u\|_{\src} \|\hat q-q\|_{\calC^{\beta}(\bbT^d)}+
\|u\|_{\calC^{1+\beta}(\bbT^d)} \|\hat q-q\|_{\calC^{1+\beta}(\bbT^d)}\Big). 
\end{align*}
\end{lemma}
We defer the proof to Appendix~\ref{app:lem_comparing_psin_psi}.
Implicit in the above assertion is the fact that the map $\widehat\Psi_{\varphi_0}'[u] -  \Psi_{\varphi_0}'[u]$ 
is $\bbZ^d$-periodic, which can be seen by direct inspection using Theorem~\ref{thm:caffarelli}. 
The latter result also implies that the map $u=\hat\varphi - \varphi_0$ is $\bbZ^d$-periodic, 
and thus lies in $\calC_0^{2+\beta}(\bbT^d)$ since $\int_{\bbT^d}(\hat \varphi-\varphi_0)d\calL = 
\int_\calQ (\hat \varphi-\varphi_0)d\calL = 0$.  
Returning to equation~\eqref{eq:prior_psin_psi} and applying
 Lemma~\ref{lem:comparing_psin_psi}, we deduce that
\begin{align}\label{eq:pf_linearization_step_three}
\nonumber\Big\|   \Psi_{\varphi_0}'[\hat\varphi &- \varphi_0] -(\hat q(\nabla\varphi_0)-q(\nabla\varphi_0))\det(\nabla^2\varphi_0)\Big\|_{\calC^{\beta}(\bbT^d)}  \\ 
  &\leq C\Big(\|\hat\varphi - \varphi_0\|_{\src}^2+\|\hat q-q\|_{\calC^{\beta}(\bbT^d)}\|\hat\varphi - \varphi_0\|_{\src}\\ \nonumber 
&  \hspace{2in}+
  \|\hat\varphi-\varphi_0\|_{\calC^{1+\beta}(\bbT^d)}\|\hat q - q\|_{\calC^{1+\beta}(\bbT^d)}\Big).
\end{align}
We bound the right-hand side using the following stability bound, whose proof appears in the following subsection.
\begin{proposition}
\label{prop:potential_stability}
Assume the same conditions as Theorem~\ref{thm:linearization}.
Then, there exists a constant $c = c(\omega_{2+\beta}(p,q,\hat q),d,\beta) > 0$ such that
\begin{equation}
\label{eq:thm_potential_stability_C2}
\|\hat \varphi - \varphi_0\|_{\calC^{2+\beta}(\bbT^d)} \leq c \|\hat q - q \|_{\calC^\beta(\bbT^d)}.
\end{equation}
\end{proposition} 
Applying Proposition~\ref{prop:potential_stability} to equation~\eqref{eq:pf_linearization_step_three}, and using
the fact that $\Psi'_{\varphi_0} = E$ by Lemma~\ref{lem:relation_E_L}, we arrive at
\begin{align*}
\left\|E[\hat\varphi-\varphi_0]  - \det(\nabla^2\varphi_0)(\hat q(\nabla\varphi_0) - q(\nabla\varphi_0)) \right\|_{\calC^\beta(\bbT^d)}  
 \leq C\|\hat q - q\|_{\calC^{\beta}(\bbT^d)}\|\hat q-q\|_{\calC^{1+\beta}(\bbT^d)}.
\end{align*}
Now, recall  that the restriction of $E$ 
to $\calC^{2+\beta}_0(\bbT^d)$ is a bijection onto $\calC^\beta_0(\bbT^d)$, with bounded inverse $E^{-1}$. 
Furthermore, the function $\det(\nabla^2\varphi_0)(\hat q (\nabla\varphi_0)- q(\nabla\varphi_0))$
is easily seen to have mean zero, and thus lies in $\calC_0^\beta(\bbT^d)$
by assumption. It follows that
\begin{align*}
\left\|\hat\varphi-\varphi_0 - E^{-1}\left[\det(\nabla^2\varphi_0)(\hat q(\nabla\varphi_0) - q(\nabla\varphi_0))\right]\right\|_{\calC^{2+\beta}(\bbT^d)} \leq  C \|\hat q - q\|_{\calC^{\beta}(\bbT^d)} \|\hat q-q\|_{\calC^{1+\beta}(\bbT^d)}.
\end{align*}
Recalling Lemma~\ref{lem:inverse_L_to_E}, we   deduce
\begin{align}
\left\| (\hat\varphi-\varphi_0)  - L^{-1}\left[ \hat q - q  \right]\right\|_{\calC^{2+\beta}(\bbT^d)}   
 \leq C\|\hat q - q\|_{\calC^{\beta}(\bbT^d)} \|\hat q-q\|_{\calC^{1+\beta}(\bbT^d)},
\end{align}
thus proving the claim.\qed 

\subsection{H\"older Stability Bound: Proof of Proposition~\ref{prop:potential_stability}}
We will  again apply Lemma~\ref{lem:frechet}. Reasoning as in equation~\eqref{eq:z_step}, 
and applying the mean value theorem, we have
$$\big(\hat q(\nabla\varphi_0) - q(\nabla\varphi_0)\big) \det(\nabla^2\varphi_0)
 = \hat\Psi[\hat\varphi] - \hat\Psi[\varphi_0]
 = \hat\Psi_{\tilde\varphi}'[\hat\varphi -\varphi_0],$$
for some map $\tilde \varphi$ of the form $\tilde\varphi = (1-t)\varphi_0+t \hat\varphi$ 
for $t \in [0,1]$.
In particular, it follows that 
\begin{equation}\label{eq:MA_is_linear}
\langle \calA, \nabla^2 (\hat\varphi-\varphi_0)\rangle + \langle b, \nabla (\hat\varphi-\varphi_0)\rangle 
= \hat q(\nabla\varphi_0) - q(\nabla\varphi_0),
\end{equation}
where 
\begin{align*}
\calA = -\frac{\det(\nabla^2\tilde\varphi)}{\det(\nabla^2\varphi_0)} \hat q(\nabla\tilde\varphi) (\nabla^2\tilde\varphi)^{-1},
\quad \text{and}
\quad b = -\frac{\det(\nabla^2\tilde\varphi)}{\det(\nabla^2\varphi_0)}  \nabla \hat q(\nabla\tilde\varphi).
\end{align*}
Notice once again that $\hat\varphi - \varphi_0$ is
$\bbZ^d$-periodic by Theorem~\ref{thm:caffarelli}, and likewise, the entries of $\calA$ 
and the right-hand side of~\eqref{eq:MA_is_linear} are $\bbZ^d$-periodic. Thus, 
the equality~\eqref{eq:MA_is_linear} defines a second-order elliptic equation over $\bbT^d$.
From equation~\eqref{eq:linearization_prop2}, and the boundedness and positivity of the densities
$p,q,\hat q$, we have  $\calA  \succeq  I_d/C$ over $\bbT^d$. 
Therefore, the operator $\langle \calA, \nabla^2(\cdot)\rangle+\langle b,\nabla(\cdot)\rangle$ is uniformly elliptic. Furthermore, using
property~\eqref{eq:linearization_prop1} and the regularity of the densities, the map $\calA$ satisfies the conditions
of the  interior Schauder estimates and the De Giorgi-Nash-Moser estimates (cf. Lemmas~\ref{lem:a_priori}--\ref{lem:a_priori_moser}).
We deduce,
\begin{align*}
\|\hat\varphi-\varphi_0\|_{\calC^{2+\beta}(\bbT^d)} 
 &\lesssim
\|\hat\varphi - \varphi_0\|_{L^2(\bbT^d)} + \left\| \hat q(\nabla\varphi_0) - q(\nabla\varphi_0)\right\|_{\calC^\beta(\bbT^d)} \\
 &\lesssim  
\|\hat\varphi - \varphi_0\|_{L^2(\bbT^d)} + \left\| \hat q  - q \right\|_{\calC^\beta(\bbT^d)}
\end{align*}
where the final inequality follows from Lemma~\ref{lem:holder_comp}, 
using the fact that  $\nabla\varphi_0$
is a $\calC^{1+\beta}(\bbT^d)$-diffeomorphism.
To conclude, we use the stability bound of Lemma~\ref{lem:L2_stab} to obtain
$$\|\hat\varphi - \varphi_0\|_{H^1(\bbT^d)} \lesssim \|\hat q - q\|_{H^{-1}(\bbT^d)}.$$
Combining the preceding two displays leads to the claim.\qed 

\begin{remark}[Different Source Measures]\label{rem:two_sample}
By a simple extension of the preceding proofs, it is also 
possible to derive a linearization bound under variations of both
the source and target measures. 
Concretely, it can be shown that for any densities $p,q,\hat p,\hat q \in \calC_+^{2+\beta}(\bbT^d)$, if 
 $\widebar \varphi$ denotes the unique continuously differentiable convex function,
with mean zero over $\calQ$,  whose gradient pushes forward
$\hat p$ onto $\hat q$, then, one has
\begin{align*}
\Big\| (\widebar \varphi-\varphi_0) &- E^{-1}[p-\hat p] - 
L^{-1}[\hat q-q]\Big\|_{\calC^{2+\beta}(\bbT^d)}\\ 
&\lesssim \|\hat p-p\|_{\calC^{1+\beta}(\bbT^d)} \|\hat p- p\|_{\calC^\beta(\bbT^d)}
+ \|\hat q-q\|_{\calC^{1+\beta}(\bbT^d)} \|\hat q- q\|_{\calC^\beta(\bbT^d)}.
\end{align*}
Such a bound may be used to derive 
the pointwise limiting distribution 
of the optimal transport map pushing forward
a kernel density estimator onto another, which would yield a two-sample analogue
of Theorem~\ref{thm:clt_map}. We omit this extension in the interest of brevity.
\end{remark}

\section{Bias and Variance Bounds}
\label{sec:bias_variance_bounds}
We now return our attention  to the estimator $\hat T_n = \nabla\hat\varphi_n$ 
defined in Section~\ref{sec:introduction}. 
Theorem~\ref{thm:linearization}
suggests that for any given $x \in \bbT^d$, the sequence $(\hat T_n - T_0)(x)$ is well-approximated
by  
\begin{align}\label{eq:bias_variance}
\nabla L^{-1}[\hat q_n-q] (x)
 &= \nabla L^{-1}[\hat q_n - q_{h_n}](x) + \nabla L^{-1}[q_{h_n} - q](x). 
\end{align}
As discussed in Section~\ref{sec:introduction},   the  stochastic  term on the right-hand side of the above display 
will characterize the fluctuations of $\hat T_n(x)$ around its mean, 
while the deterministic term above will characterizes the bias of $\hat T_n(x)$. We analyze  these two quantities next.

By linearity of $L$, the stochastic term  can be decomposed as
\begin{align}
\label{eq:stochastic_term_summation}
\nabla L^{-1}[\hat q_n - q_{h_n}](x)
  = \frac 1 n \sum_{i=1}^n \nabla L^{-1} \big[ \widebar K_{h_n}(Y_i-\cdot) - q_{h_n}\big](x),
\end{align}
where $\widebar K_{h_n}$ denotes the $\bbZ^d$-periodization of $K_{h_n}$, as defined in Appendix~\ref{app:kde}. 
The following result describes the limiting covariance of the summands in the above display.
\begin{lemma} 
\label{lem:variance_bound_map}
Let $p,q \in \calC_+^{2+\beta}(\bbT^d)$ for some $\beta \in (0,1)$, and let the kernel $K$ satisfy condition~\Kernel{$\gamma$} for some $\gamma > 0$. 
Let $Y \sim Q$. 
Then, there exist constants $C,\epsilon > 0$
depending only on $\omega_{2+\beta}(p,q), K, d,\beta$ such that
for any $x \in \bbT^d$, 
$$\left\| h_n^{d-2} \Cov\Big\{ \nabla L^{-1}\big[ \widebar K_{h_n}(Y-\cdot) - q_{h_n}\big](x)\Big\}  - \Sigma(x)\right\| 
\leq C h_n^{\epsilon},$$
where $\Sigma(x)$ is the positive definite matrix defined in equation~\eqref{eq:Sigma}.
\end{lemma}
Lemma~\ref{lem:variance_bound_map} implies that the 
 covariance matrix of any given term in the summation~\eqref{eq:stochastic_term_summation} has norm 
diverging at the pointwise rate $h_n^{2-d}$, which  coincides with the corresponding $L^2(\bbT^d)$ rate of divergence.
Indeed, by Lemma~\ref{lem:dual_norm_equiv}, it holds that
$$\bbE\big\| \nabla L^{-1}\big[ \widebar K_{h_n}(Y-\cdot) - q_{h_n}\big]\big\|_{L^2(\bbT^d)}^2 
\asymp \bbE  \big\| \widebar K_{h_n}(Y-\cdot) - q_{h_n}\big\|_{H^{-1}(\bbT^d)}^2\asymp h_n^{2-d}, $$
where the final order assessment can be deduced as in the proof of Proposition~\ref{prop:kde_negative_sobolev}.
The proof of Lemma~\ref{lem:variance_bound_map} turns out to require a  more involved argument, 
and appears in Section~\ref{sec:pf_variance_bound_map}. 
%
%
%

Let us now provide a bound on the deterministic term. 
\begin{lemma}
\label{lem:bias_bound_map}
Let $p,q \in \calC_+^{s}(\bbT^d)$. Assume that $K$ satisfies
condition~\Kernel{$s+1$}. Then, for all $\epsilon > 0$, there exists $C = C(\omega_{s}(p,q),   K, \epsilon,d,s) > 0$
such that
$$\big\| \nabla L^{-1} [q_{h_n}-q]\big\|_{L^\infty(\bbT^d)}  \leq C h_n^{s+1-\epsilon}.
$$\end{lemma}
The proof of Lemma~\ref{lem:bias_bound_map} appears in Section~\ref{sec:pf_bias_bound_map}, and is a consequence of existing
gradient estimates for uniformly elliptic PDEs. 
Once again, up to the arbitrarily small constant $\epsilon > 0$, the rate $h_n^{s+1-\epsilon}$ in the above display 
coincides with the corresponding $L^2(\bbT^d)$ rate of decay:  one has, 
by Lemmas~\ref{lem:dual_norm_equiv} and~\ref{lem:equiv_H1_norm}, 
$$\big\| \nabla L^{-1} [q_{h_n}-q]\big\|_{L^2(\bbT^d)} \asymp 
  \| q_{h_n}-q\|_{H^{-1}(\bbT^d)} \asymp h_n^{s+1},
$$
where the final order assessment again follows from Proposition~\ref{prop:kde_negative_sobolev}, 
under the assumption~\Kernel{$s+1$}. 
\subsection{Proof of Lemma~\ref{lem:variance_bound_map}}
\label{sec:pf_variance_bound_map}
Throughout the proof, the implicit constants in the symbols $\lesssim$ and $\asymp$ 
depend only on $\omega_{2+\beta}(p,q),K,d,\beta$.
In particular, such constants do not depend on the variable $x$. 
By Theorem~\ref{thm:caffarelli}, there exists $\lambda > 1$   such that $\|\varphi_0\|_{\calC^3(\bbT^d)} \leq \lambda$
and $\nabla^2\varphi_0^* \succeq \lambda^{-1} I_d$, and we fix this value of $\lambda$ throughout the proof. 
Fix $x \in \bbT^d$ throughout what follows, and abbreviate $X_0 = \nabla\varphi_0^*(Y)$.
We may assume without loss of generality
that the representative of $X_0$ in $\bbT^d$ is chosen such that 
\begin{align}\label{eq:representative_of_X0}
\|X_0-x\|_{\bbT^d} = \|X_0 - x\|.
\end{align}
Next, we abbreviate 
$$ \Sigma_n(x) = \Cov\big(\nabla L^{-1}[\widebar K_{h_n}(Y-\cdot) - q_{h_n}](x)\big).$$
Since $q_{h_n}$ is deterministic, and $L$ is linear, we may rewrite $\Sigma_n(x)$ as
$$\Sigma_n(x) = \Cov\big(\nabla L^{-1}[\widebar K_{h_n}^o(Y-\cdot)](x)\big),\quad \text{with } \widebar K_{h_n}^o =\widebar  K_{h_n}-1.$$
Deduce from Lemma~\ref{lem:inverse_L_to_E} that
\begin{align}
\label{eq:key_lemma_step1}
\Sigma_n (x) = 
  \Cov \big(\nabla E^{-1}\left[\det(\nabla^2\varphi_0) \widebar K_{h_n}^o(Y-\nabla\varphi_0(\cdot))\right](x)\big).
\end{align}
We will derive the limit of this quantity by approximating the operator $E$ with a 
constant-coefficient differential operator, and showing that the remainder of 
this approximation is of low order. We proceed in several steps.
The proofs of all intermediary results which follow are relegated to Appendix~\ref{app:extra_proofs_variance}. 
\\

\noindent {\bf Step 1: Reduction to a constant-coefficient operator.}
Using Lemma~\ref{lem:relation_E_L}, we may write for any $u \in H_0^2(\bbT^d)$, 
$$E u = -\langle  A, \nabla^2 u\rangle - \langle b, \nabla u \rangle,$$
where 
$$A =p\nabla^2\varphi_0^*(\nabla\varphi_0),\quad\text{and}\quad b =   
 \det(\nabla^2\varphi_0)\nabla q(\nabla\varphi_0).$$
Let $E_0$ be defined as the leading term in $E$ with coefficients ``frozen'' at 
the point $x$, namely:
$$E_0u(y) = -\langle A_0, \nabla^2 u(y)\rangle = -\div(A_0 \nabla^2u(y)),\quad y\in \bbT^d$$
where $A_0 = A(x)$ is a constant matrix. 
Furthermore, abbreviate $\calB= \nabla^2\varphi_0$ and $\calB_0 = \nabla^2\varphi_0(x)$. 
Define the linear map
$$S_x(y) = \nabla\varphi_0(x) + \nabla^2\varphi_0(x)(y-x),\quad y \in \bbT^d,$$
and set
\begin{alignat*}{2}
u_n &=  E^{-1}\big[\det(\calB)\widebar K_{h_n}^o(Y-\nabla\varphi_0(\cdot))\big],
\quad U_n &&= \nabla u_n,\\
  v_n &=  E_0^{-1}\big[\det(\calB)\widebar K_{h_n}^o(Y-\nabla\varphi_0(\cdot))\big],
  \quad V_n &&= \nabla v_n,\\
 w_n &=  E_0^{-1}\big[\det(\calB_0)\widebar K_{h_n}^o(Y-S_x(\cdot))\big],
 \quad W_n&&=\nabla w_n,
 \end{alignat*}
 where we emphasize that each of the functions appearing in square brackets in the above display
 has mean zero, so that the inverses are well-defined.
%
We have the decomposition
\begin{align}
\label{eq:covariance_decomp}
\nonumber
\Sigma_n(x)
 &= \Cov[U_n(x)] \\
\nonumber 
 &= \Cov[W_n(x)] + \Cov[(V_n-W_n)(x)]  + \Cov[(U_n-V_n)(x)]   \\ &+ 2\Cov[W_n(x), (U_n-W_n)(x)] 
   + 2\Cov[(V_n-W_n)(x), (U_n-V_n)(x)],
  \end{align}
where for two random vectors $A$ and $B$, we denote by $\Cov(A,B) = \bbE[(A-\bbE[A])(B-\bbE[B])^\top]$
the cross-covariance between $A$ and $B$.
Fix
$$\calV_1 =   \Cov[W_n(x)], \quad 
  \calV_2 = \bbE\|(V_n-W_n)(x)\|^2 , \quad 
  \calV_3 = \bbE\|(U_n-V_n)(x)\|^2.$$
We will show that for some $\epsilon > 0$, 
\begin{equation}
\|h_n^{d-2} \calV_1 - \Sigma(x)\|\lesssim h_n^\epsilon,
\end{equation}
and that 
$\calV_2 \vee \calV_3 \lesssim h_n^{2-d+2\beta}$.
Using the Cauchy-Schwarz
inequality, it will then 
immediately follow that the three cross-covariances in equation~\eqref{eq:covariance_decomp} are of order at most 
$h_n^{2-d-\beta}$,
and the claim will then follow.
  We analyze each of the terms 
  $\calV_1, \calV_2, \calV_3$ in turn. 
\\

\noindent {\bf Step 2: The limit of term $\calV_1$.} 
Write
$$S_{Y,x}(y) = Y - S_x(y) = Y- \nabla\varphi_0(x) - \calB_0(y-x).$$
For this step only, we will require the additional constant-coefficient operator
$$G_0: H_0^2(\bbT^d) \to L_0^2(\bbT^d), \quad G_0u(y) = -\langle \calB_0^2 A_0 , \nabla^2 u(y)\rangle = -\langle \calC_0, \nabla^2 u(y)\rangle,$$
where $\calC_0 =p(x)\nabla^2\varphi_0(x)$. 
The results of Appendix~\ref{app:pde} imply that the restriction of $G_0$ to $\calC^{2+\beta}_0(\bbT^d)$
is a bijection 
onto $\calC_0^\beta(\bbT^d)$. 
This operator is motivated by the following simple observation.
\begin{lemma}
\label{lem:operator_G0}
For all $f \in \calC_0^{\beta}(\bbT^d)$ and $y \in \bbT^d$, it holds that
$$E_0^{-1}[f(S_{Y,x}(\cdot))](y) =  G_0^{-1}[f](S_{Y,x}(y)).$$
\end{lemma}
The proof appears in Appendix~\ref{app:pf_lem_operator_G0}.
 Deduce from Lemma~\ref{lem:operator_G0} that for any $f \in \calC^\beta_0(\bbT^d)$ and $y \in \bbT^d$, 
$$\nabla_y E_0^{-1}[f(S_{Y,x}(\cdot))] (y) = \nabla_y G_0^{-1}[f](S_{Y,x}(y))
 = -\calB_0 \nabla G_0^{-1}[f](S_{Y,x}(y)),$$
so that
\begin{align*}
W_n(x) 
 = \nabla E_0^{-1}\big[\det(\calB_0)\widebar K_{h_n}^o(S_{Y,x}(\cdot))\big](x) 
 = -\calB_0 {\det}(\calB_0)\nabla  G_0^{-1}\big[\widebar K_{h_n}^o\big](Y- \nabla\varphi_0(x)),
\end{align*}         
whence,
\begin{align}
\label{eq:pf_var_V1_first_reduction}
\calV_1 &=  {\det}^2(\calB_0) \calB_0 \Cov\big\{\nabla G_0^{-1}\big[\widebar K_{h_n}^o\big](Y- \nabla\varphi_0(x))\big\} \calB_0^\top.
\end{align} 
The following Lemma shows that
the covariance in the above display can be taken with respect to the uniform law, without substantial loss. 
\begin{lemma}
\label{lem:Q_to_uniform_reduction}
Let $U\sim \calL$ be a uniform random variable on $\bbT^d$. Then, for some $\epsilon > 0$,
$$\Big\| \calV_1 - q(\nabla\varphi_0(x)) {\det}^2(\calB_0)\calB_0 \Cov\big\{\nabla G_0^{-1}\big[\widebar K_{h_n}^o\big](U)\big\}  
\calB_0^\top\Big\|
\lesssim h_n^{2+\epsilon-d}.$$
\end{lemma}
The proof appears in Appendix~\ref{app:pf_lem_Q_to_uniform_reduction}. 
Our aim is now to study the covariance appearing in the above display.
Recall that $G_0   =- \langle \calC_0,\nabla^2 (\cdot)\rangle$.
%
Since $\calC_0$ is a constant matrix, a simple derivation shows
$$G_0^{-1}[\widebar K_{h_n}^o] = \sum_{\xi \in \bbZ_*^d} \frac{\calF[\widebar K_{h_n}](\xi)}{(2\pi)^2\langle \calC_0\xi,\xi\rangle} 
e^{2\pi i \langle\xi,\cdot\rangle}.$$
Since $K\in \calC_c^\infty(\bbR^d)$, 
the above Fourier series converges absolutely and uniformly
for any given $n$~\citep{stein1971}, and we may differentiate it term-by-term to obtain
$$\nabla G_0^{-1}[\widebar K_{h_n}^o] = \sum_{\xi \in \bbZ_*^d} \frac{i\xi \calF[\widebar K_{h_n}](\xi)}{ 2\pi\langle \calC_0\xi,\xi\rangle} 
e^{2\pi i \langle \xi,\cdot\rangle} = 
\sum_{\xi \in \bbZ_*^d} \frac{i\xi \calF[K](h_n\xi)}{ 2\pi  \langle \calC_0\xi,\xi\rangle} e^{2\pi i \langle\xi,\cdot\rangle},$$
where we used the Poisson summation formula in the final equality (cf. equation~\eqref{eq:poisson_summation}). 
Again, the above series converges uniformly, and we thus have
$$\bbE\big[ \nabla G_0^{-1}[\widebar K_{h_n}^o] (U)\big] =
\sum_{\xi \in \bbZ_*^d} \frac{i\xi \calF[K](h_n\xi)}{ 2\pi  \langle \calC_0\xi,\xi\rangle} \int_{\bbT^d} e^{2\pi i \langle\xi,y\rangle}dy=0.$$
Deduce  that 
\begin{align}\label{eq:pf_var_V2_second_reduction}
\nonumber 
\Cov&\big\{\nabla G_0^{-1}\big[\widebar K_{h_n}^o\big](U)\big\} \\
\nonumber 
 &= \bbE\left\{\left( \sum_{\xi\in \bbZ_*^d}\frac{i\xi \calF[K](h_n\xi)}{2\pi\langle \calC_0\xi,\xi\rangle} e^{2\pi i \langle\xi,U\rangle}\right)
               \left( \sum_{\zeta\in \bbZ_*^d}\frac{-i\zeta \calF[K](h_n\zeta)}{2\pi \langle \calC_0\zeta,\zeta\rangle} e^{-2\pi i \langle\zeta,U\rangle}\right)^\top  \right\}\\ 
\nonumber               
 &= \bbE\left\{  \sum_{\xi,\zeta\in \bbZ_*^d} \frac{\xi\zeta^\top \calF[K](h_n\xi)\calF[K](h_n\zeta)}{(2\pi)^2 \langle \calC_0\xi,\xi\rangle\langle \calC_0\zeta,\zeta\rangle} e^{2\pi i \langle\xi-\zeta,U\rangle}  \right\}\\
 &=  \sum_{\xi \in \bbZ_*^d} \xi\xi^\top \left(\frac{ \calF[K](h_n\xi)}{ 2\pi  \langle \calC_0\xi,\xi\rangle } \right)^2.
\end{align}  
%
Up to suitable scaling, the expression~\eqref{eq:pf_var_V2_second_reduction} is as a matrix-valued 
improper Riemann sum, whose limit is described next. 
\begin{lemma}
\label{lem:riemann_sum_approx}
There exist constants $c,\epsilon > 0$ depending only on $d$ and $K$ such that
$$\left\| h_n^{d-2} \Cov \big\{\nabla G_0^{-1}\big[\widebar K_{h_n}^o\big](U)\big\}  - \int_{\bbR^d} \xi\xi^\top 
\left(\frac{ \calF[K](\xi)}{ 2\pi  \langle \calC_0\xi,\xi\rangle } \right)^2d\xi\right\|\leq c h_n^\epsilon.$$
\end{lemma}
The proof appears in Appendix~\ref{app:pf_lem_riemann_sum_approx}. 
By Lemmas~\ref{lem:Q_to_uniform_reduction}--\ref{lem:riemann_sum_approx}, after possibly decreasing the value of $\epsilon$, we deduce that
$$\big\| h_n^{d-2} \calV_1 - \widetilde \Sigma(x)\big\| \lesssim h_n^\epsilon,$$
where
$$\widetilde \Sigma(x) = q(\nabla\varphi_0(x)) {\det}^2(\calB_0)\calB_0 \left(\int_{\bbR^d} \xi\xi^\top 
\left(\frac{ \calF[K](\xi)}{ 2\pi  \langle \calC_0\xi,\xi\rangle } \right)^2d\xi\right)\calB_0^\top.$$
To complete our analysis of term $\calV_1$, it thus remains to show that  $\Sigma(x) = \widetilde \Sigma(x)$. Recalling the definition of $\calC_0$, we have
\begin{align*}
\widetilde \Sigma(x)
 &= \frac{q(\nabla\varphi_0(x)){\det}^2(\calB_0)}{p^2(x)} \int_{\bbR^d} (\calB_0\xi)(\calB_0\xi)^\top  \left(\frac{ \calF[K](\xi)}{2\pi \langle \calB_0 \xi,\xi\rangle}\right)^2    d\xi.
\end{align*}
Applying the change-of-variable $\zeta = \calB_0 \xi$, we arrive at 
\begin{align*}
\widetilde\Sigma(x)
 &= \frac{q(\nabla\varphi_0(x)){\det}(\calB_0)}{p^2(x)} \int_{\bbR^d} \zeta\zeta^\top  \left(\frac{ \calF[K](\calB_0^{-1}\zeta)}{2\pi \langle  \zeta, \calB_0^{-1}\zeta\rangle}\right)^2     d\zeta \\
 &= \frac{1}{p(x)} \int_{\bbR^d} \zeta\zeta^\top  \left(\frac{ \calF[K](\calB_0^{-1}\zeta)}{2\pi \langle  \zeta, \calB_0^{-1}\zeta\rangle}\right)^2     d\zeta = \Sigma(x),
\end{align*}
where we used the Monge-Amp\`ere equation $\det(\nabla^2\varphi_0) = p/q(\nabla\varphi_0)$.
We thus conclude that
$$\|h_n^{d-2} \calV_1 - \Sigma(x)\| \lesssim h_n^\epsilon.$$
Finally, let us also argue that $\Sigma(x)$ is positive definite. We have, for all $v \in \bbR^d$, 
\begin{align*}
v^\top \Sigma(x)v
 &= \frac{1}{p(x)} \int_{\bbR^d} (v^\top\zeta)^2 \left(\frac{ \calF[K](\calB_0^{-1}\zeta)}{2\pi \langle  \zeta, \calB_0^{-1}\zeta\rangle}\right)^2     d\zeta \\
 &\geq c   \int_{\bbR^d} (v^\top\zeta)^2 \|\zeta\|^{-4}\left(  \calF[K](\calB_0^{-1}\zeta) \right)^2     d\zeta,
\end{align*}
for a constant $c > 0$ not depending on $v$, where we used the fact that $\calB_0$ has all eigenvalues lying in the 
positive interval
$[\lambda^{-1},\lambda]$. By assumption~\Kernel{$\gamma$}, there exists $\kappa > 0$ such that
$$|\calF[K](\calB_0^{-1}\xi) - 1| \leq \kappa\|\calB_0^{-1}\xi\|^\gamma \leq \kappa\lambda^\gamma \|\xi\|^\gamma,$$
thus, letting $S$ be the ball $\{\xi: \kappa\lambda^\gamma \|\xi\|^\gamma \leq 1/2\}$, we have 
\begin{align*}
v^\top \Sigma(x)v
  \geq c   \int_{S} (v^\top\zeta)^2 \|\zeta\|^{-4}\left( 1 - \kappa \|\calB_0^{-1}\zeta\| \right)^2     d\zeta 
 \geq \frac c 4  \int_S (v^\top\zeta)^2 \|\zeta\|^{-4}     d\zeta  > 0,
\end{align*}
which implies that $\Sigma(x)$ is positive definite. 
This concludes Step 2 of the proof.
\\

\noindent {\bf Step 3: Bounding term $\calV_2$.} Our aim is now to bound 
\begin{align*}
\calV_2 &= \bbE \big\|\nabla E_0^{-1}\big[\det(\calB) \widebar K_{h_n}^o(Y-\nabla\varphi_0(\cdot))-  \det(\calB_0) 
\widebar K_{h_n}^o(Y-S_x(\cdot)) \big](x)\big\|^2.
\end{align*}
By Lemma~\ref{lem:fundamental_solution}, the operator $E_0$ admits
a periodic Green's function $\Gamma_{E_0}(y,x)$, which is continuously
differentiable with respect to $x$ 
away from the diagonal $x=y$, and whose gradient satisfies 
$\|\nabla_x \Gamma_{E_0}(y,x)\| \lesssim \|x-y\|_{\bbT^d}^{1-d}$. 
We may therefore write
\begin{align*}
\calV_2 
 \lesssim \calV_{2,1} + \calV_{2,2},
\end{align*}
where
\begin{align*}
\calV_{2,1} &= 
\bbE\left[\left(\int_{\bbT^d}
\|x-y\|_{\bbT^d}^{1-d}
\det(\calB_0) \big|\widebar K_{h_n}^o(Y-\nabla\varphi_0(y)) - \widebar K_{h_n}^o(Y-S_x(y))\big| dy \right)^2\right], \\
\calV_{2,2} &= 
\bbE\left[\left(\int_{\bbT^d}
\|x-y\|_{\bbT^d}^{1-d}
\big| \widebar K_{h_n}^o(Y-\nabla\varphi_0(y)) |\cdot | \det(\calB) - \det(\calB_0)| dy \right)^2\right].
\end{align*}
%
%
To bound $\calV_{2,1}$, we make use of the following observation.
\begin{lemma}
\label{lem:var_bound_V2_tech}
It holds that\footnote{By abuse of notation, we say the support
of a periodic function is contained in a set $B \subseteq [0,1]^d$ if it is contained
in the $\bbZ^d$-translates of $B$.}
\begin{align*}
\supp \big(\widebar K_{h_n}(Y-\nabla\varphi_0(\cdot))\big) &\subseteq B(X_0, \lambda h_n), \quad \text{and}\\
\supp \big(\widebar K_{h_n}(Y-S_x(\cdot))\big) &\subseteq B(S_x^{-1}(Y), \lambda h_n).
\end{align*} 
\end{lemma}
The proof of Lemma~\ref{lem:var_bound_V2_tech} appears in Section~\ref{sec:pf_lem_var_bound_V2_tech}. 
Let $$B_0 = B(X_0, \lambda h_n) \cup B(S_x^{-1}(Y), \lambda h_n).$$ 
By Lemma~\ref{lem:var_bound_V2_tech},  we have
\begin{align*}
\calV_{2,1}
  &\lesssim \bbE \left[\left(\int_{B_0} \|x-y\|_{\bbT^d}^{1-d} \big|\widebar K_{h_n}(Y-\nabla\varphi_0(y))-\widebar K_{h_n}(Y-S_x(y))\big| dy \right)^2 \right].
\end{align*}
%
%
Now, and define the random variable
$$ \chi = \begin{cases}
0, & d \leq 6 \text{ and }  \|x-X_0\|_{\bbT^d} \wedge \|x-S_x^{-1}(Y)\|_{\bbT^d} \geq \alpha_n\\ 
1, & \mathrm{otherwise}
\end{cases},$$
where, given a fixed scalar $\delta \in (0,1)$, we set $\alpha_n = 2\lambda^3 h_n^{1-\delta}$.
We have the decomposition  $\calV_{2,1}  \lesssim \calV_{2,1,1} + \calV_{2,1,2}+ \calV_{2,1,3}$, 
where
\begin{align*}
\calV_{2,1,1} &= \bbE \left[\left(\int_{B_0} \|x-y\|_{\bbT^d}^{1-d} \big|\widebar K_{h_n}(Y-\nabla\varphi_0(y))-\widebar K_{h_n}(Y-S_x(y))\big| dy \right)^2 \chi\right], \\
\calV_{2,1,2} &= \bbE \left[\left(\int_{B_0} \|x-y\|_{\bbT^d}^{1-d} \big|\widebar K_{h_n}(Y-\nabla\varphi_0(y))\big| dy \right)^2  (1-\chi) \right], \\ 
\calV_{2,1,3} &= \bbE \left[\left(\int_{B_0} \|x-y\|_{\bbT^d}^{1-d} \big|\widebar K_{h_n}(Y-S_x(y))\big| dy \right)^2  (1-\chi)\right].
\end{align*}
 Let us begin by bounding the first term. 
Since $\|\widebar K_{h_n}\|_{\calC^1(\bbT^d)} \lesssim h_n^{-(d+1)}$, we have
\begin{align*}
\begin{multlined}
\big|\widebar K_{h_n}(Y-\nabla\varphi_0(y))-\widebar K_{h_n}(Y-S_x(y))\big| \\ 
\lesssim h_n^{-(d+1)} \|\nabla\varphi_0(y) - \nabla\varphi_0(x) - \nabla^2\varphi_0(x)(y-x)\|_{\bbT^d}
 \lesssim h_n^{-(d+1)} \|x-y\|_{\bbT^d}^2,
 \end{multlined}
  \end{align*}
  thus we obtain 
\begin{align}
\label{eq:pf_var_bound_step_V2}
\calV_{2,1,1}
 &\lesssim 
  h_n^{-2-2d} \bbE \left[ \left(\int_{B_0} \|x-y\|_{\bbT^d}^{3-d}  dy \right)^2 \chi  \right]. 
\end{align}
We now make use of the following Lemma.
\begin{lemma}
\label{lem:ball_integral}
Let $Z$ be a random variable taking values in $\bbT^d$, admitting a density with respect
to $\calL$ which is bounded from above by some $\gamma > 0$ over $\bbT^d$. 
Let $t > 0$. 
Then, there exists $C =C(\gamma,t,d) > 0$ such that for any $\epsilon \in (0,1/2)$
and any $y \in \bbT^d$, 
$$\bbE\left[\left(\int_{B(Z, \epsilon)} \|y-z\|_{\bbT^d}^{t-d}dz\right)^2\right] \leq C \begin{cases}
\epsilon^{d+2t}, & t < d/2 \\ 
\epsilon^{2t},   & d/2 \leq t < d \\ 
\epsilon^{2d},   & d\leq t.
\end{cases}$$
\end{lemma}
The proof appears in Section~\ref{app:pf_lem_ball_integral}.
Notice that the random variables $\nabla\varphi_0^*(Y)$ and $S_x^{-1}(Y)$ 
both have probability laws which are absolutely continuous
with respect to $\calL$, with uniformly upper bounded densities. 
Returning to equation~\eqref{eq:pf_var_bound_step_V2}, we deduce from Lemma~\ref{lem:ball_integral} that
when $d \geq 7$, 
$$\calV_{2,1,1}   \lesssim h_n^{-2-2d} h_n^{d+6} \asymp h_n^{4-d}.$$
On the other hand, when $d \leq 6$, we have by Lemma~\ref{lem:simple_ball_integral} that
\begin{align} 
\calV_{2,1,1}
 &\lesssim 
  h_n^{4-2d} \bbE[\chi]  \lesssim
   h^{4-d-\delta d}.
\end{align}
Let us now bound term $\calV_{2,1,2}$, which is only nonzero when $d \leq 6$. 
By a change of variable, it holds that 
\begin{align*}
\calV_{2,1,2}  &\lesssim \bbE  \left[\left(\int_{B(0, \lambda^2h_n)} \|x-\nabla\varphi_0^*(z-Y)\|_{\bbT^d}^{1-d} \big|\widebar K_{h_n}(z)\big|dz\right)^2  (1-\chi)\right].
\end{align*}
Now, over the event $\chi=0$, we have for all $z \in B(0, \lambda^2h_n)$, 
\begin{align*}
\|x-\nabla\varphi_0^*(z-Y)\|_{\bbT^d}
 &\geq \|x-\nabla\varphi_0^*(Y)\|_{\bbT^d} - \|\nabla\varphi_0^*(Y) -  \nabla\varphi_0^*(z-Y)\|_{\bbT^d} \\ 
 &\geq \|x-\nabla\varphi_0^*(Y)\|_{\bbT^d} - \lambda^3 h_n 
 \gtrsim \|x-X_0\|,  
\end{align*}
where we used equation~\eqref{eq:representative_of_X0}. 
It follows that 
\begin{align*}
\calV_{2,1,2} 
 &\lesssim \bbE  \left[\left(\int_{B(0, \lambda^2h_n)} \|x-X_0\|^{1-d} \big|\widebar K_{h_n}(z)\big|dz\right)^2 (1-\chi)\right] \\
 &\lesssim \bbE  \left[\left( \|x-X_0\|^{1-d} \right)^2 (1-\chi)\right]  
 \lesssim  \int_{\calQ_x\setminus B(x,\alpha_n)} \|x-z\|^{2-2d} dz 
\asymp 
 h_n^{(1-\delta)(2-d)}.
 \end{align*}
A similar bound holds for term $\calV_{2,1,3}$. Choosing $\delta$ small enough, we have thus shown 
\begin{equation}
\label{eq:bound_on_V21}
\calV_{2,1} \lesssim h_n^{2-d+2\beta}.
\end{equation}
Next, we bound term $\calV_{2,2}$. Reasoning similarly as before, we obtain 
\begin{align*}
\calV_{2,2} &\lesssim \bbE\left[\left(\int_{\bbT^d} \|x-y\|_{\bbT^d}^{1-d}
 \big| \widebar K_{h_n}^o(Y-\nabla\varphi_0(y))\big|  \big| \det(\calB(y)) - \det(\calB_0)\big|dy \right)^2\right]\\ 
 &\lesssim  \bbE\left[\left(\int_{B(X_0,\lambda h_n)} \|x-y\|_{\bbT^d}^{2-d}  \big| \widebar K_{h_n}^o(Y-\nabla\varphi_0(y))\big|   dy \right)^2\right] \lesssim  \calV_{2,2,1} + \calV_{2,2,2},
\end{align*}
where we again use the decomposition
\begin{align*}
\calV_{2,2,1} &= \bbE\left[\left(\int_{B(X_0,\lambda h_n)} \|x-y\|_{\bbT^d}^{2-d}  \big| \widebar K_{h_n}^o(Y-\nabla\varphi_0(y))\big|   dy \right)^2\chi \right],\\
\calV_{2,2,2} &= \bbE\left[\left(\int_{B(X_0,\lambda h_n)} \|x-y\|_{\bbT^d}^{2-d}  \big| \widebar K_{h_n}^o(Y-\nabla\varphi_0(y))\big|   dy \right)^2(1-\chi)\right]. 
\end{align*}
We clearly have $\calV_{2,2,2} \lesssim \calV_{2,1,2}$, thus it remains to bound term $\calV_{2,2,1}$.
In the regime $d \geq 5$, we may use Lemma~\ref{lem:ball_integral}
 to obtain
$$\calV_{2,2,1} \lesssim h^{-2d} \bbE\left[\left(\int_{B(X_0,\lambda h_n)} \|x-y\|_{\bbT^d}^{2-d}     dy \right)^2 \right]
\lesssim h_n^{4-d}.$$
On the other hand, when $3 \leq d \leq 4$, we obtain by Lemma~\ref{lem:simple_ball_integral} that
\begin{align}
\nonumber 
 \calV_{2,2,1} 
  &\lesssim 
h^{-2d}\bbE\left[\left(\int_{B(X_0,\lambda h_n)} \|x-y\|_{\bbT^d}^{2-d} 
               dy \right)^2 \chi \right]  
 \lesssim h_n^{-2d + 4} \bbE[\chi] \lesssim 
               h_n^{4-d-\delta d}.
\end{align}
Altogether, taking $\delta$ small enough, we thus obtain 
\begin{equation}
\label{eq:bound_on_V2}
\calV_{2} \lesssim  h_n^{2-d+2\beta}.
\end{equation}
This concludes Step 3 of the proof.
\\

\noindent {\bf Step 4: Bounding term $\calV_3$.}
We now wish to show that $\calV_3 \lesssim h_n^{4-d-2\beta}\vee 1$. 
Notice first that for all $y\in \bbT^d$, 
\begin{align*}
0
 &= E[u_n](y) - E_0[v_n](y) \\ 
 &= -\langle A(y), \nabla^2 u_n (y)\rangle + \langle A_0, \nabla^2 v_n(y) \rangle  -\langle b(y), \nabla u_n(y)\rangle \\
 &= \langle A_0-A(y) , \nabla^2 u_n(y)\rangle - \langle A_0, \nabla^2(u_n-v_n)(y)\rangle - \langle b(y), \nabla u_n(y)\rangle,
\end{align*}
which implies that
$$E_0[u_n-v_n] =  \langle A_0 -  A,\nabla^2 u_n\rangle + \langle b, \nabla u_n(y)\rangle =:f_n,$$
and hence, 
\begin{align*}
(U_n-V_n)(x) 
 &=\nabla E_0^{-1} [f_n\big](x) = \int_{\bbT^d} \nabla_x \Gamma_{E_0}(y,x) f_n(y)dy.
\end{align*} 
It will thus suffice to bound the expected squared norm of the right-hand side. 
Using again the fact that $\|\nabla_x\Gamma_{E_0}(y,x)\| \lesssim \|x-y\|_{\bbT^d}^{1-d}$, we have
\begin{align*}
\|(U_n-V_n)(x) \|
 &= \left\| \int_{\bbT^d} \nabla_x \Gamma_{E_0}(y,x) f_n(y)dy\right\| \\ 
 &\lesssim \int_{\bbT^d} \|x-y\|_{\bbT^d}^{1-d} \big[\| A_0 - A(y)\|\| \nabla^2 u_n(y)\| +\|b(y)\|\|\nabla u_n(y)\|\big]dy \\
 &\lesssim \int_{\bbT^d} \|x-y\|_{\bbT^d}^{1-d} \big[\|x-y\|_{\bbT^d}\| \nabla^2 u_n(y)\| +\|\nabla u_n(y)\|\big]dy, 
\end{align*}
where we used the fact that $ A$ has entries in $\calC^1(\bbT^d)$. We decompose the right-hand side of the above display into the terms
\begin{align*}
\calI_1  &= \int_{\bbT^d} \|x-y\|_{\bbT^d}^{2-d} \| \nabla^2 u_n(y)\| dy,\quad  
\calI_2 = \int_{\bbT^d} \|x-y\|_{\bbT^d}^{1-d} \|\nabla u_n(y)\|dy.
\end{align*}
We have 
$$\calV_3 \lesssim \bbE[\calI_1^2] + \bbE [\calI_2^2],$$
 and we now bound the latter two
terms in turn.
To bound $\bbE[\calI_1^2]$,
we will make use of the following result concerning the regularity of $u_n$. 
\begin{lemma}
\label{lem:regularity_un}
There exists a constant $C_1 = C_1(\omega_{2+\beta}(p,q),\beta,d) > 0$  
such that , 
$$\|\nabla^2 u_n(y)\| \leq C_1 \left( h_n^{-(d+\beta)} \wedge \|y-X_0\|_{\bbT^d}^{-d}\right),
\quad \text{for all } y \in \bbT^d.$$
%
\end{lemma} 
The proof appears in Section~\ref{app:pf_lem_regularity_un}.  
We deduce, 
\begin{align*}
\begin{multlined}
\bbE\left[\left(\int_{B(X_0,C_2 h_n)} \|x-y\|_{\bbT^d}^{2-d} \|\nabla^2 u_n(y)\|dy\right)^2\right] \\
\lesssim h_n^{-2(d+\beta)}\bbE\left[\left(\int_{B(X_0,C_2 h_n)} \|x-y\|_{\bbT^d}^{2-d} dy\right)^2\right] 
\lesssim h_n^{4-d-2\beta},
\end{multlined}
\end{align*}
where the final inequality follows from Lemma~\ref{lem:ball_integral}.  
Write  $\calQ_z = z + [-1/2,1/2]^d$ for all $z \in \bbR^d$. We then have, 
\begin{align}
\label{eq:pf_V3_step_part1}
\bbE[\calI_1^2]
 &\lesssim h_n^{4-d-2\beta} +  \bbE\left[\left(\int_{\calQ_{X_0} \setminus B(X_0,C_2 h_n)} \|x-y\|_{\bbT^d}^{2-d} \|\nabla^2 u_n(y)\|dy\right)^2\right].
\end{align}
Thus, by Lemma~\ref{lem:regularity_un},
\begin{align}
\label{eq:pf_V3_step_part1}
\bbE[\calI_1^2 ]
 &\lesssim   h_n^{4-d-2\beta} +  \bbE\left[\left(\int_{\calQ_0 \setminus B(0,C_2 h_n)} \|x-X_0-y\|_{\bbT^d}^{2-d} \|y\|_{\bbT^d}^{-d}
  dy\right)^2\right].
\end{align}
Let $A_n$ be the event $\|X_0 - x\|_{\bbT^d} \leq h_n/2$. 
Notice that for all $y \in \bbT^d$ such that $\|y\|_{\bbT^d} > h_n$,
it holds
\begin{equation}
\label{eq:pf_var_bound_repr}
\|x-X_0-y\|_{\bbT^d} \geq  \|y\|_{\bbT^d} -\|x-X_0\|_{\bbT^d} \geq \|y\|_{\bbT^d}/2, 
\end{equation}
over the event $A_n$.
Therefore, we have, 
\begin{align*}
\bbE&\left[\left(\int_{\calQ_0\setminus B(C_2 h_n)} \|x-X_0-y\|_{\bbT^d}^{2-d} \|y\|_{\bbT^d}^{-d}
 dy\right)^2 I( A_n)\right]  \\
 &\qquad \lesssim \bbE\left[\left(\int_{\calQ_0\setminus B(0,C_2 h_n)} \|y\|^{2-2d}
  dy\right)^2 I( A_n)\right]
  \lesssim h_n^{d}  \left(\int_{h_n}^{\sqrt d} r^{1-d}   dr\right)^2  
 \lesssim  h_n^{4-d}
\end{align*}
On the other hand, by Lemma~\ref{lem:miranda}, 
\begin{align*}
\bbE&\left[\left(\int_{\calQ_0\setminus B(0,C_2h_n)} \|x-X_0-y\|_{\bbT^d}^{2-d} \|y\|_{\bbT^d}^{-d} dy\right)^2 I( A_n^c)\right]  \\ 
 &\lesssim h_n^{-2\beta} \bbE\left[\left(\int_{\bbT^d} \|x-X_0-y\|_{\bbT^d}^{2-d} \|y\|_{\bbT^d}^{\beta-d} dy\right)^2 I( A_n^c)\right]  \\ 
 &\lesssim h_n^{-2\beta} \bbE\left[\left(\|x-X_0\|^{2+\beta-d}\right)^2 I( A_n^c)\right] \\
 &=  h_n^{-2\beta} \int_{\calQ_x\setminus B(x, h_n/2)} \|x-z\|^{4+2\beta-2d}dQ(z) 
 \asymp h_n^{4-d-2\beta}\vee 1.
  \end{align*}
Returning to equation~\eqref{eq:pf_V3_step_part1}, we have thus shown
\begin{align*}
\bbE[\calI_1^2]
\lesssim  
 h_n^{4-d -2\beta}\vee 1. 
\end{align*}
One can bound $\bbE[\calI_2^2]$ using a similar argument, as shown in the following Lemma. 
\begin{lemma}
\label{lem:bounding_term_I2}
It holds that $\bbE [\calI_2^2] \lesssim h_n^{4-d-2\beta}$. 
\end{lemma} 
The proof of Lemma~\ref{lem:bounding_term_I2} is deferred to Appendix~\ref{app:pf_lem_bounding_term_I2}. 
Combining these facts, we have shown that $\calV_3 \lesssim h_n^{4-d-2\beta}$, and the claim follows.\qed

\subsection{Proof of Lemma~\ref{lem:bias_bound_map}}
\label{sec:pf_bias_bound_map}
Let $\beta$ be chosen as in equation~\eqref{eq:beta}, and assume $\beta<\epsilon/2$. 
By Lemma~\ref{lem:inverse_L_to_E}, 
\begin{align*}
\big\|\nabla L^{-1}[q_{h_n}-q]\big\|_{L^\infty(\bbT^d)} 
= \big\|\nabla E^{-1} g_n\big\|_{L^\infty(\bbT^d)},
\end{align*}
where  $g_n = \det(\nabla^2\varphi_0) (q_{h_n}-q)\circ(\nabla\varphi_0)$ 
has mean zero over $\bbT^d$. Let $u_n$ be the unique mean-zero solution 
to the Poisson equation $-\Delta u_n= g_n$ over $\bbT^d$. 
By Lemmas~\ref{lem:a_priori_gradient}--\ref{lem:a_priori_moser}, we  have
$$\big\|\nabla E^{-1} g_n \big\|_{L^\infty(\bbT^d)}
 \lesssim \|E^{-1} g_n \|_{L^2(\bbT^d)} + \|\nabla u_n\|_{\calC^\beta(\bbT^d)}.$$
 Now, using the fact that $E^{-1}$ is self-adjoint, we have 
\begin{align*}
\|E^{-1} g_n \|_{L^2(\bbT^d)}^2
 &= \langle E^{-1} g_n, E^{-1} g_n\rangle_{L^2(\bbT^d)} \\
 &\leq  \langle E^{-2} g_n,  g_n\rangle_{L^2(\bbT^d)} \\
 &\leq  \|E^{-2} g_n \|_{H^2(\bbT^d)}\| g_n\|_{H^{-2}(\bbT^d)} \asymp \|E^{-1} g_n \|_{L^2(\bbT^d)}\| g_n\|_{H^{-2}(\bbT^d)},
 \end{align*}
where the final bound follows from the norm equivalence in equation~\eqref{eq:norm_equivalence_E}. Deduce that
$\|E^{-1} g_n \|_{L^2(\bbT^d)}\lesssim \| g_n\|_{H^{-2}(\bbT^d)}$. Furthermore, 
using a 
Sobolev embedding for periodic spaces (e.g. Corollary 3.5.5 of \cite{schmeisser1987}),  
we have for any fixed $r > 2d/\beta$,
$$\|\nabla u_n\|_{\calC^\beta(\bbT^d)} 
 \lesssim \|u_n\|_{H^{1+2\beta,r}(\bbT^d)} = \|g_n\|_{H^{2\beta-1,r}(\bbT^d)},$$
thus we arrive at
$$\big\|\nabla E^{-1} g_n \big\|_{L^\infty(\bbT^d)}
 \lesssim   \|g_n\|_{H^{2\beta-1,r}(\bbT^d)}.$$
To bound the right-hand side, let $r'$ be the H\"older conjugate exponent of $r$.
By Lemma~\ref{lem:sobolev_dual_norm}, 
\begin{align*}
\|g_n\|_{H^{2\beta-1,r}(\bbT^d)}
 &\asymp \sup_{\substack{v \in H^{1-2\beta, r'}(\bbT^d) \\ \|v\|_{H^{1-2\beta, r'}(\bbT^d)} = 1}} \int v \det(\nabla^2\varphi_0) (q_{h_n}-q)\circ(\nabla\varphi_0) \\ 
 &= \sup_{\substack{v \in H^{1-2\beta, r'}(\bbT^d) \\ \|v\|_{H^{1-2\beta, r'}(\bbT^d)} = 1}} \int v(\nabla\varphi_0^*)(q_{h_n}-q) 
 \lesssim \sup_{\substack{w \in H^{1-2\beta, r'}(\bbT^d) \\ \|w\|_{H^{1-2\beta, r'}(\bbT^d)} = 1}} \int w\cdot ( q_{h_n}-q),
 \end{align*}
 where we used the fact that the map $w = v\circ \nabla\varphi_0^*$ satisfies 
 $\|w\|_{H^{1-2\beta,r'}(\bbT^d)} \lesssim \|v\|_{H^{1-2\beta,r'}(\bbT^d)}$ under the regularity
 conditions we have placed on $\varphi_0$, by Lemma~\ref{lem:sobolev_comp}.
It follows that
$$\|g_n\|_{H^{2\beta-1,r}(\bbT^d)} \lesssim \|q_{h_n}-q\|_{H^{2\beta-1,r}(\bbT^d)}\lesssim h_n^{s+1-2\beta},$$
where we used Proposition~\ref{prop:kde_negative_sobolev} and the assumed properties of $K$. 
The claim follows.\qed

\section{Proofs of Main Results}
\label{sec:pf_main_results}

In the following  subsections, we respectively prove  
Proposition~\ref{prop:pointwise_rate_map}, Theorem~\ref{thm:clt_map},  
  and Theorem~\ref{thm:uniform_impossible}.  
The key ingredients are  the linearization bound of Theorem~\ref{thm:linearization}, 
the variance bound of Lemma~\ref{lem:variance_bound_map}, and
the bias bounds of Lemma~\ref{lem:bias_bound_map}.

Let $b \geq 1$ be a constant whose value will be determined below. 
Let $c = c(\omega_s(p,q),K,b,\beta,s,d) > 0$ be a constant whose value is permitted to vary from line to line. 
By Lemmas~\ref{lem:kde_as}--\ref{lem:fitted_caffarelli},
and by the assumption  $p,q \in \calC^s_+(\bbT^d)$, together with the fact that $s > 2$, 
it holds  for any
choice of the free parameter $\beta$ satisfying
equation~\eqref{eq:beta} that, 
with probability at least $1 - (c/n^b)$,
$$\|\hat q_n\|_{\calC^{2+\beta}(\bbT^d)} \leq c, \quad \text{and}\quad
\hat q_n \geq 1/c,~~ \text{over } \bbT^d,$$
and hence, by Theorem~\ref{thm:caffarelli}, 
$$\|\hat \varphi_n\|_{\calC^{2+\beta}(\bbT^d)}\leq c, \quad \text{and}\quad\nabla^2 \hat\varphi_n \succeq I_d/c, ~~\text{over } \bbT^d.$$
Over the above high-probability event,  $\hat T_n$ is the unique continuous optimal transport
map pushing forward $P$ onto $\hat Q_n$. We may therefore apply Theorem~\ref{thm:linearization} to the fitted
potential $\hat\varphi_n$ over this event: With probability at least $1 - (c/n^b)$,
\begin{align*}
\left\| (\hat\varphi_n - \varphi_0) - L^{-1}[ \hat q_n - q ]\right\|_{\calC^{2+\beta}(\bbT^d)}  
 &\leq c \|\hat q_n-q\|_{\calC^{\beta}(\bbT^d)}\|\hat q_n - q\|_{\calC^{1+\beta}(\bbT^d)}.
\end{align*} 
Proposition~\ref{prop:kde_holder} implies
that with probability at least $1-(c/n^{b})$,
$$\|\hat q_n-q\|_{\calC^{1+\beta}(\bbT^d)} \lesssim h_n^{s-1-\beta} + \sqrt{\frac {\log n} {nh_n^{d+2+2\beta}}}
\lesssim h_n^{s-1-\beta} + \sqrt{\frac {1} {nh_n^{d+2+4\beta}}},$$
and
$$\|\hat q_n-q\|_{\calC^{\beta}(\bbT^d)} \lesssim h_n^{s-\beta} + \sqrt{\frac {1} {nh_n^{d+4\beta}}}.$$
Combining these facts, we deduce that there exists an event
$A_n$ with probability content at least $1-(c/n^{b})$ such that, over $A_n$,  
\begin{align*}
&\left\| (\hat\varphi_n - \varphi_0) - L^{-1}\big[ \hat q_n - q \big]\right\|_{\calC^{2+\beta}(\bbT^d)}  \\  
 &~~~~ {\lesssim} \left(h_n^{s -2\beta} {+} \sqrt{\frac{1}{nh_n^{d +4\beta}}}\right)
            \left(h_n^{s-1-\beta} {+} \sqrt{\frac {1} {nh_n^{d+2+4\beta}}}\right)    
 \asymp h_n^{2s-1-3\beta} + \frac 1 {nh_n^{d+1+4\beta}}  =:\delta_n.
\end{align*}
Since the operator $L^{-1}$ is linear, we may write 
\begin{align*}
L^{-1}\left[\hat q_n - q_{h_n} \right](x) 
 =  \frac 1 n \sum_{i=1}^n Z_{n,i}(x),\quad\text{where } ~~Z_{n,i}(x) 
 := L^{-1}\Big[\widebar K_{h_n}(X_i-\cdot) - q_{h_n} \Big](x).
\end{align*}
We thus have, over the event $A_n$, 
\begin{align}
\sup_{x\in \bbT^d} \left\|\nabla L^{-1}[q_{h_n}-q](x) + \frac 1 n \sum_{i=1}^n \nabla Z_{n,i}(x)
  - (\hat T_n- T_0)(x) \right\| \lesssim \delta_n.
\label{eq:linearization_implication}
\end{align}
With this bound in place, we now turn to the proofs of the various results, in turn. 

\subsection{Proof of Proposition~\ref{prop:pointwise_rate_map}} 
By assumption, we have $h_n \asymp n^{-a}$ with $  (d+4(s-1))^{-1} < a < (d+s+2 )^{-1}$.
Let $x \in \bbT^d$. To bound the bias of $\hat T_n(x)$, use the decomposition
\begin{align*}
\big\|\bbE[\hat T_n(x) - T_0(x)]\big\|
 &\leq \big\| \bbE[(\hat T_n(x) - T_0(x))I_{A_n^c}]\big\| + 
 \big\| \bbE[(\hat T_n(x) - T_0(x))I_{A_n}]\big\|.
 \end{align*}
 For the first term, recall that the definition of $\hat T_n$ implies that $\|\hat T_n(x)-T_0(x)\|$ is 
 uniformly bounded by a constant independent of $x$, and thus  
 $$\big\| \bbE[(\hat T_n(x) - T_0(x))I_{A_n^c}]\big\| \lesssim \bbP(A_n^\cp) \lesssim n^{-b}.$$
For the second term, it follows from equation~\eqref{eq:linearization_implication} that
\begin{align*} 
\big\| \bbE[(\hat T_n(x) - T_0(x))I_{A_n}]\big\|
\leq \left\| \bbE\left[\left(\nabla L^{-1}[q_{h_n} - q](x) +
   (1/n) \textstyle\sum_{i=1}^n \nabla Z_{n,i}(x)\right)I_{A_n}\right]\right\|
  + \delta_n. 
\end{align*}
Now, we make use of the following.
\begin{lemma}\label{lem:expectation_zero}
It holds that
$$\sup_{x \in \bbT^d}\left\|\bbE [ \left((1/n) \textstyle\sum_{i=1}^n \nabla Z_{n,i}(x)\right)I_{A_n}] \right\| \lesssim n^{a(d+\beta) -b}.$$
 \end{lemma}
The proof of Lemma~\ref{lem:expectation_zero} appears in Appendix~\ref{app:pf_lem_expectation_zero}.  
Altogether, we arrive at 
\begin{align*}
\big\|\bbE[\hat T_n(x) - T_0(x)]\big\|
  &\leq \big\|\nabla L^{-1}[q_{h_n} - q](x)\big\|+\delta_n + n^{a(d+\beta) -b} \\ 
  &\lesssim_\epsilon h_n^{s+1-\epsilon} + \frac 1 {nh_n^{d+1+4\beta}} + h_n^{2s-1-3\beta} + n^{a(d+\beta) -b},
\end{align*}
using Lemma~\ref{lem:bias_bound_map}.
Since $h_n\asymp n^{-a}$, we may choose $b$ large enough to make the final term of low order.
Furthermore, the second term is of lower order than the first if $\beta$ is chosen sufficiently small
in terms of $\epsilon$, due to the condition that $a < (d+s+2)^{-1}$. Likewise, after possibly
decreasing $\beta$ further, the third term is of lower order than the first
due to the condition $s > 2$. We thus have
\begin{align*}
\big\|\bbE[\hat T_n(x) - T_0(x)]\big\| &\lesssim_\epsilon h_n^{s+1-\epsilon},
\end{align*}
thus proving the desired bias bound. Reasoning similarly, we have the following
bound on the fluctuations,
\begin{align*}
\big\|&\Cov\big[ \hat T_n(x)\big]\big\|\\ 
 &\lesssim \left\|\Cov\left[\frac 1 n \sum_{i=1}^n \nabla  Z_{n,i}(x)\right]\right\| + 
           \left\|\Cov\left[\hat T_n(x) - \frac 1 n \sum_{i=1}^n \nabla Z_{n,i}(x)\right]\right\| \\
 &= \frac 1 n \left\|\Cov\left[\nabla Z_{n,1}(x)\right] \right\| +
  \left\|\Cov\left[ \hat T_n(x) - T_0(x) - \frac 1 n \sum_{i=1}^n \nabla Z_{n,i}(x)- \nabla L^{-1}[q_{h_n} - q](x)\right]\right\| \\
 &\lesssim\frac 1 {nh_n^{d-2}} + \left(\frac 1 {nh_n^{d+1+4\beta}} + h_n^{2s-1-3\beta}\right)^2  + n^{-b} 
 \end{align*}
where we used Lemma~\ref{lem:variance_bound_map} and equation~\eqref{eq:linearization_implication}.
Under the conditions $(d+4(s-1))^{-1}<a < (d+s+2)^{-1}$, the first term dominates, and the claim follows.\qed 

\subsection{Proof of Theorem~\ref{thm:clt_map}}
From equation~\eqref{eq:linearization_implication} and Lemma~\ref{lem:bias_bound_map},  
we obtain that for any fixed $x \in \bbT^d$ and $\epsilon > 0$,  
$$\hat T_n(x) - T_0(x) = \frac 1 n \sum_{i=1}^n \nabla Z_{n,i}(x)
  + O_p(h_n^{s+1-\epsilon}+\delta_n),$$
  where the symbols $O_p$ and $o_p$ are to be understood coordinate-wise. 
Let $\Sigma_n(x) = \Cov[\nabla Z_{n,1}(x)]/n$. We have 
$\|\Sigma_n(x)\| \asymp h_n^{2-d}/n$ 
by Lemma~\ref{lem:variance_bound_map}, 
thus
\begin{align*}
\Sigma_n^{-1/2}(x) &\big( \hat T_n(x) - T_0(x)\big) \\
 &= \frac {\Sigma_n^{-1/2}(x)} {n} \sum_{i=1}^n \nabla Z_{n,i}(x)  + O_p\left(\sqrt {n h_n^{ d- 2}} \left(h_n^{s+1-\epsilon}+h_n^{2s-1-3\beta} + \frac 1 {nh_n^{d+1+4\beta}}\right)\right).
 \end{align*}
Choosing $\epsilon$ and $\beta$  sufficiently small in terms of $a$, 
we deduce from condition~\eqref{assm:bandwidth_map_clt}
 that 
\begin{equation}
\label{eq:pf_clt_potential_step}
\Sigma_n^{-1/2}(x)  \big( \hat T_n(x) - T_0(x)\big) 
= \frac {\Sigma_n^{-1/2}(x) } {n}\sum_{i=1}^n \nabla Z_{n,i}(x)  + 
o_p(1).
\end{equation}
We will show that the sample average appearing on the right-hand side of the above display is asymptotically 
Gaussian, by appealing to Lyapunov's central limit theorem (Lemma~\ref{lem:lyapunov}). 
We will make use of the following Lemma. 
\begin{lemma}
\label{lem:lyapunov_technical_map}
For any $\delta > 0$, it holds that
$$\bbE \big\|\nabla Z_{n,1} (x)\big\|^{2+\delta} \lesssim h_n^{-\delta(d+\beta)+2-d}.$$
\end{lemma}
The proof of Lemma~\ref{lem:lyapunov_technical_map} is simple, and is deferred to Appendix~\ref{app:pf_lem_lyapunov_technical_map}.  
From Lemmas~\ref{lem:variance_bound_map} and~\ref{lem:lyapunov_technical_map}, we have
\begin{align*}
\frac{\sum_{i=1}^n \bbE\|\nabla Z_{n,i}(x)\|^{2+\delta}}{\left( \lmin [\sum_{i=1}^n \nabla Z_{n,i}(x)]\right)^{\frac{2+\delta}{2}}}
 &\lesssim \frac{n h_n^{-\delta(d+\beta)+2-d}}{n^{\frac\delta 2 + 1} h_n^{-(d-2)(2+\delta)/2}(1+o(1))}
  \lesssim \left( n h_n^{d+2+2\beta}\right)^{-\frac\delta 2}. 
\end{align*}
Since $\beta$ can be taken to be an arbitrarily small constant, 
the right-hand side of the above display vanishes under condition~\eqref{assm:bandwidth_map_clt}
on $h_n$. Lyapunov's condition is thus satisfied,  
so we deduce from Lemma~\ref{lem:lyapunov}  that 
$$\frac{\Sigma_n^{-1/2}(x)}{n}\sum_{i=1}^n \nabla Z_{n,i}(x) \overset{w}{\longrightarrow} N(0,I_d).$$
By Lemma~\ref{lem:variance_bound_map}, we have 
$\|nh_n^{d-2}\Sigma_n(x) - \Sigma(x)\| \to 0$, 
and since $\Sigma(x)$ is positive definite, it follows that 
$\|(nh_n^{d-2}\Sigma_n(x))^{-1/2} - \Sigma^{-1/2}(x)\| \to 0$.
The claim follows from here.\qed

\subsection{Proof of Theorem~\ref{thm:uniform_impossible}}
\label{sec:pf_uniform_asymptotics}
The proof is inspired by~\cite{nishiyama2011} and \cite{stupfler2014}. 
We make use of the following two propositions,  
which may be of independent interest. 
The first provides upper bounds on  projections of $\hat\varphi_n-\varphi_0$
under the inner product $\langle\cdot,\cdot\rangle_A$ defined in Section~\ref{sec:linearization}. 
\begin{proposition}
\label{prop:clt_projections}
Suppose $p,q \in \calC^s_+(\bbT^d)$. Assume that $K$ satisfies condition~\Kernel{$s+1$}. 
Then, for any given $\epsilon > 0$ and   $v \in H_0^1(\bbT^d)$, 
$$\langle \hat\varphi_n - \varphi_0, v\rangle_A =  O_p\left( n^{-1/2} + h_n^{s+1}+ 
\frac 1 {nh_n^{d+\epsilon}} \right).$$
\end{proposition}
The proof appears in Appendix~\ref{app:pf_bandwidth_projection_clt}.
Next, we state a bound on the $L^2(\bbT^d)$ convergence rate of $\hat T_n$,
which is essentially already contained in~\cite{manole2021}.
Let $\widebar T_{h_n}$ be the unique continuous optimal transport
map pushing $P$ forward onto the probability law $Q_{h_n}$ with density $q_{h_n} = \bbE[\hat q_n(\cdot)]$. 
 \begin{proposition}
\label{prop:L2_rate_map}
 Under the same conditions as Proposition~\ref{prop:pointwise_rate_map},  
it holds that
\begin{align}
\label{eq:L2_rate_opt_map}
\bbE \|\hat T_n -  \widebar T_{h_n}\|_{L^2(\bbT^d)}^2 \asymp \frac{ h_n^{2-d}}{n}, ~~ \text{and}~~ 
\big\| \widebar  T_{h_n} - T_0 \big\|_{L^2(\bbT^d)} &\lesssim h_n^{s+1},
\end{align}
and for any $r > 2$, 
$$\bbE \|\hat T_n - \widebar T_{h_n}\|_{L^2(\bbT^d)}^r \lesssim n^{-\frac r 2} h_n^{r\left(1-\frac d 2\right)}.$$
\end{proposition}
The proof appears in Appendix~\ref{app:pf_L2_rate_map}.
With these two propositions in place, we turn to the proof. 
We will begin by proving the following. 
\begin{lemma}\label{lem:weak_limit_rootn}
Suppose there exists $\epsilon > 0$ such that 
$$\alpha_n\left(\frac 1 {\sqrt n} + {h_n^{s+1}} + \frac 1 {nh_n^{d+1+\epsilon}}\right)=o(1),$$
and that $\bbG_{n}$ converges weakly in $H_0^1(\bbT^d)$
to a tight random element $\bbG$ in $H_0^1(\bbT^d)$. Then, $\bbG = 0$. 
\end{lemma} 
To prove Lemma~\ref{lem:weak_limit_rootn}, let us recall some standard facts about the spectral
properties of~$E$. Using the norm equivalence~\eqref{eq:norm_equivalence_E}, and the
fact that $H_0^1(\bbT^d)$ is compactly embedded in $L_0^2(\bbT^d)$ by the
Rellich-Kondrachov theorem, it can be seen that $E^{-1}$ is a compact operator when viewed
as a map from $L_0^2(\bbT^d)$ into itself. Furthermore, we have already noted that $E^{-1}$ is self-adjoint
and positive definite,
so the spectral theorem for compact and self-adjoint operators implies that 
$E^{-1}$ admits a discrete spectrum, which in turn also implies that $E$ has a discrete spectrum
$$0 < \lambda_1 \leq \lambda_2 \leq \dots$$
with a corresponding   sequence of eigenfunctions $\widebar \eta_1,\widebar \eta_2, \dots \in L_0^2(\bbT^d)$
satisfying $E\widebar \eta_j=\lambda_j\widebar\eta_j$ for all $j \geq 1$, and forming
an orthonormal basis of $L_0^2(\bbT^d)$. Furthermore, 
it follows from Lemma~\ref{lem:a_priori_schauder} that $\widebar \eta_j \in \calC^{2}_0(\bbT^d)$ for all $j \geq 1$. 

Now, define
$$\eta_j = \widebar \eta_j / \|\widebar \eta_j\|_A, \quad j\geq 1.$$
We claim that $\{\eta_j\}_{j\geq 1}$ forms an orthonormal basis of $H_0^1(\bbT^d)$, 
when the latter is endowed with the inner product $\langle\cdot,\cdot\rangle_A$. 
This system is clearly dense in $H_0^1(\bbT^d)$, since $\{\widebar \eta_j\}$ is dense in $L_0^2(\bbT^d)$.
It is also clearly normalized.
Furthermore, note that for all $j\neq k$,  
\begin{align*}
\langle \eta_j,\eta_k\rangle_A = \langle E \eta_j ,\eta_k\rangle_{L^2(\bbT^d)} = 
\lambda_j \left\langle \eta_j,\eta_k\right\rangle_{L^2(\bbT^d)} = 0.
\end{align*}
Thus, $\{\eta_j\}$ is indeed an orthonormal basis of $H_0^1(\bbT^d)$.
 
Since the elements of this basis belong to $H_0^1(\bbT^d)$, 
Proposition~\ref{prop:clt_projections} implies that for any fixed $j\geq 1$, we have 
$$\langle \bbG_n,\eta_j\rangle_A =
O_p\left(
\alpha_n\left(\frac 1 {\sqrt n} + h_n^{s+1} + \frac 1 {nh_n^{d+1+\epsilon}}\right)\right)
=o_p(1),$$
where the implicit constants depend, in particular, on $\epsilon,j$. 
On the other hand, by the weak convergence of $\bbG_n$ to $\bbG$ in the Hilbert space $H_0^1(\bbT^d)$, 
endowed with the inner product $\langle\cdot,\cdot\rangle_A$,
we have 
$$\langle \bbG_n, \eta_j\rangle_A \overset{w}{\longrightarrow} \langle \bbG, \eta_j\rangle_A,\quad
j=1,2,\dots$$
(see~\cite{vandervaart1996}, Theorem 1.8.4.).
The preceding two displays imply that $\langle \bbG, \eta_j\rangle_A = 0$, for any $j\geq 1$. 
 Since $\{\eta_j\}_{j\geq 1}$ forms a basis of $H_0^1(\bbT^d)$, it readily follows that $\bbG = 0$, thus proving the Lemma. 
 
Let us now show why the Lemma implies the claim. 
On the one hand, if $\alpha_n = o(\sqrt{nh_n^{d-2}})$, then we have  by Proposition~\ref{prop:L2_rate_map},
\begin{align*}
\bbE\|\bbG_n\|_{H^1(\bbT^d)} 
 &\asymp \alpha_n \bbE \|\hat T_n - T_0\|_{L^2(\bbT^d)} = o\left(\left(\sqrt{nh_n^{d-2}}\right) \left(\frac 1 {\sqrt{nh_n^{d-2}}} + h_n^{s+1}\right)\right).
\end{align*}
Under condition~\eqref{assm:bandwidth_map_clt} on $h_n$, we have $h_n = o(h_n^*)$ with $h_n^* =n^{-1/(d+2s)}$. 
We deduce that $\bbE\|\bbG_n\|_{H^1(\bbT^d)} = o(1)$,
thus $\bbG_n$ must converge weakly to 0 in $H_0^1(\bbT^d)$. 

On the other hand, if $\alpha_n \asymp  \sqrt{nh_n^{d-2}}$, then we have by Proposition~\ref{prop:L2_rate_map}, 
\begin{align}\label{eq:Gn_bded_from1}
\bbE \|\bbG_n\|_{H^1(\bbT^d)}^2
 &\gtrsim  \frac { \alpha_n^2} { nh_n^{d-2}}\gtrsim 1,
\end{align}
where we used again 
the fact that $h_n=o(h_n^*)$. 
Now, suppose by way of a contradiction that $\bbG_n$ converges
weakly to a tight random element $\bbG$ taking values in $H_0^1(\bbT^d)$.
Under condition~\eqref{assm:bandwidth_map_clt} on the bandwidth $h_n$, 
there exists $\epsilon > 0$ such that the sequence $\alpha_n \asymp \sqrt{nh_n^{d-2}}$
fulfills the assumption of Lemma~\ref{lem:weak_limit_rootn}, thus $\bbG$ must be zero. 
Now, for this choice of sequence $\alpha_n$, it follows from Proposition~\ref{prop:L2_rate_map} that
$\bbE\|\bbG_n\|_{H^1(\bbT^d)}^3$ is uniformly bounded, and thus that
the collection
$\{\|\bbG_n\|_{H^1(\bbT^d)}^2: n \geq 1\}$
is uniformly integrable. Since $\bbG_n$ converges weakly to zero, we must then
have \citep[Theorem 1.11.3]{vandervaart2023} 
that $\bbE\|\bbG_n\|_{H^1(\bbT^d)}^2 \to 0$, which contradicts 
equation~\eqref{eq:Gn_bded_from1}.
This proves that for $\alpha_n \asymp \sqrt{nh_n^{d-2}}$, the process $\bbG_n$ does not converge weakly. 

Finally, suppose $\alpha_n / \sqrt{nh_n^{d-2}} \to \infty$. If $\bbG_n$ were to converge weakly in $H_0^1(\bbT^d)$, 
then it is clear that the process $\sqrt{nh_n^{d-2}} \bbG_n$ would have to converge weakly to zero, which
contradicts what we have already shown. The claim follows.
\qed

\section{Discussion: Toward the Two-Dimensional Case}
\label{sec:discussion}
Our aim in this manuscript has been to derive pointwise central limit theorems 
for the optimal transport map pushing forward a fixed density onto a kernel density estimator. 
There are many natural questions left open by our results. We discuss  here 
one avenue of future work which we find particularly exciting: 
the case where 
the underlying measures are two-dimensional.

%
Although we assumed that $d \geq 3$ throughout our development, 
this condition was only used to derive the limiting covariance matrix
of the process $\sqrt{nh_n^{d-2}}\hat T_n(x)$ in Lemma~\ref{lem:variance_bound_map}. 
In particular, our linearization of the Monge-Amp\`ere equation continues to hold when $d < 3$, 
and it is therefore natural to expect that $\hat T_n$ continues
to satisfy a pointwise Gaussian limiting distribution in this case.
For one-dimensional measures, this fact can already be deduced 
from the work of~\cite{ponnoprat2023}, however the most  
challenging case $d=2$ remains open. 
The following is a first result in this direction, for
 the special case of uniform measures on the two-dimensional torus.
\begin{proposition} 
\label{prop:clt_map_2d} 
Assume that $P$ and $Q$ are equal to the uniform measure $\calL$ over $\bbT^2$.
Let $K \in \calC_c^\infty(\bbR^2)$
be a symmetric kernel that integrates to unity, and whose first
moment vanishes.  Assume further that
$h_n\downarrow 0$ is a decreasing sequence of nonnegative real numbers such that 
\begin{equation}
\label{assm:bandwidth_map_clt_2d}
h_n \gg n^{-\frac 1 6}   .
\end{equation}
Then, for any fixed $x \in \bbT^2$, it holds that
$$\sqrt{\frac{n}{\log(1/h_n)}} \big( \hat T_n(x) - x\big) \overset{w}{\longrightarrow} 
 N\left(0, \frac {I_2} {4\pi} \right),\quad \text{as } n \to \infty.$$
\end{proposition}

\begin{proof}[Proof Sketch]
Since $Q$ is uniform, it holds that $q_{h_n}(x) = \bbE[\hat q_n(x)] = q(x)$
for all $x \in \bbT^d$. Furthermore, we will prove that
for any $x \in \bbT^d$, 
$$\frac {\Cov_Q\big\{ \nabla L^{-1}[\widebar K_{h_n}(Y-\cdot)-1](x)\big\} } {\log(1/h_n)} \longrightarrow \frac {I_2}{4\pi},
\quad \text{as } n \to \infty.$$
From these two facts, and condition~\eqref{assm:bandwidth_map_clt_2d}, the claim can be deduced by following along similar
lines as the proof of Theorem~\ref{thm:clt_map}.

Notice that when $P=Q=\calL$, the operator $L$ is simply the negative Laplacian, and, 
using the fact that $K$ is smooth and compactly-supported, we have the closed form
$$\nabla L^{-1}[\widebar K_{h_n}(Y-\cdot)-1](x)
 = \sum_{\xi\in \bbZ_*^d} \frac{2\pi i \xi}{\|2\pi \xi\|^2} \calF[K](h_n\xi) e^{2\pi i \xi^\top (Y-x)},$$
so that, similarly as in the proof of Lemma~\ref{lem:variance_bound_map}, one has 
$$\Cov\big\{\nabla L^{-1}\big[\widebar K_{h_n}(Y-\cdot)-1\big](x)\big\}  =
\sum_{\xi\in \bbZ_*^d} \xi\xi^\top \left(\frac{\calF[K](h_n\xi)}{2\pi\|\xi\|^2}\right)^2 =:  A_n + B_n + C_n,$$
where,   
\begin{align*}
A_n   =\sum_{\substack{\xi \in \bbZ_*^d \\ \| h_n\xi\| < 1}}  \frac{\xi\xi^\top}{ 4\pi^2 \|\xi\|^4},
\end{align*}
\begin{align*}
B_n   = \sum_{\substack{\xi \in \bbZ_*^d \\ \|h_n\xi\| < 1}} \xi\xi^\top \frac{ |\calF[K](h_n\xi)|^2-1}{ 4\pi^2\|\xi\|^4}  ,\quad 
C_n   &=  \sum_{\substack{\xi \in \bbZ_*^d \\ \|h_n \xi\| \geq  1}} \xi\xi^\top \frac{|\calF[K](h_n\xi)|^2}{4\pi^2\|\xi\|^4}.
\end{align*}
 
On the one hand, using the fact that $K$ integrates to unity, has first moment equal to zero,
 and has Fourier transform lying in the Schwartz space, 
one has 
$$\big| \calF[K](h_n\xi) - 1\big| = \big| \calF[K](h_n\xi) - \calF[K](0) - (h_n\xi)^\top \nabla \calF[K](0) \big| \lesssim \|h_n\xi\|^2,$$
for all $\xi\in \bbZ^d$, from which it follows that
\begin{align*}
\|B_n\| \lesssim \sum_{\|h_n\xi\| < 1}  \frac{\big|\calF[K](h_n\xi)-1\big|}{\|\xi\|^2} 
\lesssim \sum_{\|h_n\xi\| < 1}  \frac{\|h_n\xi\|^2}{\|\xi\|^2} 
 = O(1),
 \end{align*}
 and, for any $\beta > 0$, 
\begin{align*}
\|C_n\| \lesssim 
\sum_{\|h_n\xi\| \geq  1}   \frac{|\calF[K](h_n\xi)|^2}{\|\xi\|^2} 
 &\lesssim h_n^{-2\beta} \sum_{\|h_n\xi\| \geq  1} \|\xi\|^{-2(\beta +1)} =O(1).
 \end{align*}   
%
On the other hand, by a simple Riemann summation argument (cf. Lemma~\ref{app:pf_lem_riemann_sum_approx_d2}
 in Appendix~\ref{app:additional_discussion} for a formal statement), 
it holds that for any fixed $c > 0$, 
\begin{align*}
A_n 
 = \int_{\xi:c \leq   \|\xi\| < 1/h_n} 
 \frac{\xi\xi^\top d\xi }{4\pi^2 \|\xi\|^4} + O(1) 
 = \frac 1 {4\pi} \log(1/h_n) I_2 + O(1).
\end{align*}
The claim follows from here.
\end{proof} 
It is fruitful to compare Proposition~\ref{prop:clt_map_2d} to Theorem~1 of~\cite{ambrosio2019b},
which asserts that
\begin{equation}
\label{eq:matching_asymp}
\lim_{n\to\infty} \frac{n}{\log n} \bbE W_2^2(Q_n,Q) = \frac 1 {4\pi},
\end{equation}
where $W_2$ denotes the 2-Wasserstein distance.
In contrast, Proposition~\ref{prop:clt_map_2d} heuristically suggests the approximation
$$\frac{n}{\log(1/h_n)} \bbE W_2^2(\hat Q_n,Q) \approx  
   \int \tr(I_2 / 4\pi) dP = \frac 1 {2\pi}. 
$$
If one were permitted to choose the bandwidth $h_n \asymp n^{-1/2}$, 
which is the natural scaling for which the kernel density estimator
becomes statistically indistinguishable from the empirical measure, then the above display would precisely reduce to
$$\frac{n}{\log n} \bbE W_2^2(\hat Q_n,Q) \approx  
   \frac 1 {4\pi}. 
$$
In this sense, the limit law of Proposition~\ref{prop:clt_map_2d}
is consistent with the results of~\cite{ambrosio2019}, though we emphasize that our result does not
rigorously allow for such small values of $h_n$. Indeed, for such values, the density $\hat q_n$ 
does not quantitatively lie in any H\"older space, and the map $\hat T_n$ may fail to be Lipschitz, in which case our arguments do not hold. It is nevertheless natural to ask whether 
condition~\eqref{assm:bandwidth_map_clt_2d} can be weakened to include the scaling $h_n\asymp n^{-1/2}$, 
and if so, whether a pointwise limit law can in fact be derived for the empirical optimal transport
map $T_n$ pushing forward $P$ onto $Q_n$.

Let us also emphasize that, in striking contrast to the case $d \geq 3$, the 
limiting distribution 
of Proposition~\ref{prop:clt_map_2d} does not depend on the choice of kernel $K$. 
A similar discrepancy arose in the predictions of~\cite{caracciolo2014} for the limiting
value of a quadratic optimal matching problem~\eqref{eq:matching_asymp}, later formalized by~\cite{ambrosio2019}.  
Indeed, these authors used a regularization argument to obtain the limit~\eqref{eq:matching_asymp}, 
which reduces precisely to kernel density estimation when working over the torus, 
but they argued that the error incurred by their regularization
would be of leading order when $d \geq 3$. 
This matches our observation that the limiting distribution of $\hat T_n(x)$
depends on $K$ only when $d \geq 3$.

Another natural question is to extend Proposition~\ref{prop:clt_map_2d} to generic measures 
$P$ and $Q$. 
Inspired by our results for the case $d \geq 3$,
one might try to use a constant-coefficient approximation---as described above equation~\eqref{eq:coeff_freezing_pde} and in the proof of Lemma~\ref{lem:variance_bound_map}---to reduce the calculation
of a general limit law to that of the uniform case. 
Unfortunately, we have found that the error incurred
by this constant-coefficient approximation is of leading order when the dimension is two. This suggests
that, unlike the case $d \geq 3$,  the two-dimensional limiting distribution of $\hat T_n(x)$  
cannot be derived by localizing the operator $L^{-1}$, and new ideas would need to be brought to bear.
We leave this as an exciting question for future work.

\appendix

\section{Periodic Function Spaces}
\label{app:function_spaces}

In this appendix, we provide explicit definitions   of the periodic function spaces used throughout
the manuscript, with similar conventions as~\cite{schmeisser1987} and~\cite{ruzhansky2009}. 

\subsection{Periodic H\"older Spaces}

Let $\Omega\subseteq \bbR^d$ be an open set.
For any function $f:\Omega \to \bbR$ which is differentiable 
up to order $k \geq 1$, 
and any multi-index $\gamma \in \bbN^d$, we write $|\gamma| = \sum_{i=1}^d \gamma_i$, and for all $|\gamma|\leq k$, 
$$D^\gamma f = \frac{\partial^{|\gamma|} f}{\partial x_1^{\gamma_1}\dots \partial x_d^{\gamma_d}}.$$
Given $\alpha > 0$, the H\"older space $\calC^\alpha(\Omega)$ is defined as the set
of functions on $\Omega$ which are continuously differentiable up to
order $\lfloor \alpha\rfloor$, and such that the H\"older norm  
$$\norm f_{\calC^\alpha(\Omega)} = \sum_{j=0}^{\lfloor\alpha\rfloor} \sup_{|\gamma| = j} \|D^\gamma f\|_{L^\infty(\Omega)}
 + \sum_{|\gamma|=\lfloor \alpha \rfloor} \sup_{\substack{x,y \in \Omega \\ x\neq y}} \frac{|D^\gamma f(x) - D^\gamma f(y)|}{\norm{x-y}^{\alpha-\lfloor \alpha\rfloor}}$$
is finite. 
We also write $\calC^\infty(\Omega) =\bigcap_{j\geq 1} \calC^j(\Omega)$,
and we denote by $\calC_c^\infty(\Omega)$ the set of maps in $\calC^\infty(\Omega)$ whose
support is compactly contained in $\Omega$.

For any $\alpha \geq 0$, $\calC^\alpha(\bbT^d)$  denotes
the set of $\bbZ^d$-periodic functions $f:\bbR^d \to \bbR$ such that $f \in \calC^\alpha(\bbR^d)$,
and in this case we write $\|f\|_{\calC^\alpha(\bbT^d)}$ instead of $\|f\|_{\calC^\alpha(\bbR^d)}$. 
Furthermore, we recall that for $0 < \alpha \leq \infty$,
$$\calC_0^\alpha(\bbT^d) = \left\{ f \in \calC^\alpha(\bbT^d): \int_{\bbT^d} f d\calL = 0\right\}.$$


\subsection{Riesz Potential Sobolev Spaces} 
Fix the collection of test functions $\calD(\bbT^d) = \calC^\infty(\bbT^d)$, endowed
with the standard test function topology, and let  $\calD_0(\bbT^d) = \calC_0^\infty(\bbT^d)$. 
The set of periodic distributions $\calD'(\bbT^d)$ is defined
as the set of continuous linear functionals on $\calC^\infty(\bbT^d)$, and we denote by $\langle\cdot,\cdot\rangle$
the induced duality pairing. $\calD'_0(\bbT^d)$ is similarly defined as the dual of $\calD_0(\bbT^d)$.
Furthermore, define the discrete Schwartz space $S(\bbZ^d)$ as the set of maps $\phi : \bbZ^d \to \bbR$
such that for any $k > 0$ there exists $C_k > 0$ such that
$$|\phi(\xi)| \leq C_k \|\xi\|^{-k}, \quad \text{for all } \xi \in \bbZ^d.$$
The set of tempered distributions on $\bbZ^d$ is denoted as  $\calS'(\bbZ^d)$, and is defined
as the set of continuous linear functionals from $\calS(\bbZ^d)$ to $\bbR$.
The Fourier transform defines a bijection $\calF:\calC^\infty(\bbT^d) \to S(\bbZ^d)$,
with inverse
$$\calF^{-1} a = \sum_{\xi\in\bbZ^d} a_\xi e^{2\pi i \langle \xi,\cdot\rangle},\quad\text{for all } a \in \calS(\bbZ_*^d)$$
which extends uniquely to a map $\calF:\calD'(\bbT^d) \to \calS'(\bbZ^d)$ via the action
$$\langle \calF u, \phi\rangle = \langle u, (\calF\phi)\circ \iota \rangle,$$
for any test function $\phi\in \calD(\bbT^d)$, where $\iota(x) =  -x$. 
We represent any periodic distribution $u \in \calD'(\bbT^d)$ by its formal power series
$$u = \sum_{\xi \in \bbZ^d} \calF u(\xi) e^{2\pi i \langle\xi,\cdot\rangle},$$ 
which coincides with the classical Fourier series of $u$ when it is sufficiently regular.

Define the Riesz kernel of order $s \in \bbR$ as the periodic distribution
$$I_s = \sum_{\xi\in \bbZ_*^d} \|\xi\|^{s} e^{2\pi i \langle \xi,\cdot\rangle},$$
which is in fact an $L^1(\bbT^d)$ function when $-d < s < 0$~\citep[Theorem 2.17]{stein1971},
and define the fractional Laplacian of order $s$ as the convolution operator
\begin{equation}
\label{eq:fractional_laplacian}
(-\Delta)^{s/2} u = I_s \star u = \calF^{-1}\big[ \|\cdot\|^{s} \calF u\big],
\end{equation}
for any periodic distribution $u \in \calD_0'(\bbT^d)$. 
We then define the inhomogeneous Sobolev space $H^{s,r}(\bbT^d)$, 
for all $s  \in \bbR$ and $1 < r < \infty$, 
as the set of periodic distributions $u\in \calD'(\bbT^d)$ for which the norm
$$\|u\|_{H^{s,r}(\bbT^d)} = \big\|(-\Delta)^{s/2} u\big\|_{L^r(\bbT^d)}$$
is finite. Likewise, the homogeneous Sobolev space $H_0^{s,r}(\bbT^d)$
is the set of periodic distributions $u \in \calD'_0(\bbT^d)$ such that the above norm is finite.
In the special case $r=2$, we omit the second superscript in the preceding definitions, and simply write $H^s(\bbT^d) := H^{s,2}(\bbT^d)$, $H_0^s(\bbT^d) := H_0^{s,2}(\bbT^d)$,
and $\|\cdot\|_{H^s(\bbT^d)} := \|\cdot\|_{H^{s,2}(\bbT^d)}$. 

Given $s\geq 0$ and $r > 1$, it follows from Theorem~3.5.6 of~\cite{schmeisser1987} that the 
$H^{-s,r}(\bbT^d)$ is isomorphic to the dual of the Banach space
$H^{s,r'}(\bbT^d)$, where $r'$ denotes the H\"older conjugate
of $r$. Combining this fact with a similar argument as in paragraph 3.13 of~\cite{adams2003}, one may deduce
the following. 
\begin{lemma}\label{lem:sobolev_dual_norm}
Let $r > 1$ and $s \geq 0$. Then, for all $u \in L_0^r(\bbT^d)$, 
$$\|u\|_{H^{-s,r}(\bbT^d)} \asymp \sup \left\{\langle u,v\rangle_{L^2(\bbT^d)}:
v\in H^{s,r'}(\bbT^d) , \|v\|_{H^{s,r'}(\bbT^d)} = 1\right\}.$$
\end{lemma}
%

It will be convenient to note the following norm equivalence. 
\begin{lemma}\label{lem:equiv_H1_norm}
For all $r > 1$ and $u \in H^{1,r}(\bbT^d)$, 
$$\|u\|_{H^{1,r}(\bbT^d)} \asymp \|u\|_{L^r(\bbT^d)} + \|\nabla u\|_{L^r(\bbT^d)},$$
and for all $u \in H_0^{1,r}(\bbT^d)$, 
$$\|u\|_{H^{1,r}(\bbT^d)} \asymp  \|\nabla u\|_{L^r(\bbT^d)},$$
where the implicit constants in the preceding two displays depend only on $d,r$.
\end{lemma}
The first assertion can be deduced from Theorem~3.5.4 of~\cite{schmeisser1987}, 
and the second is then a consequence  of the periodic Poincar\'e inequality. 

\section{Background on Elliptic PDE}
\label{app:pde}
We next summarize several
 results from the classical theory of uniformly elliptic partial differential equations,  which are used throughout our development. 
Given maps $\calA:\bbR^d \to \bbR^{d\times d}$ and $b: \bbR^d \to \bbR^d$, a
second-order differential operator of the form 
$$Mu = -\langle \calA, \nabla^2 u\rangle - \langle b, \nabla u\rangle$$
is said to be uniformly elliptic over a domain $\Omega \subseteq \bbR^d$ if it holds that
\begin{equation}
\label{eq:general_operator_condition_pd}
\calA(x) \succeq I_d/\lambda,\quad \text{over } \Omega,
\end{equation}
for some $\lambda > 0$. 
We will primarily be interested in operators $M$ for which the coefficients $\calA$ and $b$ have periodic entries,
and for which $u$ is subject to periodic boundary conditions. 
Nevertheless, we begin by stating in full generality several a priori regularity estimates for solutions
to equations of the form $Mu=f$. 
Some of the assumptions in the following statement are stronger than necessary, but sufficient for our purposes. 
\begin{lemma}
\label{lem:a_priori}
Let $\Omega$ be an open set with $\diam(\Omega) \leq D<\infty$, 
and let $\Omega_0 \subset\subset \Omega$ be an open set.
Let $\eta = \mathrm{dist}(\Omega_0, \partial\Omega)$. 
Let $f \in \calC^\beta(\widebar \Omega)$, and suppose $u \in \calC^{2+\beta}(\Omega)$
is a solution to the equation
$$Mu = f\quad \text{over } \Omega,$$
where we assume that  the coordinates of $\calA$ and $b$ satisfy
\begin{equation}
\label{eq:general_operator_condition_holder}
\max_{1 \leq i,j\leq d}\|\calA_{ij}\|_{\calC^\beta(\widebar \Omega)}\vee  \| b_i\|_{\calC^\beta(\widebar \Omega)} \leq 
\lambda  
\end{equation} 
and that equation~\eqref{eq:general_operator_condition_pd}
holds, for some $\lambda > 0$. Then, the following assertions hold.
\begin{thmlist}
\item  \label{lem:a_priori_schauder} (Schauder Estimate) There exists a constant $C = C(D,d, \lambda, \beta)$ such that 
$$\eta^2 \|u\|_{\calC^2(\Omega_0)} + \eta^{2+\beta}\|u\|_{\calC^{2+\beta}(\Omega_0)} \leq   C  \Big( \|u\|_{L^\infty(\Omega)} + \|f\|_{\calC^\beta(\Omega)}\Big).$$ 
\item \label{lem:a_priori_gradient} (Gradient Estimate)
Suppose further that $f$ takes the form $f=g+\div(G)$, where $g \in \calC^\beta(\Omega)$ and 
$G:\Omega \to \bbR^d$ is a vector field with entries in $\calC^{1+\beta}(\Omega)$.  
Then, 
 there exists a constant $C = C(D, d,\lambda, \beta, \eta) > 0$ such that
$$\|u\|_{\calC^{1+\beta}(\Omega_0)} \leq C \Big(\|u\|_{L^\infty(\Omega)}   + \|g\|_{L^\infty(\Omega)} + \|G\|_{\calC^\beta(\Omega)}\Big).$$
\end{thmlist}
\end{lemma} 
The bounds of Lemma~\ref{lem:a_priori} are standard, and 
can be deduced from Corollary 6.3 and  Theorem 8.32 of~\cite{gilbarg2001}. 
We further make use of the following De Giorgi-Nash-Moser estimate, 
as stated in  Theorem 8.17 of~\cite{gilbarg2001}.
\begin{lemma} \label{lem:a_priori_moser}  
Let $\Omega$ be an open set with $\diam(\Omega) \leq D < \infty$. 
Assume the coefficients $\calA$ and $b$ satisfy 
conditions~\eqref{eq:general_operator_condition_pd}--\eqref{eq:general_operator_condition_holder}.
Given $r > d/2$, let $g \in L^r(\Omega)$ and let $G : \Omega \to \bbR^d$ be a vector field
with entries in $L^{2r}(\Omega)$. 
Suppose $u \in H_0^1(\Omega)$ is a weak solution to the equation 
$$Mu=g+\div(G),\quad \text{in } \Omega.$$
Then, there exists   $C = C(D,M,r,\rho,\lambda, \beta) > 0$  
such that for all $R > 0$, all balls $B_{2R}:= B(y, 2R)\subseteq \Omega$ and all $\rho > 1$, 
$$\|u\|_{L^\infty(B_R)} \leq C \left(R^{-d/\rho} \|u\|_{L^\rho(2B_R)} +  R^{2-\frac{d}{r}} \|g\|_{L^r(\Omega)}
+ R^{1 - \frac d {2r}}\|G\|_{L^{2r}(\Omega)}\right).$$ 
\end{lemma}

With these preliminaries in place, let us  specialize our attention to the operators
which appear most prominently in our work: Suppose that $M$ takes the form
$$Mu = -\div(\calA\nabla u),$$
where we now assume that the matrix $\calA$ has $\bbZ^d$-periodic entries, and
we continue to assume that $\calA$ satisfies the uniform ellipticity and smoothness conditions
\begin{align}
\label{eq:periodic_uni_elliptic}
\calA(x) \succeq I_d/\lambda,\quad \text{over } \bbT^d \\ 
\label{eq:periodic_smooth_coeff}
\max_{1 \leq i,j\leq d} \|\calA_{ij}\|_{\calC^{1+\beta}(\bbT^d)} \leq \lambda.
\end{align}
By integration by parts, it is easy to see that $M$ is self-adjoint with respect to the $L^2(\bbT^d)$ inner product, 
and thus, that its range   contains only mean-zero functions. We may therefore view $M$ as an operator mapping
$H^2(\bbT^d)$ into $L_0^2(\bbT^d)$, and by convention, we will always restrict the domain of $M$ to $H_0^2(\bbT^d)$. 

Given $f \in L_0^2(\bbT^d)$, we will say that a map $u \in H_0^1(\bbT^d)$ is a weak solution to the PDE 
\begin{equation}
\label{eq:basic_torus_pde}
Mu=f, \quad \text{over } \bbT^d
\end{equation}
if for every $v \in H_0^1(\bbT^d)$, it holds that
$$ \langle f, v\rangle_{L^2(\bbT^d)}= \langle u,v\rangle_\calA := \int_{\bbT^d} \langle \calA \nabla u,\nabla v\rangle d\calL.$$
Under conditions~(\ref{eq:periodic_uni_elliptic}--\ref{eq:periodic_smooth_coeff}), 
it is straightforward to see that
$\langle \cdot,\cdot\rangle_\calA$ defines an inner product on $H_0^1(\bbT^d)$, which is equivalent to the standard
inner product $\langle \cdot,\cdot\rangle_{H^1(\bbT^d)}$. 
If a weak solution~$u$ admits a representative which lies in $\calC^2_0(\bbT^d)$, 
then it follows by integration by parts that
$$\langle f,v\rangle_{L^2(\bbT^d)} = \langle u,v\rangle_\calA = \langle Mu,v\rangle_{L^2(\bbT^d)},\quad \text{for all } v \in H_0^1(\bbT^d),$$
which implies that the equation $Mu=f$ is solved in the classical sense over $\bbT^d$. 
Thus, when they exist, classical solutions coincide with weak solutions.

Since $H_0^1(\bbT^d)$, endowed with the inner product $\langle \cdot,\cdot\rangle_\calA$, is a Hilbert
space, and since the linear functional $\langle f, \cdot\rangle_{L^2(\bbT^d)}:H_0^1(\bbT^d) \to \bbR$ is bounded,
it follows from the Riesz representation theorem that 
the PDE~\eqref{eq:basic_torus_pde} admits a unique weak solution $u \in H_0^1(\bbT^d)$ for any given $f \in L_0^2(\bbT^d)$. 
The following standard result, which can be deduced similarly as in~\cite{gilbarg2001}, shows that 
weak solutions in fact lie in $H_0^2(\bbT^d)$, and  are
classically differentiable if $f \in \calC_0^\beta(\bbT^d)$. 
\begin{lemma}\label{lem:isomorphism} 
Assume that conditions~\eqref{eq:periodic_uni_elliptic}--\eqref{eq:periodic_smooth_coeff} hold. 
Then, the operator $M$ is a bijection of $H_0^2(\bbT^d)$ onto $L_0^2(\bbT^d)$, and its restriction to
$\calC_0^{2+\beta}(\bbT^d)$ is a bijection onto $\calC_0^\beta(\bbT^d)$. Furthermore, it holds that
$$\|Mu\|_{L^2(\bbT^d)} \asymp \|u\|_{H^2(\bbT^d)},\quad \text{for all } u \in H_0^2(\bbT^d),$$
and
$$\|Mu\|_{\calC^{\beta}(\bbT^d)} \asymp \|u\|_{\calC^{2+\beta}(\bbT^d)},\quad\text{for all } u \in \calC_0^{2+\beta}(\bbT^d),$$
where the implicit constants depend only on $\lambda,\beta,d$.
\end{lemma}

It will be convenient to further note that
the operator~$M$ admits a periodic Green's function, and hence that
solutions~$u$ of the equation $Mu=f$ over~$\bbT^d$ are integral operators applied to~$f$.
We say that a map $\Gamma:\bbT^d \times \bbT^d \to \bbR$ is a periodic Green's function for $M$
if for all $x,y \in \bbT^d$, $\Gamma(\cdot,x) \in L_0^1(\bbT^d)$, $\Gamma(y,\cdot) \in L_0^1(\bbT^d)$, 
and 
$$v(x) = \langle \Gamma(\cdot,x), Mv\rangle_{L^2(\bbT^d)},\quad \text{for all }  v \in \calC_0^\infty(\bbT^d).$$
As an example, it can be deduced from~\cite{stein1971}
 that $-\Delta$ admits a  periodic Green's function of the form
$$\Gamma(y,x) = \Gamma_0(y-x) + b(y,x),$$
where $b \in \calC^\infty(\calQ\times\calQ)$, and $\Gamma_0$ is the traditional
fundamental solution of the Laplace equation on $\bbR^d$, 
that is, 
$$\Gamma_0(y-x) =  
\frac 1 {d(2-d)\omega_d} \|x-y\|^{2-d} \quad (d \geq 3),$$
where $\omega_d$ is the volume of the unit ball in dimension $d$.
The following result shows more generally that the operator $M$ admits a periodic Green's function, which has
the same blow-up behaviour as~$\Gamma_0$. 
 \begin{lemma}
\label{lem:fundamental_solution}
Assume that conditions~\eqref{eq:periodic_uni_elliptic}--\eqref{eq:periodic_smooth_coeff} hold,
and let $d\geq 3$. 
Then, there exists a unique  fundamental solution $\Gamma$ for the operator $M$,
 satisfying
 $\Gamma(\cdot,x) \in \calC^1_{\mathrm{loc}}(\calQ \setminus \{x\})$ for all $x \in \calQ$, 
and for which there exists a constant $C = C(\lambda,\beta,d) > 0$ such that
\begin{align*}
|\Gamma(y,x)| \leq C \|x-y\|_{\bbT^d}^{2-d},\quad
\|\nabla \Gamma(y,x)\| \leq C \|x-y\|_{\bbT^d}^{1-d},
\end{align*} 
for all $x,y \in \bbT^d$, $x \neq y$.  Furthermore, it holds that $\Gamma(x,y) = \Gamma(y,x)$ for any
$x,y \in \bbT^d$.  Finally, for any $f \in \calC_0^\beta(\bbT^d)$ and $x \in \bbT^d$, it holds that
$$M^{-1} f(x) = \int_{\bbT^d} \Gamma(y,x) f(y)dy, \quad \nabla M^{-1} f(x) = \int_{\bbT^d} \nabla_y\Gamma(y,x) f(y)dy.$$
\end{lemma}

For uniformly elliptic operators over $\bbR^d$, analogues of Lemma~\ref{lem:fundamental_solution}  
are classical; see for instance~\cite{littman1963}, \cite{stampacchia1965}, and \cite{gruter1982}. 
For the periodic case stated above, this result was shown for instance by~\cite{josien2019}. 

\section{Convergence Rates of the Kernel Density Estimator}
\label{app:kde}
Recall that we define the kernel density estimator 
$$\hat q_n = K_{h_n} \star Q_n = \int_{\bbR^d} K_{h_n}(\cdot - y) dQ_n(y),$$
where, in the above display, $Q_n \in \calP(\bbT^d)$ is extended by $\bbZ^d$-periodicity to a Borel measure on $\bbR^d$. 
Under condition~\Kernel{$\alpha$}, we assume that $K \in \calC_c^\infty(0,1)^d$, and thus
it is easy to see that the periodization of $K_{h_n}$, denoted by
$$\widebar K_{h_n} = \sum_{\xi\in \bbZ^d} K_{h_n}(\cdot-\xi),$$
defines a function in $\calC^\infty(\bbT^d)$, with the property that for all $x \in \bbR^d$,
there exists $\xi \in \bbZ^d$ such that $K_{h_n} (x)= \widebar K_{h_n}(x-\xi)$. By the Poisson summation formula, it holds that
\begin{equation}
\label{eq:poisson_summation}
\calF[\widebar K_{h_n}](\xi) = \calF[ K_{h_n}](\xi)=\calF[ K](h_n\xi), \quad \text{for all } \xi \in \bbZ^d.
\end{equation}
%
%
Throughout what follows, we write for all $u \in L^1(\bbT^d)$, 
$$\calK_{h_n} u = \widebar K_{h_n} \star u - u.$$
The aim of this appendix is to derive the convergence rate of the kernel density estimator under
the H\"older norms $\calC^\gamma(\bbT^d)$ with $\gamma \geq  0$, and under the negative Sobolev norms
$H^{-\gamma,r}(\bbT^d)$, $r \geq 2$. We begin with the former. 


\subsection{Convergence Rate under H\"older Norms} 
When $\gamma \geq  0$ is an integer, the question of characterizing the convergence
rate of $\hat q_n$ under the $\calC^\gamma(\bbT^d)$ norm reduces to deriving the uniform convergence
rate of derivatives of the kernel density estimator, which is a classical topic~\citep{bhattacharya1967,silverman1978,gine2016}. Using an
elementary interpolation argument, we can extend these results to non-integer values of $\gamma$.
\begin{proposition}
\label{prop:kde_holder}
Let $0 \leq \gamma < s$, and assume $q \in \calC^s(\bbT^d)$. 
Let $K$ satisfy condition~\Kernel{$s$}. Then, there exist constants $C ,b > 0$ depending
on $\|q\|_{\calC^s(\bbT^d)}, \gamma,s$
such that for any $h_n \geq 0$, 
$$\bbE \|q_{h_n} - q\|_{\calC^\gamma(\bbT^d)} \leq C h_n^{s-\gamma},$$
and  such that for any $h_n \leq u \leq 1$, 
$$\bbP\Big( \|\hat q_n - q_{h_n}\|_{\calC^\gamma(\bbT^d)} \geq u\Big) \leq C h_n^{-b} \exp\big( - nu^2 h_n^{2\gamma+d}/C\big).$$
In particular, there exists a  constant $C_1 > 0$, depending in particular on $q$, such that, 
almost surely, for all large enough $n \geq 1$, 
$$\|\hat q_n - q_{h_n}\|_{\calC^\gamma(\bbT^d)} \leq C_1 \sqrt{\frac{\log (h_n^{-1})}{nh_n^{2\gamma+d}}}.$$
\end{proposition}
\begin{proof}[Proof of Proposition~\ref{prop:kde_holder}] 
Given $a,b\geq 0$ and a bounded linear operator $F:\calC^a(\bbT^d) \to \calC^b(\bbT^d)$, denote 
the norm of $F$ by
$$\|F\|_{\calC^a(\bbT^d)\to \calC^b(\bbT^d)} = \sup_{u\in \calC^a(\bbT^d)\setminus \{0\}} \frac{\|F u\|_{\calC^a(\bbT^d)}}{\|u\|_{\calC^b(\bbT^d)}}.$$
When $\gamma$ is an integer,  it is a standard fact that the following bound holds under condition~\Kernel{$s$}
(see for instance~\cite{gine2016}, page 402):
$$\|\calK_{h_n}\|_{\calC^s(\bbT^d)\to \calC^\gamma(\bbT^d)} \lesssim h_n^{s-\gamma}. $$
If instead $\gamma > 0$ is not an integer, let $\gamma_0 = \lfloor \gamma\rfloor$,
$\gamma_1 = \lceil \gamma\rceil$ and $\theta = (\gamma-\gamma_0)/(\gamma_1-\gamma_0)$.  The periodic H\"older space 
$\calC^\gamma(\bbT^d)$ is a real interpolation space of exponent $\theta$ between $\calC^{\gamma_0}(\bbT^d)$
and $\calC^{\gamma_1}(\bbT^d)$ (cf.~\cite{schmeisser1987}, page 173), whence it follows that 
\begin{align*}
\|\calK_{h_n}\|_{\calC^s(\bbT^d)\to \calC^\gamma(\bbT^d)}  
 &\lesssim 
\|\calK_{h_n}\|_{\calC^s(\bbT^d)\to \calC^{\gamma_0}(\bbT^d)}^{1-\theta} 
\|\calK_{h_n}\|_{\calC^s(\bbT^d)\to \calC^{\gamma_1}(\bbT^d)}^\theta  \\
 &\lesssim \big(h_n^{s-\gamma_0}\big)^{1-\theta} \big(h_n^{s-\gamma_1}\big)^{\theta} = h_n^{s-\gamma}. 
\end{align*}
This proves the first claim. 
%
%
To prove the second claim, it can be deduced from the proof of
 Lemma 32 of~\cite{manole2021} (see also~\cite{gine2002}) 
 that for all integers $\gamma \geq 0$,  there exists  $C>0$ such that for all $h_n < u \leq 1$, 
$$\bbP\Big( \|\hat q_n - q_{h_n}\|_{\calC^\gamma(\bbT^d)} \geq u/h_n^\gamma\Big) \leq C h_n^{-d(d+\gamma+2)} \exp\big( - nu^2 h_n^d/C\big).$$
If instead $\gamma > 0$ is a real number, it follows again from the
interpolation property of the periodic H\"older spaces that the following inequality holds (cf.~\citet[Corollary 1.7]{lunardi2018}),
$$\|\hat q_n - q_{h_n}\|_{\calC^\gamma(\bbT^d)} \lesssim 
\|\hat q_n-q_{h_n}\|_{\calC^{\lfloor\gamma\rfloor}}^{1-\theta}
\|\hat q_n-q_{h_n}\|_{\calC^{\lceil\gamma\rceil}}^{\theta}.$$
 The preceding two displays imply,
\begin{align*}
\bbP&\left(\|\hat q_n - q_{h_n}\|_{\calC^{\gamma}(\bbT^d)} \geq u/h_n^{\gamma}\right) \\
 &\leq \bbP\left(\|\hat q_n - q_{h_n}\|_{\calC^{\lfloor\gamma\rfloor}(\bbT^d)} \geq u/h_n^{\lfloor\gamma\rfloor}\right) + 
 \bbP\left(\|\hat q_n - q_{h_n}\|_{\calC^{\lceil\gamma\rceil}(\bbT^d)} \geq u/h_n^{\lceil\gamma\rceil}\right) \\ 
 &\lesssim h_n^{-d(d+\lceil \gamma\rceil +2)} \exp\big( - nu^2 h_n^d/C\big),
\end{align*}
from which the second claim follows. The final claim is now a direct consequence of the first Borel-Cantelli Lemma.
\end{proof}
%
We deduce the following Lemma. 
\begin{lemma}
\label{lem:kde_as}
Let $q \in  \calC^{s}_+(\bbT^d)$ for some $s > 0$, and let $0\leq \gamma < s$. 
Let $K$ satisfy condition~\Kernel{$s$}. 
Suppose that for some $c > 0$, 
$$h_n = c \cdot n^{-a}, \quad \text{with } 0 < a < \frac 1 {2\gamma+d},$$ 
Then, for any $b > 0$, 
there exist  constants $C,\delta > 0$ depending on $K,\gamma,b,c,a$ such that with probability at least $1-C/n^b$.  
\begin{enumerate}
\item  $\|\hat q_n\|_{\calC^{\gamma}(\bbT^d)} \leq C$. 
\item $\hat q_n$ is a valid density, in the sense that $\hat q_n \geq 0$ over $\bbT^d$ and $\int_{\bbT^d} \hat q_n = 1$. 
\end{enumerate}
In particular, there exist   constants $C, N > 0$, possibly depending on $q$, such that the  two preceding assertions
hold almost surely for all $n \geq N$. 
\end{lemma} 
In particular, the following is an immediate consequence of Lemma~\ref{lem:kde_as} and 
Theorem~\ref{thm:caffarelli}.
\begin{lemma}
\label{lem:fitted_caffarelli}
Let  $p,q\in\calC^s_+(\bbT^d)$ and let $0 < \gamma < s$, $\gamma \not\in \bbN$.  Suppose that for some $c > 0$, 
$$h_n = c\cdot n^{-a},\quad\text{with } 0 < a < \frac 1 {2\gamma+d}.$$ 
Then, for any $b > 0$, there exist constants $C,\lambda$ depending on $\omega_s(p,q), \gamma,b,c$  
such that  
$$\|\hat\varphi_n\|_{\calC^{\gamma+2}(\calQ)} \leq \lambda, \quad \text{and } \hat \varphi_n \text{ is } \lambda^{-1} \text{-strongly convex,}$$
with probability at least $1-C/n^b$. In particular, there exist  constants $C,N,\lambda > 0$ depending
on $p$ and $q$ such that the above display holds almost surely for all $n \geq N$.
\end{lemma}

\subsection{Convergence Rate under Negative Sobolev  Norms}
We now turn to the question of bounding the convergence rate of the kernel density estimator
under the negative Riesz potential Sobolev norms. 
\begin{proposition}
\label{prop:kde_negative_sobolev}
Let $0 < \alpha < d$, $s \geq 0$, $r \geq 2$, and $q \in \calC_+^s(\bbT^d)$. 
Assume $K$ satisfies condition~\Kernel{$s+\alpha$}. 
Assume further that for some $c > 0$, $h_n \geq c \cdot n^{-1/d}$. 
Then, there exists a constant $C = C(\omega_s(q),K, c,\alpha,s,r) > 0$
such that
$$\|q_{h_n} - q\|_{H^{-\alpha,r}(\bbT^d)} \leq C h_n^{s+\alpha},$$
and, 
$$\bbE \|\hat q_n - q_{h_n}\|_{H^{-\alpha,r}(\bbT^d)}
\leq \Big(\bbE \|\hat q_n - q_{h_n}\|_{H^{-\alpha,r}(\bbT^d)}^r\Big)^{\frac 1 r}
 \leq C \frac{h_n^{\alpha - \frac d 2}}{\sqrt n}.$$
 Furthermore, 
 $$C^{-1} \frac{h_n^{2\alpha-d}}{n} 
\leq  \bbE \|\hat q_n - q_{h_n}\|_{H^{-\alpha}(\bbT^d)}^2 
 \leq C \frac{h_n^{2\alpha-d}}{n}.$$
\end{proposition}
When $\alpha$ = 0, the above result reduces to the traditional convergence
rate of the kernel density estimator under $L^r(\bbT^d)$ norms~\citep{gine2016}.
 When
$s = 0$ and $\alpha = 1$, this result has previously appeared in~\cite{divol2021a}, and our proof below 
is inspired by their approach. 

\begin{proof}[Proof of Proposition~\ref{prop:kde_negative_sobolev}]
We begin with the bound on the bias of $\hat q_n$. 
Recall the definition of the fractional
Laplacian $(-\Delta)^{-\alpha/2}$ and its associated Riesz kernel $I_\alpha$, given
in Appendix~\ref{app:function_spaces}. We have, 
\begin{align}
\label{eq:pf_bias_bound_besov}
\nonumber 
\|q_{h_n} - q\|_{H^{-\alpha,r}(\bbT^d)}
 &= \big\|I_\alpha \star \calK_{h_n}[q]\big\|_{L^r(\bbT^d)} \\ 
\nonumber 
 &= \big\|\calK_{h_n} [I_\alpha  \star q]\big\|_{L^r(\bbT^d)} \\  
 &= \big\|\calK_{h_n} \big [(-\Delta)^{-\alpha/2} q\big]\big\|_{L^r(\bbT^d)}.
\end{align} 
Now, let $\gamma = \alpha + s$, and notice that by definition of the Riesz potential spaces,
\begin{equation}
\label{eq:pf_kde_sobolev_frac_lapl_step}
\big\|(-\Delta)^{-\alpha/2}q\big\|_{H^{\gamma,r}(\bbT^d)}
 \lesssim \|q\|_{H^{s,r}(\bbT^d)}.
 \end{equation}
Therefore, the question of bounding the expression in equation~\eqref{eq:pf_bias_bound_besov} reduces
to the question of bounding the $L^r(\bbT^d)$ approximation error of the convolution of the
 $H^{\gamma,r}(\bbT^d)$ function $f := (-\Delta)^{-\alpha/2}[q]$. 
Specifically, we aim to show
%
$$\|\calK_{h_n} [f]\|_{L^r(\bbT^d)} \lesssim h_n^\gamma.$$
Define $M:= \|\cdot\|^{-\gamma}  \big(\calF[K](\cdot)-1\big)$, with $0/0=0$. We will show below that $M(a\cdot)$
is a $L^r(\bbT^d)$ Fourier multiplier for any given $a > 0$, in the sense that
$$\big\|\calF^{-1}[M(a\cdot) \calF f ]\big\|_{L^r(\bbT^d)} \leq C \|f\|_{L^r(\bbT^d)}$$
for all $f \in L^r(\bbT^d)$, for a constant $C > 0$ independent of $a$. 
Taking this fact for granted momentarily, let us show how the claim follows.
Using equation~\eqref{eq:poisson_summation}, we have
\begin{align*}
\|\calK_{h_n}[f]\|_{L^r(\bbT^d)}
 &= \big\|\calF^{-1}\big(\calF[\calK_{h_n} [f ]\big) \big\|_{L^r(\bbT^d)}  \\ 
 &= h_n^\gamma \big\|\calF^{-1}\big(M(h_n\xi)\|\xi\|^{\gamma} \calF[f](\xi)\big) \big\|_{L^r(\bbT^d)} \\
 &\lesssim h_n^\gamma \big\|\calF^{-1}\big(\|\xi\|^{\gamma} \calF[f](\xi)\big) \big\|_{L^r(\bbT^d)} 
 = h_n^\gamma \|f\|_{H^{\gamma,r}(\bbT^d)} \lesssim h_n^{s+\alpha} \|q\|_{H^{s,r}(\bbT^d)},
\end{align*} 
where we used equation~\eqref{eq:pf_kde_sobolev_frac_lapl_step} to obtain the final inequality. 

It thus remains to show that $M(a\cdot)$ is indeed a Fourier multiplier, with norm independent
of $a > 0$. 
To do so, it will suffice to show that $M$ satisfies the conditions of Mikhlin's multiplier
theorem~\citep[Theorem 6.2.7]{grafakos2008}. 
Abbreviate $g = \calF[K]$. By condition~\Kernel{$\gamma$}, $M$ is bounded over $\bbR^d$.
Furthermore, 
recall that condition~\Kernel{$\gamma$}
implies
 that $K$ is a kernel of order $\gamma-1$, i.e. $D^\kappa g(0) = 0$ for all multi-indices
 $\kappa$ such that $1 \leq |\kappa| \leq \gamma-1$. 
For any such $\kappa$, we have 
$$D^\kappa g(\xi) = D^\kappa g(0) + \sum_{\alpha: 1 \leq |\omega| \leq \gamma-1-|\kappa|} D^{\kappa+\omega} g(0) \xi^\omega
+ O(\|\xi\|^{\gamma-|\kappa|}) = O(\|\xi\|^{\gamma-|\kappa|}),$$
for  $\|\xi\|\leq 1$. Since $K$ and $g$ are Schwartz functions, the above bound also continues to hold trivially
for all $|\kappa| \geq \gamma$. Using the general Leibniz rule, we deduce that for all 
multi-indices $\omega$ satisfying $1 \leq |\omega| \leq \frac d 2 + 1$, we have\footnote{The notation $\kappa \leq \omega$ 
is to be understood componentwise. 
}
\begin{align*}
\big|D^\omega M(\xi)\big|
 &= \big|D^\omega \big( g(\xi)\|\xi\|^{-\gamma}\big)\big| \\
 &\lesssim \sum_{\kappa\leq\omega}\big|D^{\kappa} g(\xi)\big|\big| D^{\omega-\kappa} \|\xi\|^{-\gamma}\big| 
 \lesssim \sum_{\kappa\leq\omega}\|\xi\|^{\gamma-|\kappa|} \|\xi\|^{-\gamma-|\omega-\kappa|} 
 \lesssim \|\xi\|^{-|\omega|},
\end{align*}
for all $\|\xi\|\leq 1$. Finally, since $g$ is a Schwarz function, the last bound of the above display
continues to hold trivially when $\|\xi\| > 1$.
The conditions of Mikhlin's Multiplier Theorem are thus satisfied. This completes the proof of the bias bound.

To prove the first claim about the variance, we follow the proof of~\cite{divol2021a} closely. 
We would like to bound  $V = \bbE \|(1/n)\sum_{i=1}^n \big(U_i-\bbE U_i\big)\|_{L^r(\bbT^d)}^r$, where
$$U_i(x) = I_\alpha \star \widebar K_{h_n}^o(X_i-x) ,\quad x \in \bbT^d,$$
where we write
$\widebar K_{h_n}^o = \widebar K_{h_n} - 1$. 
Using Rosenthal's inequality~\citep{rosenthal1970,rosenthal1972}, it holds that
$$V \lesssim n^{-r/2} \int_{\bbT^d} \big( \bbE |U_1(x)|^2\big)^{r/2} dx + n^{1-r} \int_{\bbT^d} \bbE|U_1(x)|^r dx.$$
Notice that, due to the upper-boundedness of $q$, one has for any $x \in \bbT^d$, 
\begin{align*}
\bbE |U_1(x)|^r 
 &= \int_{\bbT^d} \big| I_\alpha\star \widebar K_{h_n}^o(y-x)\big|^r q(y)dy \\
 &\lesssim \int_{\bbT^d} \big| I_\alpha\star \widebar K_{h_n}^o(y-x)\big|^r dy
= \int_{\bbT^d} \big| I_\alpha\star \widebar K_{h_n}^o(y)\big|^r dy = \|\widebar K_{h_n}^o\|_{H^{-\alpha,r}(\bbT^d)}^r,
\end{align*}
where we used the translational invariance of $\calL$. 
We now make use of the following.
\begin{lemma}
\label{lem:sobolev_growth}
Let $0 < \alpha < d$, $r \geq 2$, and assume condition~\Kernel{$\alpha$} holds.  
Then, there exists  $C = C(\alpha,r) > 0$ such that 
$$C^{-1} h_n^{r\alpha - (r-1)d} \leq \|\widebar K_{h_n}^o\|_{H^{-\alpha,r}(\bbT^d)}^r \leq C h_n^{r\alpha - (r-1)d}.$$
\end{lemma}
The proof appears below. We thus have
\begin{equation}
\label{eq:pf_negH_kde_var_step1}
V \lesssim n^{-r/2}  \big(h_n^{2\alpha-d}\big)^{r/2} + n^{1-r} h_n^{r\alpha - (r-1)d}
= n^{-r/2}  h_n^{r\left(\alpha-\frac{d}{2}\right)} + n^{1-r} h_n^{r\alpha - (r-1)d}.
\end{equation}
Note that
\begin{align*}
c\cdot  n^{-r/2}  h_n^{r\left(\alpha-\frac{d}{2}\right)}\geq  n^{1-r} h_n^{r\alpha - (r-1)d}
 & \Longleftrightarrow h_n \geq c \cdot n^{-1/d},
\end{align*}
thus, under our assumption that $h_n\geq cn^{-1/d}$, the first term on the right-hand
side of equation~\eqref{eq:pf_negH_kde_var_step1} dominates, and we obtain
$$V\lesssim n^{-r/2} h_n^{r\left(\alpha - \frac d 2\right)}.$$
It thus remains to prove the last claim, for which it suffices to prove the lower bound. 
We~have,
\begin{align*}
\bbE\|\hat q_n-q_{h_n}\|_{H^{-\alpha}(\bbT^d)}^2
 &= \bbE \|(1/n)\sum_{i=1}^n \big(U_i-\bbE U_i\big)\|_{L^2(\bbT^d)}^2 \\
 &\asymp  \frac 1 n \int_{\bbT^d} \int_{\bbT^d} \big| I_\alpha\star \widebar K_{h_n}^o(y-x)\big|^2dx dy  
 = \frac 1 n \|\widebar K_{h_n}^o\|_{H^{-\alpha}(\bbT^d)}^2 \asymp \frac 1 n h_n^{2\alpha-d},
\end{align*}
where we used the fact that $q$ is bounded from above and below over $\bbT^d$, and we invoked Lemma~\ref{lem:sobolev_growth} in the final order assessment. 
The claim follows. 
\end{proof} 

\begin{proof}[{Proof of Lemma~\ref{lem:sobolev_growth}}]
Let
$$\kappa = \sup_{\xi \in \bbR^d \setminus\{0\}} \frac{|\calF[K](\xi) - 1|}{\|\xi\|},$$
which is finite by condition~\Kernel{1}. 
By the Hausdorff-Young inequality, notice that 
$$\|\widebar K_{h_n}^o\|_{H^{-\alpha,r}(\bbT^d)} = 
 \left\|\calF^{-1}\big[\|\cdot\|^{-\alpha} \calF[\widebar K_{h_n}^o]\big]\right\|_{L^r(\bbT^d)}
 \leq \big\|\big(\|\cdot\|^{-\alpha} \calF[\widebar K_{h_n}^o]\big)_{\xi \in \bbZ^d_*}\big\|_{\ell^{r'}(\bbZ^d_*)},$$
 where $r' = r/(r-1)$ is the H\"older conjugate of $r$. We thus have, 
 \begin{align*}
\|\widebar K_{h_n}^o\|_{H^{-\alpha,r}(\bbT^d)}^{r'} 
 &\leq \sum_{\xi \in \bbZ_*^d} \frac{|\calF[K](h_n\xi)|^{r'}}{\|\xi\|^{\alpha r'}},
 \end{align*}
%
where
\begin{align*}
S_{n,1} = \sum_{\substack{\xi \in \bbZ_*^d \\ \kappa\|h_n\xi\| \leq 1/2}} \frac{\big|\calF[K](h_n\xi)\big|^{r'}}{\|2\pi\xi\|^{r'\alpha}},\qquad 
S_{n,2} = \sum_{\substack{\xi \in \bbZ_*^d \\ \kappa\|h_n\xi\| > 1/2}} \frac{\big|\calF[K](h_n\xi)\big|^{r'}}{\|2\pi\xi\|^{r'\alpha}}.
\end{align*}
Regarding term $S_{n,1}$, apply condition~\Kernel{$\alpha$} to obtain
\begin{align*}
S_{n,1} 
 \geq\sum_{\substack{\xi \in \bbZ_*^d \\ \kappa\|h_n\xi\| \leq 1/2}} \frac{(1 - \kappa\|h_n\xi\|)_+^{r'}}{\|2\pi\xi\|^{r'\alpha}} 
 \geq\frac 1 {2^{r'}\|2\pi\|^{r'\alpha}} \sum_{\substack{\xi \in \bbZ_*^d \\ \kappa\|h_n\xi\| \leq 1/2}} \frac{1}{\|\xi\|^{r'\alpha}} 
 \asymp h_n^{r'\alpha-d},
\end{align*} 
where the final order assessment holds due to the condition $r\alpha < d(r-1)$, which implies
$r' \alpha < d$. A similar argument shows that $S_{n,1} \lesssim h_n^{r'\alpha-d}$. 
It thus remains to upper bound $S_{n,2}$. Recall that $K \in \calC_c^\infty(\bbR^d)$, thus
$K$ and $\calF[K]$ both belong to the Schwartz space. In particular, $\calF[K](\xi) \lesssim \|\xi\|^{-\ell}$ for any $\ell > 0$.
Choosing $\ell$ such that $ r'(\ell+\alpha) > d$, it follows that 
\begin{align*}
S_{n,2} \lesssim h_n^{- r' \ell} \sum_{\substack{\xi \in \bbZ_*^d \\ \kappa\|h_n\xi\|^\rho > 1/2}} \|\xi\|^{-r'(\ell+\alpha)}
 \lesssim h_n^{-r'\ell} h_n^{-d+r' (\ell+\alpha)} = h_n^{r'\alpha-d}.
\end{align*}
Deduce from here that 
$$\|\widebar K_{h_n}^o\|_{H^{-\alpha,r}(\bbT^d)}^r \asymp h_n^{\frac r {r'}(r'\alpha-d)} = h_n^{r\alpha - (r-1)d}.$$ 
This proves the claim. 
\end{proof}

\section{Additional Proofs from Section~\ref{sec:main_results}}
\label{app:additional_proofs_2}

\subsection{Proof of Lemma~\ref{cor:radial}}
\label{app:pf_cor_radial}
Since $K$ is radial, $\calF[K]$ is also radial, thus $\Sigma$ takes the form
$$\Sigma_0 = \int_{\bbR^d} g(\xi) \xi\xi^\top d\xi,$$
for a radial Schwartz function $g$. Consider the entry $(1,2)$ of this matrix: 
$$\sigma_{1,2} = \int_{\bbR^d} g(\xi) \xi_1\xi_2 d\xi.$$
By passing to the spherical coordinates:
\begin{align*}
\xi_1 &= r \cos\theta_1 \\ 
\xi_2 &= r \sin\theta_1\cos\theta_2 \\ 
\xi_3 &= r \sin\theta_1\sin\theta_2\cos\theta_3 \\ 
&\vdots \\
\xi_{d-1} &= r \sin\theta_1\dots \sin\theta_{d-2} \cos\theta_{d-1} \\
\xi_d &=  r \sin\theta_1\dots \sin\theta_{d-2} \sin\theta_{d-1},
\end{align*}
with $\theta_1,\dots,\theta_{d-2} \in [0,\pi]$, $\theta_{d-1} \in [0,2\pi]$, $r \geq 0$, 
and using Jacobian identity 
$$dx = (r^{d-1} \sin^{d-2}\theta_1 \sin^{d-3}\theta_2 \dots \sin\theta_{d-2}) d\theta_n \dots d\theta_1 dr,$$
we obtain:
\begin{align*}
\begin{multlined}[0.9\textwidth]
\sigma_{1,2}^2 = \int_0^\infty \int_0^\pi \dots \int_0^\pi \int_0^{2\pi} g(r) (r\cos\theta_1)\times \\[0.05in]
\times(r\sin\theta_1\cos\theta_2) 
(r^{d-1} \sin^{d-2}\theta_1 \sin^{d-3}\theta_2 \dots \sin\theta_{d-2}) d\theta_{d-1} \dots d\theta_1 dr.
\end{multlined}
\end{align*}
Notice that the integral with respect to $\theta_2$ in the above display vanishes, i.e.
$$ \int_0^\pi \cos\theta_2 \sin^{d-3}\theta_2 d\theta_2 = 0,$$
thus we obtain $\sigma_{1,2}^2 = 0$. A similar argument shows that the other off-diagonal entries
of $\Sigma_0$ vanish, and the claim follows.
\qed

\section{Additional Proofs from Section~\ref{sec:linearization}}
\label{app:additional_proofs_3}

\subsection{Proof of Lemma~\ref{lem:dual_norm_equiv}}
\label{app:pf_dual_norm_equiv}
Given $u \in H_0^2(\bbT^d)$, let $f = Lu$. By Lemma~\ref{lem:inverse_L_to_E}, 
$u$ is a weak $H_0^1(\bbT^d)$ solution to the equation 
$$Eu = f(\nabla\varphi_0^*)\det(\nabla\varphi_0^*),\quad \text{over } \bbT^d,$$
and hence satisfies 
$$\langle u,v\rangle_A = \langle f(\nabla\varphi_0^*)\det(\nabla\varphi_0^*), v\rangle_{L^2(\bbT^d)}
= \langle f, v(\nabla\varphi_0^*)\rangle_{L^2(\bbT^d)},$$
for all $v \in H_0^1(\bbT^d)$. Using the fact that $\langle\cdot,\cdot\rangle_A$ defines an equivalent
inner product to $\langle\cdot,\cdot\rangle_{H^1(\bbT^d)}$, we deduce that 
$$\|u\|_{H^1(\bbT^d)}^2 {\lesssim }
\langle u,u\rangle_A {= }\langle f,u(\nabla\varphi_0^*)\rangle_{L^2(\bbT^d)}
\leq \|f\|_{H^{-1}(\bbT^d)} \|u(\nabla\varphi_0^*)\|_{H^1(\bbT^d)}\lesssim 
\|f\|_{H^{-1}(\bbT^d)} \|u\|_{H^1(\bbT^d)},$$
where we used Lemma~\ref{lem:sobolev_comp}. We thus have $\|u\|_{H^1(\bbT^d)} \lesssim \|f\|_{H^{-1}(\bbT^d)}$, and to prove
the reverse inequality, use Lemma~\ref{lem:sobolev_dual_norm} to write
\begin{align*}
\|f\|_{H^{-1}(\bbT^d)} 
 &\asymp \sup_{\substack{v \in H^1(\bbT^d) \\ \|v\|_{H^1(\bbT^d)} = 1}}
\langle f,v\rangle_{L^2(\bbT^d)}  \\
 &= \sup_{\substack{v \in H^1(\bbT^d) \\ \|v\|_{H^1(\bbT^d)} = 1}}
\langle u,v(\nabla\varphi_0)\rangle_{A} 
\leq \|u\|_{H^1(\bbT^d)} \sup_{\substack{v \in H^1(\bbT^d) \\ \|v\|_{H^1(\bbT^d)} = 1}} \|v(\nabla\varphi_0)\|_{H^1(\bbT^d)}.
\end{align*}
The claim now follows from Lemma~\ref{lem:sobolev_comp}.  \qed 
%
%

\subsection{Proof of Lemma~\ref{lem:frechet}}  
\label{app:pf_lem_frechet}
 We will make use of the following Lemma, which is proven below. 
In what follows, for any integers $d,k\geq 1$ and any differentiable function $f:\bbR^{d\times k} \to \bbR$,
we denote by $\partial f(A)$   the $d\times k$ matrix with entries $(\partial f/\partial x_{ij})(A)$, $i=1, \dots, d$, $j=1, \dots, k$.
Furthermore, for any open subset $\Theta \subseteq \bbR^{d\times k}$ and $\alpha > 0$, we denote by $\calC^\alpha(\Omega;\Theta)$
the set of matrix-valued maps $A:\calQ \to \Theta$ with entries in $\calC^\alpha(\calQ)$. 
\begin{lemma}
\label{lem:holder_expansion}
Let $d,k \geq 1$ be integers, and $\Theta  \subseteq \bbR^{d \times k}$.
Let $f\in \calC^{2}(\Theta)$ and $A, B \in \calC^{1}(\calQ;\Theta)$. 
Then, there exists a universal constant $C > 0$ such that  function
$$g:\calQ \to \bbR, \quad g(x) = f(A(x) + B(x)) - f(A(x)) - \left\langle \partial f(A(x)), B(x)\right\rangle$$
satisfies
$$\|g\|_{\calC^\beta(\calQ)} \leq C \|f\|_{\calC^{2+\beta}(\Theta)}\Big(1\vee \|A\|_{\calC^1(\calQ;\Theta)}^{\beta}\Big) 
\|B\|_{\calC^1(\calQ;\Theta)}^2.$$
\end{lemma}
A proof of Lemma~\ref{lem:holder_expansion} appears in Appendix~\ref{app:holder_expansion}. 
We now turn to the claim. 
We prove the Fr\'echet differentiability of $\hat\Psi$, and a symmetric argument may be used for $\Psi$. 
Let $u \in \calC^{2+\beta}_0(\calQ)$. Then, 
\begin{align*}
 \big\|&\widehat \Psi[\varphi + u] - \widehat \Psi[\varphi] - \widehat \Psi_{\varphi}'[u]\big\|_{\calC^\beta(\calQ)} \\
 &= \big\| \det(\nabla^2\varphi)  \hat q(\nabla\varphi) - 
           \det(\nabla^2\varphi + \nabla^2 u) \hat q(\nabla\varphi+\nabla u)\\ &\qquad +
  		   \det(\nabla^2\varphi)\big[\hat q(\nabla\varphi)\langle (\nabla^2\varphi)^{-1}, \nabla^2 u\rangle+
  										   	 \langle \nabla u, \nabla \hat q(\nabla \varphi)\rangle\big] \big\|_{\calC^\beta(\calQ)} \\ 
 &\leq  \big\| \det(\nabla^2\varphi)
 		  \left[\hat q(\nabla\varphi) - \hat q(\nabla \varphi {+}  \nabla u )-
 		  \langle \nabla u, \nabla \hat q(\nabla \varphi)\rangle\right]   \big\|_{\calC^\beta(\calQ)}  
 		  \\ &+ 
  			\big\|\hat q(\nabla\varphi {+} \nabla u)
  			\left[  \det(\nabla^2\varphi) -\det(\nabla^2\varphi {+}  \nabla^2 u) \right] + \hat q(\nabla\varphi)\det(\nabla^2\varphi) 
  			\langle (\nabla^2\varphi)^{-1}, \nabla^2u \rangle\big\|_{\calC^\beta(\calQ)} \\
 &=:(I)+(II).
\end{align*}
We first bound $(I)$. By Lemma~\ref{lem:holder_products}, we have, 
$$
(I) \lesssim \|\varphi\|_{\calC^{2+\beta}(\calQ)}
 		  \big\|\hat q(\nabla\varphi) - \hat q(\nabla \varphi {+}  \nabla u )-
 		  \langle \nabla u, \nabla \hat q(\nabla \varphi)\rangle \big\|_{\calC^\beta(\calQ)}.$$
Now, since $\hat q \in \calC^{2+\beta}(\calQ)$, we may invoke
Lemma~\ref{lem:holder_expansion} to deduce
$$ (I) \lesssim   \|\hat q\|_{\calC^{2+\beta}(\calQ)} (1+\|\varphi\|_{\calC^{2+\beta}(\calQ)}^{1+\beta})\|u\|_{\calC^1(\calQ)}^{2}
    \lesssim \|u\|_{\calC^{2+\beta}(\calQ)}^{2}.
$$
To bound term $(II)$, apply Lemma~\ref{lem:holder_products} to obtain, almost surely, 
\begin{align*}
(II) &= \big\|\hat q(\nabla\varphi {+} \nabla u)
\left[  \det(\nabla^2\varphi) -\det(\nabla^2\varphi {+}  \nabla^2 u) \right] + \hat q(\nabla\varphi)\det(\nabla^2\varphi) 
\langle (\nabla^2\varphi)^{-1}, \nabla^2 u \rangle\big\|_{\calC^\beta(\calQ)}  \\ 
 &\leq \big\|\hat q(\nabla\varphi {+} \nabla u)
\left[  \det(\nabla^2\varphi) -\det(\nabla^2\varphi {+}  \nabla^2 u)+\det(\nabla^2\varphi) 
\langle (\nabla^2\varphi)^{-1}, \nabla^2 u \rangle \right]
\big\|_{\calC^\beta(\calQ)} \\ 
 &+ \big\|\big[ \hat q(\nabla\varphi+\nabla u ) - \hat q(\nabla\varphi)\big]\det(\nabla^2\varphi) 
\langle (\nabla^2\varphi)^{-1}, \nabla^2u \rangle\big\|_{\calC^\beta(\calQ)} \\
 &\lesssim \big\| \det(\nabla^2\varphi) -\det(\nabla^2\varphi {+}  \nabla^2 u)+\det(\nabla^2\varphi) 
\langle (\nabla^2\varphi)^{-1}, \nabla^2 u \rangle 
\big\|_{\calC^\beta(\calQ)} \\ 
 &+ \|u\|_{\calC^{2+\beta}(\bbT^d)}\big\|\hat q(\nabla\varphi+\nabla u ) - \hat q(\nabla\varphi)\big\|_{\calC^\beta(\calQ)}.
\end{align*}
By Lemma A.1 of~\cite{figalli2017} and Lemma~\ref{lem:holder_expansion}, the first term in the final line of the above display 
is of order  $O(\|u\|_{\calC^{2+\beta}(\calQ)}^{2})$. Furthermore, since $\hat q \in \calC^1(\calQ)$, 
the second term is of order $O(\|u\|_{\calC^{2+\beta}(\calQ)}^2)$. The claim follows.
\qed 

\subsection{Proof of Lemma~\ref{lem:holder_expansion}}
\label{app:holder_expansion}
By a first-order Taylor expansion, it readily holds that for all $x \in \calQ$,
$$|g(x)| \lesssim \|f\|_{\calC^{2}(\Theta)} \|B(x)\|^{2} \lesssim \|f\|_{\calC^{2}(\Theta)} \|B\|_{\calC^1(\Theta)}^{2}$$
It thus suffices to show that $g$ is uniformly $\beta$-H\"older continuous, i.e. 
$|g(x) - g(y)| \leq L \|x-y\|^\beta$ for some $L > 0$ and all $x,y \in \calQ$. 
To do so, define for all such $x,y$ the function
$$h(x,y) = f(A(x) + B(y)) - f(A(x)) - \langle \partial f(A(x)), B(y)\rangle.$$
We have, 
\begin{align*}
|g(x) - h(x,y)| 
 &= \left| f(A(x) + B(x)) - f(A(x) + B(y)) - \langle \partial f(A(x)), B(x) - B(y)\rangle\right| \\
 &\leq  \left| f(A(x) + B(x)) - f(A(x) + B(y)) - \langle \partial f(A(x)+B(y)), B(x) - B(y)\rangle\right| 
  \\ &+ \|\partial f(A(x) + B(y)) - \partial f(A(x))\| \|B(x)-B(y)\| =: (a) + (b). 
 \end{align*}
Clearly, 
\begin{align*}
(a) \leq \|f\|_{\calC^{2}(\Theta)} \|B(x) - B(y)\|^{2} \leq \|f\|_{\calC^2(\Theta)} \|B\|_{\calC^1(\calQ,\Theta)}^{2} \|x-y\|^{2}.
\end{align*}
Furthermore, 
\begin{align*}
(b) \leq \|f\|_{\calC^{2}(\Theta)} \|B(y)\| \|B(x) - B(y)\|
\leq \|f\|_{\calC^2(\Theta)} \|B\|_{\calC^1(\calQ)}^{2}\|x-y\|.
\end{align*}
Thus, 
$$|g(x) - h(x,y)| \leq \|f\|_{\calC^{2}(\Theta)} \|B\|_{\calC^1(\calQ)}^2\|x-y\|.$$
Furthermore, we have
\begin{align*}
\big|h(x,y) - g(y)\big|
 &= \Big| \big[f(A(x) + B(y)) - f(A(y) + B(y))\big] + \big[ f(A(x)) - f(A(y) )\big]  \\ &\qquad + 
 \langle \partial f(A(x))-\partial f(A(y)), B(y)\rangle \Big|\\
 &\leq \sum_{i=1}^d \sum_{j=1}^k \left|\frac{\partial^2 f}{\partial X_{ij}}(A(x)) - \frac{\partial^2 f}{\partial X_{ij}}(A(y))\right| |B_i(y)B_j(y)| \\ 
 &\leq \|f\|_{\calC^{2+\beta}(\Theta)} \|B\|_{L^\infty(\Theta)}^2 \|A(x) - A(y)\|^\beta \\ 
 &\leq \|f\|_{\calC^{2+\beta}(\Theta)} \|B\|_{\calC^1(\Theta)}^2 \|A\|_{\calC^1(\Theta)}^\beta \|x - y\|^\beta.
\end{align*}
 The claim follows from here. \qed


\subsection{Proof of Lemma~\ref{lem:comparing_psin_psi}}
\label{app:lem_comparing_psin_psi}

By Lemma~\ref{lem:frechet},
\begin{align*}
\big\|&\widehat \Psi'_{\varphi_0}[u] -  \Psi'_{\varphi_0}[u]\big\|_{\calC^\beta(\bbT^d)} \\
 &= \left\|\det(\nabla^2\varphi_0)\Big[\big(\hat q(\nabla\varphi_0) - 
 q(\nabla\varphi_0)\big)\langle \nabla^2\varphi_0, \nabla^2 u\rangle +	\langle \nabla u, \nabla\hat q(\nabla\varphi_0) - \nabla q(\nabla \varphi_0) \rangle\Big] \right\|_{\calC^\beta(\bbT^d)}  \\ 
 &\lesssim  \|\hat q(\nabla\varphi_0) - q(\nabla\varphi_0)\|_{\calC^\beta(\bbT^d)} \|\nabla^2 u\|_{\calC^\beta(\bbT^d)} +	
 \|\nabla u\|_{\calC^\beta(\bbT^d)} \|\nabla\hat q(\nabla\varphi_0) - \nabla q(\nabla \varphi_0)\|_{\calC^\beta(\bbT^d)},
 \end{align*}
 where we used Lemma~\ref{lem:holder_products} to bound the H\"older norms
 of products. By Lemma~\ref{lem:holder_comp}, we further deduce,
\begin{align*}
\big\|\widehat \Psi'_{\varphi_0}[u] -  \Psi'_{\varphi_0}[u]\big\|_{\calC^\beta(\bbT^d)}
 &\lesssim \|\hat q - q\|_{\calC^\beta(\bbT^d)} \|\nabla^2 u\|_{\calC^\beta(\bbT^d)} +	
 \|\nabla u\|_{\calC^\beta(\bbT^d)} \|\nabla\hat q - \nabla q\|_{\calC^\beta(\bbT^d)} \\
 &\lesssim \|u\|_{\calC^{2+\beta}(\bbT^d)}\|\hat q - q\|_{\calC^\beta(\bbT^d)}  +	
 \|u\|_{\calC^{1+\beta}(\bbT^d)} \|\hat q - q\|_{\calC^{1+\beta}(\bbT^d)},
\end{align*} 
as was to be shown.\qed   

\section{Additional Proofs from Section~\ref{sec:bias_variance_bounds}}
\label{app:extra_proofs_variance}
 \subsection{Proof of Lemma~\ref{lem:operator_G0}}
  \label{app:pf_lem_operator_G0}
Let $u = E_0^{-1}[f(S_{Y,x}(\cdot))]$. Notice that $\calB_0$ is symmetric, 
thus we have  for all $y \in \bbT^d$, 
\begin{align*}
f(y)
 &= E_0[u](S_{Y,x}^{-1}(y)) \\
 &= -\text{tr}\big(A_0^\top \nabla^2 u(S_{Y,x}^{-1}(y))\big) \\ 
 &= -\text{tr}\big((\calB_0^2 A_0)^\top \calB_0^{-2} \nabla^2 u(S_{Y,x}^{-1}(y))\big) \\
 &= -\text{tr}\big((\calB_0^2 A_0)^\top \nabla_y^2 u(S_{Y,x}^{-1}(y))\big)  \\
 &= G_0[u(S_{Y,x}^{-1}(\cdot))](y)
\end{align*} 
It follows that
$u = G_0^{-1}[f](S_{Y,x}(y)) = E_0^{-1}[f(S_{Y,x}(\cdot))](y)$, as was to be shown.\qed 

\subsection{Proof of Lemma~\ref{lem:Q_to_uniform_reduction}}
\label{app:pf_lem_Q_to_uniform_reduction}
To prove the claim, it suffices to show that 
$$\Big\| \Cov\big\{\nabla G_0^{-1}\big[\widebar K_{h_n}^o\big](Y- \nabla\varphi_0(x))\big\}
 - q(\nabla\varphi_0(x)) \Cov\big\{\nabla G_0^{-1}\big[\widebar K_{h_n}^o\big](U)\big\} \Big\|
\lesssim h_n^{2+\epsilon-d}$$
for some $\epsilon>0$.
Let $R(u) = \nabla G_0^{-1}[\widebar K_{h_n}^o](u)$, and notice that 
\begin{align*}
\Big\|\bbE[R&(Y-\nabla\varphi_0(x))R(Y-\nabla\varphi_0(x))^\top] - q(\nabla\varphi_0(x)) \bbE[R(U)R(U)^\top]\Big\| \\ 
 &= \left\| \int_{\bbT^d} R(y)R(y)^\top q(y+\nabla\varphi_0(x))dy - q(\nabla\varphi_0(x))\int_{\bbT^d} R(y)R(y)^\top dy \right|\\ 
 &= \left\| \int_{\bbT^d} R(y)R(y)^\top [  q(y+\nabla\varphi_0(x)) - q(\nabla\varphi_0(x)) \big] \right\|
 \lesssim \int_{\bbT^d} \|R(y)\|^2 \|y\|_{\bbT^d}dy,
\end{align*}
where we used the fact that $q \in \calC^1(\bbT^d)$. Reasoning in the same way as we did below the statement of Lemma~\ref{lem:Q_to_uniform_reduction}, 
we have $\bbE[R(U)] = 0$, and thus
\begin{align*}
\Big\|\bbE[R(Y-\nabla\varphi_0(x))] \Big\|
 &= \Big\|\bbE[R(Y-\nabla\varphi_0(x))] - q(\nabla\varphi_0(x))\bbE[R(U)] \Big\| 
\lesssim \int_{\bbT^d} \| R(y)\|\|y\|_{\bbT^d}dy. 
\end{align*}
By combining the previous two displays, and using Jensen's inequality, we have 
$$\Big\| \Cov\big\{R(Y- \nabla\varphi_0(x))\big\}
 - q(\nabla\varphi_0(x)) \Cov\big\{R(U)\big\} \Big\|
\lesssim 1+\calI,$$
with $\calI := \int_{\bbT^d} \|R(y)\|^2\|y\|dy$. 
To bound $\calI$, 
recall from Lemma~\ref{lem:fundamental_solution} that $G_0$ admits a 
periodic Green's function $\Gamma_{G_0}(y,z)$, which is continuously
differentiable away from the diagonal $z=y$, and whose gradient satisfies $\|\nabla \Gamma_{G_0}(y,z)\| \lesssim \|y-z\|_{\bbT^d}^{1-d}$. 
Deduce that,  
\begin{align*}
\calI 
 &\lesssim \int_{\bbT^d} \left( \int_{\bbT^d} \|y-z\|_{\bbT^d}^{1-d} \big|\widebar K_{h_n}^o(z)\big|dz\right)^2 \|y\|_{\bbT^d}dy \\
 &\lesssim h_n^{-2d} \int_{[-1/2,1/2]^d} \left( \int_{B(0,h_n)} \|y-z\|_{\bbT^d}^{1-d} dz\right)^2 \|y\|dy \\ 
 &\lesssim h_n^{-2d} \int_{[-1/2,1/2]^d}  \int_{B(0,h_n)} \int_{B(0,h_n)} \|y-z\|_{\bbT^d}^{1-d}\|y-w\|_{\bbT^d}^{1-d} \|y\| dz dw dy \\ 
 &\lesssim 1 + \calI_1 + \calI_2,
 \end{align*}
 where, given a constant $a > 2$ satisfying $(2-2d)/a \geq (5/2)-d$, we define 
\begin{align*}
\calI_1 &=  h_n^{-2d} \int_{B(0,h_n^{1/a})}  \int_{B(0,h_n)} \int_{B(0,h_n)} \|y-z\|_{\bbT^d}^{1-d}\|y-w\|_{\bbT^d}^{1-d} \|y\| dz dw dy \\ 
\calI_2 &=  h_n^{-2d} \int_{B(0,1/4)\setminus B(0, h_n^{1/a})}  \int_{B(0,h_n)} \int_{B(0,h_n)} \|y-z\|_{\bbT^d}^{1-d}\|y-w\|_{\bbT^d}^{1-d} \|y\| dz dw dy. 
\end{align*}
To bound $\calI_1$, 
invoke Lemmas~\ref{lem:miranda}--\ref{lem:simple_ball_integral} to obtain
\begin{align*}
\calI_1
 &\lesssim h_n^{\frac 1 a - 2d} \int_{B(0,h_n)} \int_{B(0,h_n)} \|z-w\|_{\bbT^d}^{2-d}dzdw
 \lesssim h_n^{\frac 1 a + 2 - 2d} \int_{B(0,h_n)} dw \lesssim 
 h_n^{\frac 1 a + 2 - d} \lesssim h_n^{\frac 5 2 - d}.
\end{align*}  
Regarding $\calI_2$, we have the crude bound
\begin{align*}
\calI_2 
 = h_n^{-2d} \int_{B(0,1/4)\setminus B(0,h_n^{1/a})}  
 \int_{B(0,h_n)} \int_{B(0,h_n)} \|y-z\|_{\bbT^d}^{1-d}\|y-w\|_{\bbT^d}^{1-d}  dz dw dy 
 \lesssim h_n^{\frac{2-2d}{a}}.
\end{align*} 
Combining these facts together with the definition of $a$, we arrive at the claim.\qed 

\subsection{Proof of Lemma~\ref{lem:riemann_sum_approx}}
\label{app:pf_lem_riemann_sum_approx}
Abbreviate $f(\xi) = (\calF[K](\xi)/2\pi)^2$. 
Define
$$g_n = h_n^{d-2 }\sum_{\xi\in \bbZ_*^d}\xi\xi^\top 
\frac{f(h_n\xi)}{\langle \calC_0\xi,\xi\rangle^2},\quad 
 g_\infty = \int_{\bbR^d} zz^\top \frac{f(z)dz}{\langle \calC_0 z,z\rangle^2}, $$
 and for any $\epsilon > 0$, define  
\begin{align*}
g_{n,\epsilon}
 &= h_n^{d-2 }\sum_{\substack{\xi\in \bbZ_*^d \\ \epsilon \leq \|h_n\xi\|_\infty< \epsilon^{-1}}}\xi\xi^\top 
\frac{f(h_n\xi)}{\langle \calC_0\xi,\xi\rangle^2},\quad 
g_{\infty,\epsilon} = \int_{[\epsilon^{-1},\epsilon]^d} zz^\top \frac{f(z)dz}{\langle \calC_0 z,z\rangle^2}.
\end{align*} 
Notice that
\begin{align*}
g_{n,\epsilon}&-g_{\infty,\epsilon} \\
 &=\sum_{\substack{\xi\in \bbZ_*^d \\ \epsilon \leq \|h_n\xi\|_\infty< \epsilon^{-1}}}
 \int_{I_{\xi,n}} \left[(h_n\xi)(h_n\xi)^\top \frac{f(h_n\xi)}{\langle \calC_0(h_n\xi),(h_n\xi)\rangle^2}  - 
                        zz^\top \frac{ f(z)}{\langle  \calC_0 z,z\rangle^2 }\right]dz
\end{align*} 
where $I_{\xi,n}$ is the hypercube of side length $h_n$, with vertices lying in $h_n\bbZ^d$,  whose vertex nearest to the origin
is at the point $h_n\xi$.  For all $\xi$ satisfying $\epsilon \leq \|h_n\xi\|_\infty< \epsilon^{-1}$ and $z \in I_{\xi,n}$, 
one readily shows that, for some $a > 0$, 
\begin{equation}
\label{eq:riemann_sum_deviation_pf}
\|g_{n,\epsilon}-g_{\infty,\epsilon}\| \lesssim \epsilon^{-a}h_n.
\end{equation}
 Next, we bound the distance between $g_{n,\epsilon}$ and $g_n$.  
Since $f$ is a Schwartz function, 
we have $|f(\xi)| \lesssim_\ell \|\xi\|^{-\ell}$ for all $\ell \geq 0$. 
Taking $\ell=d-1$, we obtain  
\begin{align*}
\Bigg\|h_n^{d-2}\sum_{\substack{\xi\in\bbZ_*^d \\  \|\xi h_n\|_\infty \geq 1/\epsilon}} \xi\xi^\top 
 \frac{f(h_n\xi)}{\langle \calC_0 \xi,\xi\rangle^2}  \Bigg\| 
 &\lesssim h_n^{d-2} \sum_{\substack{\xi\in\bbZ_*^d \\  \|\xi h_n\|_\infty \geq 1/\epsilon}}  \frac{\|h_n\xi\|^{1-d}}{\|\xi\|^2} 
 \asymp \epsilon,
\end{align*}
for a dimension-dependent constant $c_d > 0$.  Similarly,
\begin{align*}
\Bigg\|h_n^{d-2}\sum_{\substack{\xi\in\bbZ_*^d \\  \|\xi h_n\|_\infty < \epsilon}} \xi\xi^\top 
 \frac{f(h_n\xi)}{\langle \calC_0 \xi,\xi\rangle^2}  \Bigg\| 
 \asymp \epsilon^{d-2}.
\end{align*}
Combining the preceding two displays, we thus have
\begin{equation}
\label{eq:deviation_gne_gn}
\|g_{n,\epsilon} - g_{n}\| \lesssim \epsilon.
\end{equation}
Analogous derivations show that
\begin{equation}
\label{eq:deviation_ginftye_ginfty}
\|g_{\infty,\epsilon} - g_{\infty}\| \lesssim \epsilon.
\end{equation}
Combining equations~\eqref{eq:riemann_sum_deviation_pf},~\eqref{eq:deviation_gne_gn}, and~\eqref{eq:deviation_ginftye_ginfty}, we deduce
$$\|g_{n}-g_\infty\| \lesssim \epsilon + \epsilon^{-a} h_n,$$
for a large enough constant $a > 0$. The claim now follows by taking $\epsilon \asymp h_n^{\frac 1 {1+a}}$. 
 \qed

\subsection{Proof of Lemma~\ref{lem:var_bound_V2_tech}}
\label{sec:pf_lem_var_bound_V2_tech}
Recall that $\supp(K_{h_n}) \subseteq B(0, h_n)$.
Therefore, a point $y \in \bbT^d$ can only lie in the support of $K_{h_n}(Y-\nabla\varphi_0(\cdot))$ if
$$Y-\nabla\varphi_0(y) \in B(0, h_n),~~ \text{or equivalently, }~  y \in \nabla\varphi_0^*(B(Y, h_n)).$$
Since $\nabla\varphi_0^*$ is $\lambda$-Lipschitz, we have
$$\nabla\varphi_0^*(B(Y, h_n)) \subseteq B(\nabla\varphi_0^*(Y), \lambda h_n),$$
and we deduce that
$$\supp \big(K_{h_n}(Y-\nabla\varphi_0(\cdot))\big)  \subseteq B(\nabla\varphi_0^*(Y), \lambda h_n).$$ 
Notice that $S_x^{-1}$ is also Lipschitz with parameter $\lambda$, thus the same argument implies 
$$\supp \big(K_{h_n}(Y-S_x(\cdot))\big)  \subseteq B(S_x^{-1}(Y), \lambda h_n).$$
The claim follows. 
 \qed 

\subsection{Proof of Lemma~\ref{lem:ball_integral}}
\label{app:pf_lem_ball_integral}
Let $I = \bbE\left[\left(\int_{B(Z, \epsilon)} \|x-z\|_{\bbT^d}^{t-d}dz\right)^2\right]$. It holds that, 
\begin{align*}
I &\leq \gamma \bbE_{Z\sim \calL}\left[\left(\int_{B(Z, \epsilon)} \|x-z\|_{\bbT^d}^{t-d}dz\right)^2\right] \\ 
 &= \int_{\bbT^d} \int_{B(y,\epsilon)} \int_{B(y,\epsilon)}  \|x-z\|_{\bbT^d}^{t-d}\|x-w\|_{\bbT^d}^{t-d} dzdw dy  \\ 
 &= \int_{\bbT^d} \int_{B(0,\epsilon)} \int_{B(0,\epsilon)}  \|x-z-y\|_{\bbT^d}^{t-d}\|x-w-y\|_{\bbT^d}^{t-d} dzdw dy \\ 
 &= \int_{B(0,\epsilon)} \int_{B(0,\epsilon)} \left(\int_{\bbT^d} \|x-z-y\|_{\bbT^d}^{t-d}\|x-w-y\|_{\bbT^d}^{t-d}dy\right) dzdw  \\ 
 &= \int_{B(0,\epsilon)} \int_{B(0,\epsilon)} \left(\int_{\bbT^d} \|z-y\|_{\bbT^d}^{t-d}\|w-y\|_{\bbT^d}^{t-d}dy\right) dzdw  \\ 
 &\lesssim \int_{B(0,\epsilon)} \int_{B(0,\epsilon)} \|z-w\|_{\bbT^d}^{2t-d}  dzdw \\ 
 &\lesssim \int_{B(0,\epsilon)} \int_{B(0,\epsilon)} \|z-w\|^{2t-d}dzdw,
\end{align*}
where the penultimate line follows from Lemma~\ref{lem:miranda}, together with the condition
$t < d/2$.
Passing to spherical coordinates, we obtain
\begin{align*}
I
 &\leq  \int_{B(0,\epsilon)} \int_{B(0,\epsilon)} \big| \|z\|-\|w\|\big|^{2t-d}dzdw \\
 &\leq  \int_0^{\epsilon} \int_0^{\epsilon} | r-s|^{2t-d}r^{d-1}s^{d-1}drds \\  
 &=   \int_0^{\epsilon} \int_0^s (s-r)^{2t-d}r^{d-1}s^{d-1}drds 
  + \int_0^{\epsilon} \int_s^\epsilon (r-s)^{2t-d}r^{d-1}s^{d-1}drds \\  
 &=: (A) + (B).
 \end{align*}
 We  bound terms $(A)$ and $(B)$ in turn. On the one hand, 
 \begin{align*}
(A)
\leq  \int_0^{\epsilon} \int_0^s s^{2t-d}r^{d-1}s^{d-1}drds 
=  \int_0^{\epsilon} \int_0^s s^{2t-1}r^{d-1}drds 
\asymp \int_0^\epsilon s^{d+2t-1}ds \asymp \epsilon^{d+2t},
 \end{align*}
while on the other hand, 
\begin{align*}
(B)
 \leq \int_0^\epsilon \int_s^\epsilon r^{2t-1} s^{d-1}drds 
 \leq  \int_0^\epsilon \epsilon^{2t} s^{d-1}ds \asymp \epsilon^{d+2t}.
\end{align*}
This proves the claim.\qed

\subsection{Proof of Lemma~\ref{lem:regularity_un}}
\label{app:pf_lem_regularity_un}
By Lemmas~\ref{lem:isomorphism},~\ref{lem:holder_products},  and~\ref{lem:holder_comp}, we have
\begin{align*}
\|u_n\|_{\calC^{2+\beta}(\bbT^d)} 
\lesssim \|\det(\calB)\widebar K_{h_n}^o(Y-\nabla\varphi_0(\cdot))\|_{\calC^\beta(\bbT^d)}
\lesssim \|\widebar K_{h_n}^o\|_{\calC^\beta(\bbT^d)} 
\lesssim h_n^{-(d+\beta)}. 
\end{align*} 
Now, recall that $\lambda > 1$ is such that $\|\varphi_0\|_{\calC^3(\bbT^d)} \leq \lambda$
and $\nabla^2\varphi_0^* \succeq \lambda^{-1} I_d$ uniformly over $\bbT^d$. 
To prove the claim, it will be sufficient to show that for all
$y\in \bbT^d$ such that $\|y-X_0\|_{\bbT^d} \geq 2 \lambda h_n$, 
we have $\|\nabla^2 u_n(y)\| \lesssim  \|X_0-y\|_{\bbT^d}^{-d}$. 
To prove this, note that it suffices to consider all $y \in \bbT^d$ such that $\|X_0-y\| = \|X_0-y\|_{\bbT^d}$, 
and we shall assume this equality holds throughout the remainder of the proof. 

Assume thus that $\|y-X_0\| \geq 2 \lambda h_n$, 
and set 
$$B = B(y, \|y-X_0\|/2),\quad B_0 =  B(y, h_n/4).$$ 
Notice that for all $v \in B$, 
\begin{equation}
\label{eq:linfty_un_supp_sep}
\|v - X_0\| \geq \|y-X_0\| - \|v-y\| \geq \frac {\|y-X_0\|} {4}.
\end{equation}
Furthermore, we have $\mathrm{dist}(B,B_0) \geq \|y-X_0\|/4$, 
thus we may apply the a priori bound in Lemma~\ref{lem:a_priori_schauder} to obtain
\begin{align}
\label{eq:linfty_un_step}
\|u_n&\|_{\calC^{2}(B_0)} \lesssim \|y-X_0\|^{-2} 
\Big[\|\det(\calB) \widebar K_{h_n}^o(Y - \nabla\varphi_0(\cdot))\|_{\calC^\beta(B)} + 
\|u_n\|_{L^\infty(B)}\Big],
\end{align}
where we again used Lemma~\ref{lem:holder_products}.
Recall from Lemma~\ref{lem:var_bound_V2_tech} 
that $\supp(\widebar K_{h_n}(Y - \nabla\varphi_0(\cdot))) \subseteq B(X_0, \lambda h_n)$.
The latter is disjoint from $B$ by equation~\eqref{eq:linfty_un_supp_sep}, thus we have
$$\widebar K_{h_n}^o(Y - \nabla\varphi_0(v)) = -1\quad \text{for all } v \in B,$$
and hence, from equation~\eqref{eq:linfty_un_step} we obtain
\begin{align}
\label{eq:linfty_un_step2}
\nonumber 
\|u_n\|_{\calC^{2}(B_0)} &\lesssim 
\|y-X_0\|^{-2} \Big[\|\det(\calB)\|_{\calC^\beta(B)} + 
\|u_n\|_{L^\infty(B)}\Big]  \\
&\lesssim  \|y-X_0\|^{-2} \Big[1 + 
\|u_n\|_{L^\infty(B)}\Big].
\end{align}
It thus remains to bound $\|u_n\|_{L^\infty(B)}$. To this end, recall that we can write for any $v \in B$, 
$$u_n(v) = \int_{\bbT^d} \Gamma_{E}(z,v) \widebar K_{h_n}^o(Y  - \nabla\varphi_0(z))dz,$$
where $\Gamma_E$ is the periodic Green's function of $E$ (cf. Lemma~\ref{lem:fundamental_solution}). 
We thus have, for all $v \in B$,
\begin{align}
\label{eq:linfty_un_step3}
\nonumber
|u_n(v)| 
 &\lesssim \int_{\bbT^d} \|z-v\|_{\bbT^d}^{2-d} \big|\widebar K_{h_n}^o(Y  - \nabla\varphi_0(z))\big|dz  \\ 
\nonumber
 \nonumber
 &\lesssim 1 + h_n^{-d}\int_{B(0, \lambda h_n)} \|z-X_0-v\|_{\bbT^d}^{2-d}dz \\ 
 &\lesssim 1 +  \big[\|X_0-v\|_{\bbT^d} - \lambda h_n\big]^{2-d} 
 \lesssim 1 +  \|X_0-v\|_{\bbT^d}^{2-d},
\end{align}
where the final inequality follows from equation~\eqref{eq:linfty_un_supp_sep}. 
From equations~(\ref{eq:linfty_un_step2}--\ref{eq:linfty_un_step3}),
we obtain
\begin{align*} 
\|u_n\|_{\calC^{2}(B_0)} 
 &\lesssim \sup_{v\in B} \|y-X_0\|^{-2} \Big[1 + \|X_0-v\|_{\bbT^d}^{2-d}\Big]\\
 &\leq \|y-X_0\|^{-2} \big(\|X_0-y\|_{\bbT^d} - h_n/2\big)^{2-d}
 \lesssim \|y-X_0\|^{-d} .
\end{align*}
The claim follows.
\qed

\subsection{Proof of Lemma~\ref{lem:bounding_term_I2}}
\label{app:pf_lem_bounding_term_I2}
We will make use of the following gradient estimate for $u_n$.
\begin{lemma}
\label{lem:regularity_un_grad}
There exists a constant $C > 0$ such that for all $y \in \bbT^d$, 
$$\|\nabla u_n(y)\| \leq C \big( h_n \vee \|X_0-y\|_{\bbT^d}\big)^{1-d}.$$
\end{lemma}
The proof of Lemma~\ref{lem:regularity_un_grad} appears in Appendix~\ref{app:pf_lem_regularity_un_grad}. 
With this result in place, we can bound $\bbE[\calI_2^2]$ in nearly the same way as $\bbE[\calI_1^2]$. 
To see this, notice first that from Lemma~\ref{lem:regularity_un_grad}, we have
\begin{align*}
\begin{multlined}
\bbE\left[\left(\int_{B(X_0,h_n)} \|x-y\|_{\bbT^d}^{1-d} \|\nabla u_n(y)\|dy\right)^2\right]  \\ 
 \lesssim h_n^{2(1-d)}\bbE\left[\left(\int_{B(X_0,h_n)} \|x-y\|_{\bbT^d}^{1-d} dy\right)^2\right] 
\lesssim h_n^{4-d},
\end{multlined}
\end{align*}
by Lemma~\ref{lem:ball_integral}. Thus, as before, 
\begin{align}
\nonumber 
\bbE[\calI_2^2]
 &\lesssim  h_n^{4-d} +  \bbE\left[\left(\int_{\calQ_0 \setminus B(0,h_n)} \|x-X_0-y\|_{\bbT^d}^{1-d} \|y\|_{\bbT^d}^{1-d} dy\right)^2\right].
\end{align}
Let the event $A_n$ be defined as in Step 4 of the proof of Lemma~\ref{lem:variance_bound_map}, and recall equation~\eqref{eq:pf_var_bound_repr}, from which it follows that
\begin{align*}
\bbE&\left[\left(\int_{\calQ_0\setminus B(0,h_n)} \|x-X_0-y\|_{\bbT^d}^{1-d} \|y\|_{\bbT^d}^{1-d} dy\right)^2 I( A_n)\right] 
\\ &\qquad 
\lesssim \bbE\left[\left(\int_{\calQ_0\setminus B(0,h_n)} \|y\|^{2-2d} dy\right)^2 I( A_n)\right]  
\lesssim h_n^{4-d}.
\end{align*}
Furthermore, by Lemma~\ref{lem:miranda}, 
\begin{align*}
\bbE&\left[\left(\int_{\calQ_0\setminus B(0,h_n)} \|x-X_0-y\|_{\bbT^d}^{1-d} \|y\|^{1-d} dy\right)^2 I( A_n^c)\right] \\ 
 & \lesssim \bbE\left[\left(\|x-X_0\|^{2-d}\right)^2 I( A_n^c)\right]  \asymp h_n^{4-d}\vee 1.
  \end{align*}The claim follows. 
\qed 

\subsection{Proof of Lemma~\ref{lem:regularity_un_grad}}
\label{app:pf_lem_regularity_un_grad}
From Lemma~\ref{lem:fundamental_solution} and the definition of $\widebar K_{h_n}^o$, we have  
for all $y \in \bbT^d$,  
\begin{align*}
\|\nabla u_n(y)\|
\lesssim 1 + \int_{\bbT^d} \|y-z\|_{\bbT^d}^{1-d} \big|\widebar K_{h_n}(Y - \nabla\varphi_0(z))\big| dz.
\end{align*}
By~Lemma~\ref{lem:var_bound_V2_tech}, we have $\supp(\widebar K_{h_n}(Y - \nabla\varphi_0(\cdot))) \subseteq B(X_0, \lambda h_n)$, thus 
\begin{align*}
\|\nabla u_n(y)\|
\lesssim 1 + h_n^{-d}\int_{B(X_0, \lambda h_n)} \|y-z\|_{\bbT^d}^{1-d}   dz \lesssim h_n^{1-d},
\end{align*}
by Lemma~\ref{lem:simple_ball_integral}.\qed 

\section{Additional Proofs from Section~\ref{sec:pf_main_results}}
\label{app:additional_proofs_5}

\subsection{Proof of Lemma~\ref{lem:expectation_zero}}
\label{app:pf_lem_expectation_zero}
 Let $x \in \bbT^d$. 
 Let us first show that $\bbE[\nabla Z_{n,i1}(x)] = 0$. Using Lemma~\ref{lem:inverse_L_to_E} and Lemma~\ref{lem:fundamental_solution}, 
 it can be deduced that  
 $$\nabla L^{-1}\big[\widebar K_{h_n}(Y_1-\cdot) - q_{h_n}](x) = 
 \int_{\bbT^d} \nabla \Gamma_{E}(\nabla\varphi_0^*(y),x)  \big[\widebar K_{h_n}(Y_1-y) - q_{h_n}(y)]dy,$$
 where $\Gamma_E$ is the periodic Green's function for the operator $E$, and so
 $$\bbE[\nabla Z_{n,1}(x)] 
  = \int_{\bbT^d}\int_{\bbT^d} \nabla \Gamma_{E}(\nabla\varphi_0^*(y),x)  \big[\widebar K_{h_n}(z-y) - q_{h_n}(y)]dydQ(z).$$
One may apply Fubini's theorem to change the order of integration in the above
display, since for any given $h_n > 0$, by Tonelli's theorem,
\begin{align*}
\int_{\bbT^d}\int_{\bbT^d} \big\| \nabla \Gamma_{E}(\nabla\varphi_0^*(y),x) \big[\widebar K_{h_n}(z-y) - q_{h_n}(y)\big] \big\|dydQ(z)   \\ 
 \lesssim h_n^{-d} \int_{\bbT^d} \| \nabla\varphi_0^*(y)-x \|^{1-d}dy
 < \infty.
 \end{align*}
Deduce that 
 $$\bbE[\nabla Z_{n,i1}(x)] 
  = \int_{\bbT^d} \nabla \Gamma_{E}(\nabla\varphi_0^*(y),x) \left(\int_{\bbT^d} \big[\widebar K_{h_n}(z-y) - q_{h_n}(y)]dQ(z)\right)dy = 0.$$
From here, we deduce
\begin{align*}
&\left\|\bbE [ \left((1/n) \textstyle\sum_{i=1}^n \nabla Z_{n,i}(x)\right)I_{A_n}] \right\|\\
 &~~\leq \left\|\bbE [ \left((1/n) \textstyle\sum_{i=1}^n \nabla Z_{n,i}(x)\right)] \right\| + \left\|\bbE [ \left((1/n) \textstyle\sum_{i=1}^n \nabla Z_{n,i}(x)\right)I_{A_n^\cp}] \right\| \\ 
 &~~=  \left\|\bbE [ \left((1/n) \textstyle\sum_{i=1}^n \nabla Z_{n,i}(x)\right)I_{A_n^\cp}] \right\|.
\end{align*} 
Now, recall from Lemma~\ref{lem:inverse_L_to_E} that 
$$\|Z_{n,1}\|_{\calC^{2+\beta}(\bbT^d)} \asymp 
\|\widebar K_{h_n}(Y_1-\cdot) - q_{h_n}]\|_{\calC^\beta(\bbT^d)} \lesssim h_n^{-(d+\beta)},$$
and we deduce that 
\begin{align*}
\left\|\bbE [ \left((1/n) \textstyle\sum_{i=1}^n \nabla Z_{n,i}(x)\right)I_{A_n}] \right\|
 \lesssim h_n^{-(d+\beta)} \bbP(A_n^\cp)\lesssim n^{a(d+\beta) -b}.
\end{align*} 
It is clear that the implicit constants in these various assertions 
do not depend on $x$, and the claim follows.
\qed

  \subsection{Proof of Lemma~\ref{lem:lyapunov_technical_map}}
\label{app:pf_lem_lyapunov_technical_map}
By Lemma~\ref{lem:inverse_L_to_E}, we have
$$\|\nabla Z_{n,1}(x)\| \lesssim \|L^{-1}[\hat q_{n}-q_{h_n}]\|_{\calC^{2+\beta}(\bbT^d)} \lesssim \|\hat q_n - q_{h_n}\|_{\calC^\beta(\bbT^d)}
\lesssim h_n^{-(d+\beta)}.$$
Thus, using Lemma~\ref{lem:variance_bound_map}, we have
$$\bbE \|\nabla Z_{n,1}(x)\|^{2+\delta} \lesssim h_n^{-\delta(d+\beta)} \bbE \|\nabla Z_{n,1}(x)\|^2 \lesssim h_n^{-\delta(d+\beta)+2-d},$$
as claimed.
\qed

\subsection{Proof of Proposition~\ref{prop:clt_projections}}
\label{app:pf_bandwidth_projection_clt}
Let $u_n = L^{-1}[\hat q_n - q]$. We will begin by bounding the quantity
$\langle u_n,v\rangle_A$, for any $v \in H_0^1(\bbT^d)$. 
By Lemma~\ref{lem:inverse_L_to_E},  the function $u_n$
solves the PDE 
$$E u_n = \det(\nabla^2\varphi_0) (\hat q_n(\nabla\varphi_0) - q(\nabla\varphi_0)),\quad \text{over } \bbT^d,$$
in the classical sense, and hence also in the weak $H_0^1(\bbT^d)$ sense
(cf. Appendix~\ref{app:pde}). Therefore, 
for all $v \in H_0^1(\bbT^d)$, 
\begin{align*}
\langle u_n,v\rangle_A 
 &= \langle \det(\nabla^2\varphi_0) (\hat q_n(\nabla\varphi_0) - q(\nabla\varphi_0)),v\rangle_{L^2(\bbT^d)} \\
 &= \langle \hat q_n-q, v(\nabla\varphi_0^*)\rangle_{L^2(\bbT^d)} \\
 &= \int_{\bbT^d} \big(v(\nabla\varphi_0^*) \star K_{h_n}\big)d(Q_n-Q) + \int_{\bbT^d}v(\nabla\varphi_0^*) d(Q_{h_n} - Q).
\end{align*}
Reasoning as in Lemma~48 of~\cite{manole2021}, it is easy to see that
$$\int_{\bbT^d} \big(v(\nabla\varphi_0^*) \star K_{h_n}\big)d(Q_n-Q) = O_p(n^{-1/2}),$$
and furthermore, 
$$\left|\int_{\bbT^d}v(\nabla\varphi_0^*) d(Q_{h_n} - Q)\right|
\leq \|v(\nabla\varphi_0^*)\|_{H^1(\bbT^d)} \|q_{h_n} - q\|_{H^{-1}(\bbT^d)} \lesssim \|q_{h_n}-q\|_{H^1(\bbT^d)},$$
where we used Lemma~\ref{lem:sobolev_comp}. By Proposition~\ref{prop:kde_negative_sobolev}, 
the final factor decays on the order $h_n^{s+1}$, thus we have shown 
$$\langle u_n, v\rangle_{L^2(\bbT^d)} = O_p(n^{-1/2}+h_n^{s+1}).$$
To prove the  claim, it thus remains to quantify the discrepancy between $\langle u_n,v\rangle_A$ and $\langle \hat\varphi_n- \varphi_0,v\rangle_A$. 
To this end, recall from equation~\eqref{eq:linearization_implication} that, 
for any given $\beta > 0$ sufficiently small,
$$\left\|\nabla\hat\varphi_n - \nabla \varphi_0- \nabla L^{-1}[\hat q_n - q] \right\|_{L^\infty(\bbT^d)} = O_p\left( h_n^{2s - 1-3\beta} + 
\frac 1 {nh_n^{d+1+4\beta}} \right).$$
It readily follows that 
\begin{align*}
\langle \hat\varphi_n - \varphi_0,v\rangle_A 
 &=  \langle L^{-1}[\hat q_n-q],v\rangle_A + O_p\left( h_n^{2s - 1-3\beta} + 
\frac 1 {nh_n^{d+1+4\beta}} \right) \\ 
 &=   O_p\left( n^{-1/2} + h_n^{s+1}+ \frac 1 {nh_n^{d+1+4\beta}} \right). 
\end{align*}
The claim follows. 
\qed

\subsection{Proof of Proposition~\ref{prop:L2_rate_map}}
\label{app:pf_L2_rate_map}

On the one hand, we have by Lemma~\ref{lem:L2_stab} that
$$\|\widebar T_{h_n} - T_0\|_{L^2(\bbT^d)} \asymp \|q_{h_n} - q\|_{H^{-1}(\bbT^d)}
\lesssim h_n^{s+1},$$
where the final inequality follows from Proposition~\ref{prop:kde_negative_sobolev}. 
On the other hand, under our assumptions, it follows from Proposition~\ref{prop:kde_holder} that 
$$\|q_{h_n}-q\|_{\calC^\epsilon(\bbT^d)} =o(1)$$
for some small enough $\epsilon > 0$, thus since $q$ is bounded away from zero by a positive constant
and itself lies in $\calC^\epsilon_+(\bbT^d)$, we deduce that
$\|q_{h_n}\|_{\calC^\epsilon(\bbT^d)} \lesssim 1.$
Thus, we may again apply Lemma~\ref{lem:L2_stab} and Proposition~\ref{prop:kde_negative_sobolev} to deduce that
$$\bbE \|\hat T_{h_n} - \widebar T_{h_n}\|_{L^2(\bbT^d)}^2 \asymp \bbE \|\hat q_n - q_{h_n}\|_{H^{-1}(\bbT^d)}^2
\asymp \frac{h_n^{2-d}}{n},$$
and, for any $r > 2$, 
$$\bbE \|\hat T_{h_n} - \widebar T_{h_n}\|_{L^2(\bbT^d)}^r \asymp \bbE \|\hat q_n - q_{h_n}\|_{H^{-1}(\bbT^d)}^r
\lesssim n^{-\frac r 2} h_n^{r\left(1-\frac d 2\right)},$$
as claimed.\qed

\section{Additional Proofs from Section~\ref{sec:discussion}}
\label{app:additional_discussion}

We state and prove the following elementary Lemma, which was used in the proof of Proposition~\ref{prop:clt_map_2d}.
Recall the sequence of matrices $(A_n)_{n\geq 1}$ defined therein.
\begin{lemma}
\label{app:pf_lem_riemann_sum_approx_d2}
For any fixed $c > 0$, 
$$ A_n = \int_{\xi: c \leq  \|\xi\| < 1/h_n} 
 \frac{\xi\xi^\top d\xi }{4\pi^2 \|\xi\|^4}+O(1).$$ 
\end{lemma}

\subsection{Proof of Lemma~\ref{app:pf_lem_riemann_sum_approx_d2}}
We reason similarly as in the proof of Lemma~\ref{app:pf_lem_riemann_sum_approx}.
Write 
$$g_n = \sum_{\substack{\xi\in \bbZ_*^d \\ \|\xi\| < 1/h_n}} 
\frac{\xi\xi^\top}{ \|\xi\|^4},\quad 
 g_\infty = \int_{z: c\leq \|z\| < 1/h_n}  \frac{zz^\top dz}{\|z\|^4},$$
It suffices to show that $g_n=g_\infty+ O(1)$. Notice that
\begin{align*}
g_{n}&-g_{\infty} =\sum_{\substack{\xi\in \bbZ_*^d \\ c\leq  \|\xi\| < 1/h_n}}
 \int_{I_{\xi}} \left[ \frac{\xi\xi^\top}{\|\xi\|^4} - \frac{zz^\top}{\|z\|^4} \right]dz + O(1), 
\end{align*} 
where 
we denote by $I_{\xi}$ the hypercube of unit length,  whose vertex nearest to the origin
is at the point $\xi$. Notice that for any $\xi$ such that $c\leq \|\xi\| < 1/h_n$, 
and for any $z \in I_\xi$, it holds that 
\begin{align*}
 \bigg\| \frac{\xi\xi^\top}{\|\xi\|^4}   
     - \frac{zz^\top}{\|z\|^4} \bigg\| 
 &\lesssim \left\|\frac{\xi\xi^\top}{\|\xi\|^4\|z\|^4}
   \big[ \|z\|^4 - \|\xi\|^4\big]\right\| + 
   \left\|\frac{1}{\|z\|^4} 
   \big[ \xi\xi^\top -  zz^\top \big]\right\| 
    \\ 
 &\lesssim \frac{\|\xi\|^2\big(\|\xi\|+\|z\|\big)^3}{\|\xi\|^4\|z\|^4}\|\xi-z\|
 + 
 \frac{\|\xi\| + \|z\|}{\|z\|^4} \|\xi-z\|\lesssim \frac 1 {\|\xi\|^3},
\end{align*}
where the final display follows from the fact that $z \in I_\xi$. We deduce that
$$\|g_{n}-g_{\infty}\| \lesssim 1 + \sum_{\substack{\xi\in \bbZ_*^d \\ c < \|\xi\| < 1/h_n}} \frac 1 {\|\xi\|^3} \lesssim 1,
$$
as claimed.
\qed

\section{Additional Technical Lemmas}
\label{app:additional_technical_lemmas}
The following stability bound is a consequence of Theorem 6 and Proposition 16 
of~\cite{manole2021}, together with the bound of~\citet[Theorem 1]{peyre2018}, and 
Caffarelli's regularity theory (Theorem~\ref{thm:caffarelli}).
See also~\citet[Lemma 5.1]{ghosal2022}. 
\begin{lemma}\label{lem:L2_stab}
Suppose that $p,q,\hat q \in \calC_+^\epsilon(\bbT^d)$ for some $\epsilon > 0$. Then,  
 if $\hat T$ denotes the unique optimal transport
map pushing forward $p$ onto $\hat q$, it holds that
$$C^{-1} \|\hat q-q\|_{H^{-1}(\bbT^d)} \leq \|\hat T - T_0\|_{L^2(\bbT^d)} \leq C \|\hat q-q\|_{H^{-1}(\bbT^d)},$$
for a constant $C = C(\omega_\epsilon(p,q,\hat q),d,\epsilon)$. 
\end{lemma}

%

The following bound can be deduced from the Leibniz rule together with equation (4.7)  of~\cite{gilbarg2001}. 

\begin{lemma}[Products of H\"older Functions]
\label{lem:holder_products}
For any $\alpha \geq 0$, there exists a constant $C=C(\alpha)>0$ such that for all $f,g \in \calC^\alpha(\bbT^d)$, 
$$ \|fg\|_{\calC^\alpha(\bbT^d)} \leq C_\alpha \|f\|_{\calC^\alpha(\bbT^d)} \|g\|_{\calC^\alpha(\bbT^d)}.$$
\end{lemma}

The following is straightforward. 
\begin{lemma}[Compositions of H\"older Functions]
\label{lem:holder_comp}
For any $\alpha,\beta \in (0,1]$, 
there exists a constant $C = C(\alpha,\beta) > 0$ such that for all $f \in \calC^{1+\alpha}(\bbT^d)$
and $g \in \calC^\beta(\bbT^d)$, 
$$ \|g\circ \nabla f\|_{\calC^{\alpha\beta}(\bbT^d)} \leq C \|f\|_{\calC^{1+\alpha}(\bbT^d;\bbT^d)}^\beta \|g\|_{\calC^\beta(\bbT^d)}.$$
\end{lemma}
We next provide  similar bounds for Sobolev norms.
\begin{lemma}[Compositions of Sobolev Functions]
\label{lem:sobolev_comp}
Let $r > 1$. Let $f \in \calC^2(\bbT^d)$ be such that 
$$\nabla^2 f \succeq I_d/\lambda\text{ over } \bbT^d, \quad \text{and}\quad \|f\|_{\calC^2(\bbT^d)} \leq \lambda,$$
for some $\lambda > 0$. Then, 
there exists a constant $C=C(\lambda,d,r) > 0$ such that
for all $\alpha \in [0,1]$ and all $g \in H^{\alpha,r}(\bbT^d)$,
$$\|g\circ \nabla f\|_{H^{\alpha,r}(\bbT^d)} \lesssim C  \|g\|_{H^{\alpha,r}(\bbT^d)}.$$
\end{lemma}
\begin{proof}[Proof of Lemma~\ref{lem:sobolev_comp}]
When $\alpha=0$, we have for all $g \in L^r(\bbT^d)$, 
\begin{align*}
\|g\circ \nabla f\|_{L^r(\bbT^d)}^r 
 = \int |g(\nabla f)|^r    
 = \int |g|^r \det(\nabla^2 f^*) 
 \leq \lambda  \|g\|_{L^r(\bbT^d)}^r. 
 \end{align*} 
To prove the claim when $\alpha=1$, let $g \in H^{1,r}_0(\bbT^d)$, and notice
that the identity $\nabla(g\circ \nabla f) = \nabla^2 f \nabla g(\nabla f)$ 
holds in the sense of weak derivatives. Together with Lemma~\ref{lem:equiv_H1_norm}, we deduce
\begin{align*}
\|g\circ \nabla f\|_{H^{1,r}(\bbT^d)} ^r
 &\asymp \|g\circ \nabla f\|_{L^r(\bbT^d)}^r + \int \|\nabla^2 f\cdot  (\nabla g(\nabla f))\|^r  \\
    &\lesssim_\lambda \|g\|_{L^r(\bbT^d)}^r + \int \| \nabla^2 f(\nabla f^*)\cdot \nabla g   \|^r \det(\nabla^2 f^*)
 \lesssim_\lambda  \|g\|_{H^{1,r}(\bbT^d)}^r.
 \end{align*}
The bound for remaining values of $\alpha$ follows by interpolation. 
Indeed, notice that the operator
$$F: H^{\alpha,r}(\bbT^d) \to H^{\alpha,r}(\bbT^d), \quad Fg = g(\nabla f)$$
is well-defined and linear. Let $\|F\|_{\alpha,r} = \sup\{ \|Fg\|_{H^{\alpha,r}(\bbT^d)}: g \in H^{\alpha,r}(\bbT^d)\}$
denote its operator norm. Since $H^\alpha(\bbT^d)$ is the $\alpha$-complex
interpolation space of $L^2(\bbT^d)$ and $H^1(\bbT^d)$~(cf. Theorem~3.6.1/2 of~\cite{schmeisser1987}), 
we deduce~\citep[Theorem 2.6]{lunardi2018}
$$\|F\|_\alpha \leq \|F\|_0^{1-\alpha} \|F\|_{1}^\alpha \lesssim_\lambda 1.$$
The claim follows.
\end{proof}



The following bound can be deduced from~\citet[Theorem 11,I]{miranda2013}.  
\begin{lemma}[Composition of Riesz Potentials]\label{lem:miranda}
Let $d \geq 1$. For all $\alpha_1,\alpha_2 \geq 0$, 
there exists a constant $C > 0$ such that for any $w_1, w_2 \in \bbT^d$,
$$\int_{\bbT^d}  \|y-w_1\|_{\bbT^d}^{\alpha_1-d}\|y-w_2\|_{\bbT^d}^{\alpha_2-d}dy \leq C 
\begin{cases}
\|w_1-w_2\|_{\bbT^d}^{\alpha_1+\alpha_2-d}, & \alpha_1+\alpha_2 < d, \\ 
\log\big(1 / \|w_1-w_2\|\big), & \alpha_1+\alpha_2 = d, \\ 
1, & \alpha_1+\alpha_2 > d.
\end{cases}$$
\end{lemma}

We also state the following related estimate.
\begin{lemma}
\label{lem:simple_ball_integral}
Let $t \in (1,d)$. Then, for any $\epsilon \in (0,1/4)$ and $x \in B(0,1/4)$,
$$\int_{B(0, \epsilon)} \|y-x\|_{\bbT^d}^{t-d}dy \leq  \epsilon^t \wedge \|x\|^{t-d}\epsilon^d.$$
\end{lemma}
{\bf Proof of Lemma~\ref{lem:simple_ball_integral}.} 
With the convention that $\int_\emptyset (\cdot) = 0$, we have, 
\begin{align*}
\int_{B(0, \epsilon)} \|y-x\|_{\bbT^d}^{t-d}dy
 &\leq \int_{B(x, \epsilon/2)} \|y-x\|_{\bbT^d}^{t-d}dy + 
       \int_{B(0, \epsilon) \setminus B(x, \epsilon/2)} \|y-x\|_{\bbT^d}^{t-d}dy \\ 
 &\lesssim \int_0^{\epsilon/2} r^{t-d}r^{d-1}dr + 
       \int_{B(0, \epsilon)} \epsilon^{t-d}dy \lesssim \epsilon^t.
\end{align*} 
Furthermore, if $\|x\| \geq 2\epsilon$ we have
\begin{align*}
\int_{B(0, \epsilon)} \|y-x\|_{\bbT^d}^{t-d} dy 
 &\leq \int_{B(0,\epsilon)} \big| \|x\|-\|y\|\big|^{t-d}dy \\ 
 &\leq \int_0^\epsilon \big| \|x\|-r\big|^{t-d}r^{d-1}dr
  \lesssim \|x\|^{t-d} \int_0^\epsilon r^{d-1}dr \lesssim \|x\|^{t-d} \epsilon^d.
\end{align*}
This proves the  claim. \qed 

\vspace{0.1in}

Finally, we state a multivariate version of Lyapunov's central limit theorem. 
\begin{lemma} 
\label{lem:lyapunov}
Let $\{U_{n,i}:1 \leq i \leq n < \infty\}$ be a triangular array of independent
random variables in $\bbR^d$ with mean zero and finite second moment. Define
$V_n = \sum_{i=1}^n \Cov[U_{n,i}]$, 
and let $v_n^2 = \lmin(V_n)$, 
for all $n=1, 2, \dots$. 
If for some $\delta > 0$, Lyapunov's condition
$$\frac{1}{v_n^{2+\delta}} \sum_{i=1}^n \bbE\Big[ |U_{n,i} |^{2+\delta}\Big] = o(1)$$
holds, then
$ V_n^{-1/2} \sum_{i=1}^n U_{n,i} \overset{d}{\longrightarrow} N(0,I_d).$ 
\end{lemma}

\section*{Acknowledgements}
TM would like to thank Alberto Gonz\'alez-Sanz for his detailed comments on an earlier
version of this draft. TM also thanks Axel Munk and Housen Li for discussions related
to this work. TM was supported in part by the Natural Sciences and Engineering Research Council of Canada,
through a PGS D scholarship. TM, SB, and LW gratefully acknowledge the support 
of   National Science Foundation grant   DMS-2310632.
JNW gratefully acknowledges the support of National Science Foundation grant DMS-2210583 and a fellowship from the Alfred P.~Sloan Foundation.

\bibliographystyle{apalike}
\bibliography{manuscript_arxiv}

\end{document}